\documentclass{amsart}

\usepackage{amssymb,stmaryrd,blkarray,varwidth}
\usepackage{amsthm}
\usepackage{enumitem}
\usepackage{tikz-cd}
\usepackage[a4paper,left=2.25cm,right=2.25cm,top=3cm,bottom=3cm,headsep=1cm]{geometry}
\usepackage{euler} 
\usepackage{tikz}\usetikzlibrary{graphs,quotes,fit,positioning,matrix,calc,decorations.markings,angles,decorations.pathmorphing,decorations.pathreplacing}
\usetikzlibrary{arrows.meta,decorations.markings}

\tikzset{mynode/.style={circle,draw=black,fill=black,inner sep=1.8pt,outer sep=0pt}}
\tikzset{edgelabel/.style={\mcol,inner sep=0pt}}
\tikzset{invlabel/.style={draw=black,text=black,circle,inner sep=0pt,minimum size=3mm}}
\newcommand\tikzif[2][]{
\tikzifinpicture{#2}{\begin{tikzpicture}[#1]#2\end{tikzpicture}}
}
\tikzset{math mode/.style = {execute at begin node=$, execute at end node=$}}
\def\dr{red!80!black} 
\def\dg{green!80!black} 
\def\db{blue!80!black} 
\tikzset{d/.style={ultra thick}}
\tikzset{dr/.style={draw=\dr,d}}
\tikzset{dg/.style={draw=\dg,d}}
\tikzset{db/.style={draw=\db,d}} 
\def\mcol{black}
\def\m#1{{\color{\mcol}#1}}
\tikzset{rt/.style={text=blue,execute at begin node=$\sf,execute at end node=$}}

\newlength\myshift
\newcommand\getshift{%
\pgfmathsetlength{\myshift}{0.8mm}
}
\def\noparenx#1{%
\ifx\relax#1
\else
\if)#1%
\else
\if(#1%
\else
#1%
\fi
\fi
\expandafter\noparenx
\fi
}
\def\noparen#1{\noparenx #1\relax}
\newcommand\rh[5][]{
\tikzif[baseline=0,scale=1.25]{
	\getshift
	\draw[thick,black] (0,0) coordinate (ad) -- node[edgelabel,left,xshift=\myshift] {$#2$} ++(-60:1) coordinate (ba) -- node[edgelabel,right,xshift=-\myshift] {$#3$} ++(60:1) coordinate (cb) -- node[edgelabel,right,xshift=-\myshift] {$#4$} ++(120:1) coordinate (dc) -- node[edgelabel,left,xshift=\myshift] {$#5$} cycle;
	\ifx\&#1\&\else\node[invlabel] at (0.5,0) {$\ss#1$};\fi
	}}
	
	\newcommand\uptri[4][]{
\tikzif[baseline=0.34cm,scale=1.25]{
	\getshift
	\draw[thick,black] (-0.5,0) -- node[edgelabel] (horiz) {$\vphantom{\noparen{#3}}\smash{#3}$} ++(0:1) -- node[edgelabel,right,xshift=-\myshift] (NE) {$#4$} ++(120:1) -- node[edgelabel,left,xshift=\myshift] (NW) {$#2$} ++(240:1) -- cycle; 
	\ifx\&#1\&\else\node[invlabel] at (0,0.33) {$\ss#1$};\fi
	}}
	\newcommand\downtri[4][]{
\tikzif[baseline=-0.54cm,scale=1.25]{
	\getshift
	\begin{scope}[scale=-1]
		\draw[thick,black] (-0.5,0) -- node[edgelabel] (horiz) {$\vphantom{\noparen{#3}}\smash{#3}$} ++(0:1) -- node[edgelabel,left,xshift=\myshift] (NE) {$#4$} ++(120:1) -- node[edgelabel,right,xshift=-\myshift] (NW) {$#2$} ++(240:1) -- cycle; 
		\ifx\&#1\&\else\node[invlabel] at (0,0.33) {$\ss#1$};\fi
	\end{scope}
	}}
	\tikzset{ 
	centerzero/.style={>=To,baseline={([yshift=-0.5ex](#1))}},
	centerzero/.default={0,0}
}

\tikzset{mynode/.style={circle,draw=black,fill=black,inner sep=1.8pt,outer sep=0pt}}
\tikzset{edgelabel/.style={\mcol,inner sep=0pt}}
\tikzset{invlabel/.style={draw=black,text=black,circle,inner sep=0pt,minimum size=3mm}}
\tikzset{triv/.style={circle,fill=black,inner sep=0.2mm}}
\tikzset{math mode/.style = {execute at begin node=$, execute at end node=$}}
\def\dr{red!80!black} 
\def\dg{green!80!black} 
\def\db{blue!80!black} 
\tikzset{d/.style={ultra thick}}
\tikzset{dr/.style={draw=\dr,d}}
\tikzset{dg/.style={draw=\dg,d}}
\tikzset{db/.style={draw=\db,d}} 
\def\mcol{}
\def\m{}
\tikzset{rt/.style={text=blue,execute at begin node=$\sf,execute at end node=$}}

\tikzset{arrow/.style={postaction={decorate,decoration={markings,mark = at position #1 with {\arrow{Straight Barb[line width=0.2mm,length=1.2mm]}}}}},arrow/.default=0.5}
\tikzset{invarrow/.style={postaction={decorate,decoration={markings,mark = at position #1 with {\arrowreversed{Straight Barb[line width=0.2mm,length=1.2mm]}}}}},invarrow/.default=0.5}
\def\noparenx#1{%
	\ifx\relax#1
	\else
	\if)#1%
	\else
	\if(#1%
	\else
	#1%
	\fi
	\fi
	\expandafter\noparenx
	\fi
}
\def\noparen#1{\noparenx #1\relax}

\usepackage{hyperref}
\usepackage[capitalize]{cleveref}   
\allowdisplaybreaks[1]
\renewcommand\ss{\scriptstyle}

\newcommand\Otimes\bigotimes

\newcommand\into\hookrightarrow
\newcommand\onto\twoheadrightarrow
\newcommand\otno\twoheadleftarrow
\newcommand\dom\backslash
\newcommand\tensor\otimes
\newcommand\Tensor\bigotimes
\newcommand\iso\cong

\tikzset{gauged/.style={rectangle,rounded corners=2mm,draw,inner sep=0.5mm,minimum size=4mm}}
\tikzset{framed/.style={rectangle,draw,inner sep=0.5mm,minimum size=4mm}}
\tikzset{script math mode/.style = {execute at begin node=$\ss , execute at end node=$}}

\newcommand\actson\circlearrowright

\allowdisplaybreaks[1]
\renewcommand\ss{\scriptstyle}

\renewcommand\rightmark{\footnotesize Motivic Segre classes of Schubert cells and the connective formal group law} 

\theoremstyle{plain}
\newtheorem{lem}{Lemma}[section]

\newtheorem{prop}[lem]{Proposition}

\newtheorem{theo}[lem]{Theorem}

\newtheorem{assumption}[lem]{Assumption}
\newtheorem{problem}[lem]{Problem}

\theoremstyle{definition}
\newtheorem{egg}[lem]{Example}
\newtheorem{dfn}[lem]{Definition}
\newtheorem{rmk}[lem]{Remark}
\newtheorem{nota}[lem]{Notation}

\crefname{lem}{Lemma}{Lemmas}
\crefname{cor}{Corollary}{Corollaries}
\crefname{prop}{Proposition}{Propositions}
\crefname{alg}{Algorithm}{Algorithms}
\crefname{theo}{Theorem}{Theorems}
\crefname{calc}{Calculation}{Calculations}
\crefname{assumption}{Assumption}{Assumptions}
\crefname{problem}{Problem}{Problems}
\crefname{egg}{Example}{Examples}
\crefname{dfn}{Definition}{Definitions}
\crefname{rmk}{Remark}{Remarks}
\crefname{nota}{Notation}{Notations}
\crefname{conj}{Conjecture}{Conjectures}

\newcommand{\mrm}{\mathrm}
\newcommand{\mbb}{\mathbb}
\newcommand{\mcal}{\mathcal}

\newcommand{\arxiv}[1]{\href{http://arxiv.org/abs/#1}{\tt arXiv:\nolinkurl{#1}}}

\newcommand{\arxivpdf}[1]{\href{http://arxiv.org/pdf/#1}{\tt arXiv:\nolinkurl{#1}}}

\def\mcol{black}
\def\m#1{{\color{\mcol}#1}}

\newtoggle{comments}
\newtoggle{details}
\newtoggle{detailsnote}

\toggletrue{comments}   
\togglefalse{details}   

\toggletrue{detailsnote}   

\iftoggle{comments}{%
	\newcommand{\comments}[1]{
		\ \\
		{\color{red}
			\textbf{RG:} #1
		}
		\\
	}
	
}{%
	\newcommand{\comments}[1]{}
}

\iftoggle{details}{%
	\newcommand{\details}[1]{
		\ \\
		{\color{OliveGreen}
			\textbf{Details:} #1
		}
		\\
	}
}{%
	\newcommand{\details}[1]{}
}

\begin{document}
	
	\title[]{Motivic Segre classes of Schubert cells and the \\ connective formal group law}
	
	\author[]{Raj Gandhi}
	
	\address{
		Department of Mathematics \\
		Cornell University
	}
	\email{rg593@cornell.edu}
	
	\address{} 
	\thanks{}
	
	\address{}
	\thanks{}
	
	\subjclass[2020]{}
	
	\keywords{}
	
	\maketitle

	\begin{abstract} 
		We use the connective formal group law to define a one-parameter ($\beta$-)deformation of the motivic Segre classes of Schubert cells in the $d$-step flag variety. This $\beta$-deformation specializes to the motivic Segre classes of Schubert cells when $\beta=1$ and to the Segre-Schwartz-MacPherson classes of Schubert cells when $\beta=0$ (though the $\beta$-deformations are not simply the $\beta$-homogenizations of the motivic Segre classes). We define rational function representatives for the $\beta$-deformed classes in the $d=1$ case in terms of a solvable lattice model, and we prove a combinatorial formula for the structure constants in the $\beta$-deformed basis in the $d=1$ case using Knutson-Tao puzzles. Unexpectedly, the structure constants are built out of entries of $R$-matrices for the multi-parameter quantum group of type $\widehat{a}_2$ (which suggests an action of multi-parameter quantum groups on equivariant connective $K$-rings of Nakajima quiver varieties, with their convolution algebra structures). We show that our $\beta$-deformations arise from canonical elements in a quotient of the torus-equivariant algebraic cobordism ring of the cotangent bundle of the flag variety by proving that the canonical elements satisfy a GKM type condition (in future work, we will show that these canonical elements also lie in the torus-equivariant connective $K$-ring, in the sense of Karpenko--Merkurjev, of this cotangent bundle). 
	\end{abstract}
	
	\tableofcontents

	\section{Introduction}
	
	Schubert calculus is a field of mathematics that interacts with combinatorics, representation theory, and algebraic geometry.	In this paper, we will explore connections between these three disciplines, and for the benefit of the reader, we have included a lot of the relevant background material in this introduction. In \cref{subsection:LR-rule}, we will give an overview of the combinatorial aspects of Schubert calculus which will be relevant to us. In particular, we will describe Knutson-Tao puzzles and explain how quantum integral systems have been applied to compute the structure constants in the ``Schubert basis" in the equivariant cohomology ring and equivariant $K$-ring of the Grassmannian and its cotangent bundle. In \cref{subsection:generalized-cohomology-theories}, we will give an overview of how Schubert classes can be defined for generalized cohomology rings of flag varieties, and we will explain where formal group laws appear in this theory. The main results in this paper are summarized in \cref{subsection:research-results}. The organization of this paper is described in \cref{subsection:organization}.
	
	\subsection{Littlewood-Richardson rules and Schubert calculus}\label{subsection:LR-rule}
	When $G=\mrm{GL}_n(\mbb{C})$ is the group of invertible $n\times n$ matrices with entries in $\mbb{C}$, and  $P$ is the subgroup of \textit{block upper-triangular} matrices in $\mrm{GL}_n(\mbb{C})$, the set $G/P$ can be thought of as the set of sequences of vector spaces in $\mbb{C}^n$ of the form
	\[(0)\subsetneq V_{1}\subsetneq V_{2}\subsetneq \dotsb\subsetneq V_{d}\subsetneq \mbb{C}^n,\quad \mrm{dim}(V_{j})=i_j,\]
	where the sequence $(i_1<i_2\dotsb < i_d)$ depends entirely on $P$. The set $G/P$ has the structure of a smooth, projective algebraic variety defined over $\mbb{C}$, and $G/P$ is is called a \textbf{flag variety}. In the special case when $P$ is a \textit{maximal} parabolic subgroup of $G$, the 
	space $G/P$ is the \textbf{Grassmannian} consisting of $k$-planes in $\mbb{C}^n$. When $P=B$ is the set of upper-triangular matrices in $G$, the space $G/B$ is called the \textbf{complete flag variety}. We will denote by $B_-$ the subgroup of lower-triangular matrices in $G$.

	Let $W=S_n$ be the \textit{Weyl group} of $G/P$. If $(i_1<i_2\dotsb < i_d)$ is the sequence corresponding to $P$, then set $W_P=S_{i_1}\times S_{i_2}\times \dotsc S_{i_d}\times S_{n-i_1-\dotsb-i_d}$. Let $W^P$ be the minimal length coset representatives in $W/W_P$. 
	For each $w\in W^P$, one can define a locally closed subvariety $X_w^\circ:=B_- w P/P$ in $G/P$, called a \textbf{Schubert cell}. The closure of a Schubert cell $X_w:=\overline{X_w^\circ}$ in $G/P$ is called a \textbf{Schubert variety}\footnote{We will define a Schubert variety to be a $B_-$-orbit closure $\overline{B_- w P/P}$ in $G/P$ everywhere in this paper, except in \cref{subsection:generalized-cohomology-theories}, where we will assume that a Schubert variety is a $B$-orbit closure $\overline{B w P/P}$ in $G/P$. Our precise conventions are given in \cref{notation:Schubert}, and the reasoning behind these choices is explained in \cref{rmk:reason-for-these-choices}.}, and $G/P$ has a stratification:
	\begin{equation}\label{eqn:Bruhat-decomposition}G/P=\sqcup_{w\in W^P}X_w^\circ.
	\end{equation}
	Consider the singular cohomology ring $H^*(G/P)$. This ring has an extra structure: it is a ring with basis. The classes $[X_w]$ associated to the Schubert varieties $X_w$ form a $\mbb{Z}$-basis for $H^*(G/P)$. In particular, the product of two Schubert classes can be expanded uniquely as a $\mbb{Z}$-linear combination of the Schubert classes:
	\[[X_u]\cdot [X_v]=\sum_{w\in W^P}(c_{u,v}^w)^H\cdot [X_w], \quad (c_{u,v}^w)^H
	\in\mbb{Z}.\]
	The coefficients $(c_{u,v}^w)^H$ are called \textbf{Littlewood-Richardson coefficients}, and they are non-negative as they count \textit{intersection multiplicities}. We can now state what is arguably the most important problem in Schubert calculus:
	\begin{equation}\label{eqn:formula}\textit{Find a combinatorial formula for the $(c_{u,v}^w)^H$.}
	\end{equation}
	The problem (\ref{eqn:formula}) remains wide open outside of special cases.
	We will hereby call any combinatorial formula that computes the $(c_{u,v}^w)^H$ a \textbf{Littlewood-Richardson rule}.

	A permutation $w$ in $S_n$ is said to have a \textbf{descent} at position $k$ if, when written in one-line notation
	\[w=\left(\begin{matrix} 
		1 & 2 & \cdots & n\\
		\text{ }i_1 & \text{ }i_2 & \cdots & \text{ }i_n
	\end{matrix}\right),
	\]
	the permutation satisfies $i_k>i_{k+1}$. There are well-known combinatorial formula for the $(c_{u,v}^w)^H$ in several special cases, one being when $u,v,w$ are \textit{Grassmannian permutations}: a permutation in $S_n$ is called $k$-\textbf{Grassmannian} if it has no descents or if it has exactly one descent and it occurs at position $k$. If $G/P=\mrm{Gr}(k,n)$, then the indexing set $W^P$ in the stratification (\ref{eqn:Bruhat-decomposition}) is the set of $k$-Grassmannian permutations in $S_n$. The $k$-Grassmannian permutations in $S_n$ are in bijection with  $\Gamma_{k,n}$, the set of sequences of $0$'s and $1$'s consisting exactly $k$ $1$'s and $n-k$ $0$'s.

	We will focus our attention now on $\mrm{Gr}(k,n)$, identifying $W^P$ with $\Gamma_{k,n}$. One of the many combinatorial formulas for the structure constants $(c_{u,v}^w)^H$, now viewing $\lambda,\mu,\nu$ as sequences in $\Gamma_{k,n}$, is the \textit{Knutson-Tao puzzle formula}. A \textbf{Knutson-Tao puzzle} with side labels $\lambda$, $\mu$, $\nu$ in $\Gamma_{k,n}$ is a triangle $\tikz[scale=0.8,baseline=0cm]{\uptri{\lambda}{\nu}{\mu}}$ with side labels $\lambda$, $\mu$, $\nu$ that is tiled by the following \textbf{puzzle pieces} with edge labels $0$, $1$, and $10$, and allowing rotations of the puzzle pieces:
	
	\begin{equation}\label{eqn:basic-tiles}
		\uptri000 \quad \quad
		\uptri111 
		\quad \quad
		\uptri01{10}	
	\end{equation}
	In \cite[Theorem 1]{KTW04}, Knutson, Tao, and Woodward showed that the number of such tilings equals the Littlewood-Richardson coefficient $(c_{u,v}^w)^H$. This gives a Littlewood-Richardson rule (\ref{eqn:formula}) in the special case that $\lambda,\mu,\nu$ are $k$-Grassmannian permutations.

	The problem posed in (\ref{eqn:formula}) has many natural generalizations, one being to equivariant cohomology. The flag variety $G/P$ has a natural action of the $n\times n$ diagonal matrices $T$ in $G$. The Schubert classes $X_w$ are invariant under this $T$-action, and from this we deduce that they define classes $[X_w]$ in the $T$-equivariant cohomology ring $H_T(G/P)$. Analogous to the case of $H^*(G/P)$, the classes $[X_w]$ form a basis for $H_T(G/P)$ over $H_T(\mrm{point})\simeq \mbb{Z}[y_1,\dotsc,y_n]$, allowing one to expand the product of two \textit{equivariant} Schubert classes \textit{uniquely} as a $\mbb{Z}[y_1,\dotsc,y_n]$-linear combination of the equivariant Schubert classes:
	\begin{equation}\label{eqn:equivariantLR}
		[X_u]\cdot [X_v]=\sum_{w\in W^P}(c_{u,v}^w)^{H_T}\cdot [X_w], \quad (c_{u,v}^w)^{H_T}
		\in\mbb{Z}[y_1,\dotsc,y_n].
	\end{equation}
	Analogously to the problem (\ref{eqn:formula}), one can pose the following problem:
	\begin{equation}\label{eqn:equivariantformula}\textit{Find a combinatorial formula for the $(c_{u,v}^w)^{H_T}$.}
	\end{equation}
	As before, we will call any combinatorial formula that computes the $(c_{u,v}^w)^{H_T}$ a \textbf{Littlewood-Richardson rule}. This problem is open outside of special cases.
	
	Assume that $u,v,w$ are \textit{$k$-Grassmannian} permutations. In this case, the first combinatorial formula for the coefficient $(c_{u,v}^w)^{H_T}$ of (\ref{eqn:equivariantLR}) was introduced in \cite{KT03} by Knutson and Tao in terms of puzzles. Their puzzles are tiled by the puzzles pieces in (\ref{eqn:basic-tiles}) together with \textit{an additional puzzle piece}. The additional piece is a rhombus, which carries a \textit{weight} of $y_i-y_j$ if it is positioned in the $i$-th southwest-to-northeast diagonal and $j$-th southeast-to-northwest diagonal in the puzzle:
	
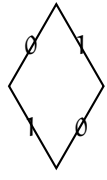
\[
	\begin{gathered}
		\rh1{0}10
	\end{gathered}
	\]
	A very insightful example of a computation of (equivariant) Littlewood-Richardson coefficients using puzzles can be found just below Theorem 3 in the expository article \cite{K23}.
	
	The problem posed in (\ref{eqn:formula}) can be generalized in an orthogonal direction to the $K$-ring of locally free coherent sheaves on $G/P$, denoted $K(G/P)$, and its $T$-equivariant version $K_T(G/P)$. The interesting classes to consider in these rings are the classes defined by the structure sheaves of Schubert varieties $[\mcal{O}_{X_w}]$. As one would have hoped, the classes $[\mcal{O}_{X_w}]$ form a basis for $K(G/P)$ over $\mbb{Z}$, and a basis for $K_T(G/P)$ over $K_T(\mrm{point})\simeq \mbb{Z}[e^{\pm y_1},\dotsc,e^{\pm y_n}]$. Consequently, the product of two classes can be expanded uniquely as a linear combination of the $[\mcal{O}_{X_w}]$ over the respective coefficient ring:
	\[[\mcal{O}_{X_u}]\cdot [\mcal{O}_{X_v}]=\sum_{w\in W^P}c_{u,v}^w\cdot [\mcal{O}_{X_w}].\] 
	As earlier, we will call any combinatorial formula  that computes the $c_{u,v}^w$ a \textbf{Littlewood-Richardson rule}.
	In both $K(\mrm{Gr}(k,n))$ and $K_T(\mrm{Gr}(k,n))$, there are combinatorial formulas for the structure constants $c_{\lambda,\mu}^\nu$ in terms of Knutson-Tao puzzles. The formula in $K(\mrm{Gr}(k,n))$ was proven by Vakil \cite{V06}, and two different puzzle formulas in $K_T(\mrm{Gr}(k,n))$ were proven, one by Pechenik and Yong in \cite{PY17I}, \cite{PY17II}, and the other by Wheeler and Zinn-Justin in \cite{WZJ19}.
	
	The Littlewood-Richardson rules for $H_{T}(G/P)$ and $K_T(G/P)$ can be generalized by incorporating an additional parameter $q$ in the following way. The cotangent bundle $T^*(G/P)$ has a natural action of $T$ that acts on the base $G/P$, together with an action of an additional $\mbb{C}^\times$ factor which scales cotangent fibres with weight $q$. Set $\widehat{T}:=T\times\mbb{C}^\times$. The $\widehat{T}$-equivariant cohomology ring $H_{\widehat{T}}(T^*(G/P))$ has a natural basis over the fraction field $\mrm{Frac}(H_{T\times \mbb{C}^\times}(\mrm{point}))\simeq \mbb{Q}(y_1,\dotsc,y_n,q)$ indexed by $W^P$ called the basis of \textit{Segre-Schwartz-MacPherson classes of Schubert cells} (i.e., \textit{SSM classes}), and the $\widehat{T}$-equivariant $K$-ring $K_{\widehat{T}}(T^*(G/P))$ has a natural basis over the fraction field $\mrm{Frac}(K_{\widehat{T}}(\mrm{point}))\simeq \mrm{Frac}\left(\mbb{Z}[e^{\pm y_1},\dotsc,e^{\pm y_n},q^{\pm 1}]\right)$ indexed by $W^P$ called the basis of \textit{motivic Segre classes of Schubert cells} (i.e., \textit{mS classes}). We will denote the SSM and mS classes uniformly by $S_w$. In \cite{KZJ17,KZJ21}, Knutson and Zinn-Justin proved a combinatorial formula for the structure constants in the $\{S_\lambda\}_{\lambda\in\Gamma_{k,n}}$ basis in both $H_{\widehat{T}}(T^*(\mrm{Gr}(k,n)))$ and $K_{\widehat{T}}(T^*(\mrm{Gr}(k,n)))$ in terms of Knutson-Tao puzzles. In both cases, there are 15 rhombus puzzle pieces, whose weights depend on a parameter $q$, and in the limit $q\to 0$ the puzzle formulas in both rings collapse to those in $H_T(\mrm{Gr}(k,n))$ and $K_T(\mrm{Gr}(k,n))$. The ``Littlewood-Richardson rule" in $K_{\widehat{T}}(T^*(\mrm{Gr}(k,n)))$ specializes to the Littlewood-Richardson rules in $H^*(\mrm{Gr}(k,n))$, $H_T(\mrm{Gr}(k,n))$, $K(\mrm{Gr}(k,n))$, $K_T(\mrm{Gr}(k,n))$, and $H_{\widehat{T}}(T^*(\mrm{Gr}(k,n)))$ by taking certain limits.
	
	The techniques used to prove the Littlewood-Richardson rule for the $S_\lambda$ basis in $K_{\widehat{T}}(T^*(\mrm{Gr}(k,n)))$ go back to work of Zinn-Justin in \cite{ZJ09}. In that paper, Zinn-Justin produced a \textit{vertex model} (also known as a \textit{tiling model} or \textit{lattice model}) that simultaneously computes both \textit{double Schur polynomials} (which are polynomial representatives of Schubert classes in $H_T(\mrm{Gr}(k,n))$) and the associated Littlewood-Richardson coefficients, giving a new proof of the Knutson-Tao puzzle formula. In \cite{WZJ19}, Zinn-Justin and Wheeler developed a vertex model that simultaneously computes both \textit{double Grothendieck polynomials for Grassmannian permutations} (which are polynomial representatives for the basis of structure sheaves $[\mcal{O}_{X_\lambda}]$ in $K_T(\mrm{Gr}(k,n))$), and the associated Littlewood-Richardson coefficients, giving a Littlewood-Richardson rule for $K_T(\mrm{Gr}(k,n))$ in terms of Knutson-Tao puzzles. In \cite{KZJ17,KZJ21}, Knutson and Zinn-Justin were able to bootstrap the techniques of \cite{ZJ09} and \cite{WZJ19} to deal with the more general Schubert calculus in  $H_{\widehat{T}}(T^*(\mrm{Gr}(k,n)))$ and $K_{\widehat{T}}(T^*(\mrm{Gr}(k,n)))$. Knutson and Zinn-Justin constructed a vertex model that simultaneously computes both \textit{rational function representatives} for the $S_\lambda$ basis in  $H_{\widehat{T}}(T^*(\mrm{Gr}(k,n)))$ or in $K_{\widehat{T}}(T^*(\mrm{Gr}(k,n)))$, and the structure constants in these bases in terms of Knutson-Tao puzzles. The vertex model incorporates the representation theory of the quantum group $U_q(\widehat{\mathfrak{a}}_2)$, and the proof of the puzzle formula requires manipulating ``$R$-matrices" via the ``Yang-Baxter equation".

	\subsection{Formal group laws and generalized Schubert calculus}\label{subsection:generalized-cohomology-theories}
	
	A \textbf{formal group law} is a pair $(R,F)$, where $R$ is a commutative, unital ring and $F(x,y)=\sum_{i,j\geq 0}a_{i,j}x^iy^j$ is a formal power series in $R\llbracket x,y\rrbracket$, satisfying three axioms: \textit{unity} (i.e., $F(0,x)=x$), \textit{commutativity} (i.e., $F(x,y)=F(y,x)$), and \textit{associativity} (i.e., $F(x,F(y,z))=F(F(x,y),z)$). These axioms impose relations on the coefficients $a_{i,j}$. The unity axiom implies that $a_{0,0}=0$ and $a_{1,0}=a_{0,1}=1$. The commutativity axiom implies that $a_{i,j}=a_{j,i}$. The relations arising from associativity are very complicated. Exposition on formal group laws can be found in, for example, \cite{F68}. Apart from being very interesting objects in themselves, formal group laws carry deep geometric information. That is, a formal group law can be viewed as a combinatorial shadow of an algebraic invariant of smooth algebraic varieties known as an \textit{algebraic oriented cohomology theory.}

	An \textbf{(algebraic) oriented cohomology theory} (OCT) is a (contravariant) functor
	\[\mrm{h}^*\colon \left\{\text{categ. of smooth, separated, finite type $\mbb{C}$-schemes}\right\}\to \left\{\text{categ. of graded, commutative, unital rings}\right\}\]
	that satisfies various axioms generalizing the \textit{Eilenberg-Steenrod axioms} from algebraic topology. OCTs were defined by Levine and Morel in \cite{LM07}, and the concept of an OCT was inspired by the notion of a \textit{complex-oriented cohomology theory}, originally defined by Daniel Quillen in his paper \cite{Q71}. The \textit{Chow theory}, which sends a variety $X$ to its Chow ring $\mrm{CH}^*(X)$, and the \textit{Grothendieck $K^0$ functor} (which we will call ``$K$-theory"), which sends a variety $X$ to the Grothendieck ring of vector bundles $K^0(X)$ on $X$, are examples of OCTs. As this is the first time we are using the words ``Chow ring," we clarify that for any parabolic subgroup $P$, the singular cohomology ring of $G/P$ is isomorphic to the Chow ring of $G/P$ due to the stratification of $G/P$ given by (\ref{eqn:Bruhat-decomposition}). Therefore, Schubert calculus in the Chow ring $CH^*(G/P)$ of $G/P$ is no different than Schubert calculus in the singular cohomology ring $H^*(G/P)$. Every OCT comes with \textit{Chern classes}, which can be constructed using a method of Grothendieck. For a vector bundle of $\mcal{V}$ rank $n$ on variety $X$, the $i$-th Chern class $c_i(\mcal{V})$ resides in $\mrm{h}^i(X)$. Chern classes commute with pullbacks and satisfy the ``Whitney sum formula". 
	Every OCT $\mrm{h}^*$ is equipped with a formal group law, and the formal group law of an OCT is determined by the following property: given two line bundles $\mcal{L}_1,\mcal{L}_2$ on a smooth variety $X$, 
	\[c_1(\mcal{L}_1\otimes \mcal{L}_2)=F(c_1(\mcal{L}_1),c_1(\mcal{L}_2)),\]
	where $F$ is the formal group law of $\mrm{h}^*$. Chow rings gives rise to the \textbf{additive formal group law} $F(x,y)=x+y$ over $R=\mbb{Z}$, and $K$-theory gives rise to the \textbf{multiplicative formal group law} $F(x,y)=x+y-xy$ over $R=\mbb{Z}$. The \textit{universal} OCT is called \textit{algebraic cobordism} $\Omega^*$, and it has a unique map to any other oriented cohomology theory. The algebraic cobordism ring of a smooth variety $X$ can be constructed formally as the set of projective morphisms $Y\to X$, with $Y$ being smooth, quasi-projective, and irreducible, modulo certain relations. A geometric construction for algebraic cobordism was developed by Levine and Pandharipande in \cite{LP09} in terms of so-called "double-point degenerations". The formal group law of algebraic cobordism is the \textbf{universal formal group law} defined over the \textbf{Lazard ring}. 
	
	Our primary focus will be the formal group law $F_C(x,y)=x+y-\beta xy$ defined over $R=\mbb{Z}[\beta]$. We will hereby call the pair $(R,F_C)$ the \textbf{connective formal group law}. The \textit{universal OCT whose formal group law is the connective one} is called \textbf{algebraic connective $K$-theory} $CK^*$. A geometric construction for $CK^*$ was given by Cai in \cite{C08}, and by Dai and Levine in \cite{DL14}. A key property of the functor $\mrm{CK}^*$ is that for any smooth algebraic variety $X$, there are always natural ring homomorphisms $\mrm{CK}^*(X)\to CH^*(X)$ and $\mrm{CK}^*(X)\to K^0(X)$ induced by evaluating the parameter $\beta$ at $0$ and at $1$, respectively.
	
	In \cref{notation:Schubert} below, will note the different conventions for ``Schubert varieties" used in the literature, and explain how we will handle these different conventions in the current subsection and in this paper. The reasoning behind these choices is explained in \cref{rmk:reason-for-these-choices}.
	
	\begin{nota}\label{notation:Schubert}
		In the current subsection (\cref{subsection:generalized-cohomology-theories}), and \textit{only} in the current subsection, we will follow the conventions used in \cite{CPZ,CZZ3,CZZ2}, and define a \textit{Schubert variety} to be a $B$-orbit closure $\overline{BwP/P}$ in $G/P$. 
		Everywhere else in this paper, apart from the current subsection, we will follow the conventions used in \cite{KZJ17,KZJ21} and define a \textit{Schubert variety} to be a $B_-$-orbit closure $\overline{B_-wP/P}$ in $G/P$, which agrees with our definition of \textit{Schubert variety} in \cref{subsection:LR-rule}. 
	\end{nota}
	
	\begin{rmk}\label{rmk:reason-for-these-choices}
		The reasoning behind our choices of conventions in \cref{notation:Schubert} is as follows. In the current subsection, we will motivate the main results of this paper by citing many constructions from \cite{CPZ,CZZ3,CZZ2} regarding Schubert calculus in generalized cohomology rings of flag varieties. In those papers, a ``Schubert variety" is defined to be a $B$-orbit closure $\overline{BwP/P}$ in $G/P$. Although we expect that results analogous to those we cite from \cite{CPZ,CZZ3,CZZ2} should hold for the $B_-$-orbit closures $\overline{B_-wP/P}$ in $G/P$, those analogous statements do not seem to be written down explicitly in the literature. Therefore, solely in the current subsection, for the sake of matching the conventions used in those three papers, a ``Schubert variety" will be taken to mean a $B$-orbit closure $\overline{BwP/P}$ in $G/P$ rather than a $B_-$-orbit closure $\overline{B_-wP/P}$. The reason that we define a ``Schubert variety" to be a $B_-$-orbit closure $\overline{B_-wP/P}$ in $G/P$ everywhere else in this paper, apart from the current subsection, is because this paper directly generalizes results from \cite{KZJ17,KZJ21}, and these papers use the convention that a ``Schubert variety" is a $B_-$-orbit closure in $G/P$. For the rest of the current subsection, we will explain how \textit{Schubert calculus} can be defined for generalized cohomology rings of flag varieties, and we will describe the geometric and algebraic intuitions that led to our main results. Our main results are described in detail in \cref{subsection:research-results}. 
	\end{rmk}

	One definition of ``generalized Schubert calculus" is the study of Schubert classes and Littlewood-Richardson rules in ``generalized cohomology theories". The first systematic study of Schubert calculus for arbitrary complex-oriented cohomology theories was initiated by Bressler and Evens in \cite{BE90}, \cite{BE92}. 	Among their contributions was the construction of a \textit{Schubert class} in any ``complex-oriented cohomology theory".	To construct it, first we need the following notion: a pair of complete flags $(F_\bullet^{(1)},F_\bullet^{(2)})$ has \textbf{relative position $s_{i}:=(i,i+1)$} in $S_n$ if $F_\bullet^{(1)}=F_\bullet^{(2)}$ \textit{or} if the flags differ \textit{only in} their $i$-th vector spaces. We will begin by taking a sequence of simple transpositions $Q=(s_{i_1},\dotsc,s_{i_k})$ in $S_n$, and considering the set
	\begin{equation}\label{eqn:Bott-Samelson-Variety}\mcal{X}^Q:=\left\{(F_\bullet^{(0)},F_\bullet^{(1)},\dotsc,F_\bullet^{(k)})\in\mrm{Fl}_n^{k+1}\mid  F_\bullet^{(0)}=\text{base flag, and} (F_\bullet^{(j)},F_\bullet^{(j+1)}) \text{ has ``relative position" $s_{i_j}$} \right\}.
	\end{equation}
	The set $\mcal{X}^Q$ has the structure of a smooth, projective algebraic variety over $\mbb{C}$, and it is known as a \textbf{Bott-Samelson variety}. This variety has a $B$-equivariant map $\pi_Q\colon \mcal{X}^Q\to G/B$ given by taking the \textit{last flag}. For each $w\in S_n$, we will choose a \textit{reduced word} $I_w$ for $w$ (in other words, $I_w=(s_{i_1},\dotsc,s_{i_k})$ and $w=s_{i_1}\dotsb s_{i_k}$, and this product is \textit{reduced}). For any choice of $I_w$ the image of $\pi_{I_{w}}\colon \mcal{X}^{I_w}\to G/B$ is the Schubert variety $X^w:=\overline{BwB/B}$, following the conventions for the current subsection established in \cref{notation:Schubert}, and $\pi_{I_w}$ is in fact a resolution of singularities for $X^w$. Given any \textit{complex-oriented cohomology theory} $\mrm{h}^*$, the pushforward of the fundamental class $1$ of $\mrm{h}^*(\mcal{X}^{I_w})$ to $\mrm{h}^*(G/B)$ is called a \textbf{Schubert class} $\xi_{I_w}$. A result of Bressler and Evens is that $\xi_{I_w}$ is independent of the choice of reduced word $I_w$ \textit{if and only if} $F(x,y)=x+y-\beta xy$ for some $\beta\in \mrm{h}^*(\mrm{point})$. If $F(x,y)$ is \textit{not} of the form $x+y-\beta xy$, then the Schubert classes $\xi_{I_w}$ will \textit{depend} on the choices of reduced words $I_w$. If $\mrm{h}^*=H^*$, then $\xi_{I_w}=[X^w]\in H^*(G/B)$, and if $\mrm{h}^*=K^0$, then $\xi_{I_w}=[\mcal{O}_{X^w}]\in K^0(G/B)$.
	
	The study of Schubert calculus for general \textit{algebraic} OCTs was initiated independently by Calm\'{e}s, Petrov, and Zainoulline in \cite{CPZ}, and by Hornbostel and Kiritchenko in \cite{HK11}. As in the complex-oriented case, for any OCT $\mrm{h}^*$ the Schubert classes $\xi_{I_w}$ in $\mrm{h}^*(G/B)$ are defined to be the pushforwards of the fundamental classes of the Bott-Samelson resolutions in $\mrm{h}^*(\mcal{X}^{I_w})$, and moreover, the classes $\{\xi_{I_w}\}_{w\in S_n}$ form an $\mrm{h}^*(\mrm{point})$-basis for $\mrm{h}^*(G/B)$ (\cite[\S 13]{CPZ}). The Schubert classes are independent of the choices of reduced words $I_w$ \textit{if and only if} $F(x,y)=x+y-\beta xy$. In particular, the Schubert classes $\xi_{I_w}$ in $CK^*(G/B)$ are \textit{independent of the choices of reduced words $I_w$}, so we will use the notation $\xi_w:=\xi_{I_w}$.
	Under the natural ring homomorphisms $CK^*(G/B)\to \mrm{CH}^*(G/B)$ and $CK^*(G/B)\to K^0(G/B)$, the Schubert class $\xi_w$ is mapped to the ordinary Schubert class $[X^w]$ in $\mrm{CH}^*(G/B)$, and to the class $[\mcal{O}_{X^w}]$ in $K^0(G/B)$.

	There exists an \textit{equivariant} version of algebraic cobordism, which was constructed independently by Krishna in \cite{K12}, and by Heller and Malag\'{o}n-Lopez in \cite{HM13}. Their constructions incoporate the action of a smooth linear algebraic group on a scheme. The equivariant algebraic cobordism ring of a scheme is a module over the equivariant algebraic cobordism ring of a point. For $T\simeq (\mbb{C}^\times)^n$ a complex torus, the $T$-equivariant algebraic cobordism ring of a point $\Omega_T(\mrm{point})$ is a complete and Hausdorff topological ring, and this ring has an algebraic model that depends only on the character lattice $\Lambda:=\mrm{Hom}(T,\mbb{C}^\times)$ of the torus $T$ and the universal formal group law $(\mbb{L},F_U)$. The algebraic model is called the \textit{formal group algebra} $\mbb{L}\llbracket\Lambda\rrbracket_{F_U}$, and was constructed in \cite{CPZ} by Calm\'es, Petrov, and Zainoulline. One can replace the universal formal group law $(\mbb{L},F_U)$ defining $\mbb{L}\llbracket\Lambda\rrbracket_{F_U}$ by \textit{any} formal group law $(R,F)$ and obtain a new algebra $R\llbracket\Lambda\rrbracket_F$. Consider the subring $\mcal{F}(R,F)$ of $\bigoplus_{w\in W}R\llbracket \Lambda\rrbracket_F$ satisfying the so called \textit{GKM conditions}, i.e., the elements of $\mcal{F}(R,F)$ are sequences $(f_w)_{w\in W}$ such that $f_w-f_{s_\alpha w}$ is divisible by $x_\alpha$ in $R\llbracket \Lambda\rrbracket_F$, where $\alpha$ is a \textit{root} in $\Lambda$ and $x_\alpha$ is the element indexed by $\alpha$ (see \cref{dfn:GKM-conditions}). Notably, $\mcal{F}(\mbb{L},F_U)\simeq \Omega_{T}(G/B)$, and there is \textit{always} a ring homomorphism $\Omega_{T}(G/B)\to \mcal{F}(R,F)$ (see \cref{lem:surjective-restriction}). In \cref{rmk:inclusion of rings0}, we explain that there are inclusions of rings:
	\begin{equation}\label{eqn:inclusions}
		\mrm{CH}_T(G/B)\hookrightarrow \mcal{F}(\mbb{Z},F_A)\quad \text{and}\quad K_T(G/B)\hookrightarrow \mcal{F}(\mbb{Z},F_M).
	\end{equation}

	As in the non-equivariant case, if we choose a reduced word $I_w$ for each $w\in S_n$, then the Schubert classes $\xi_{I_w}$ in $\Omega_T(G/B)$ can be defined as the pushforwards of the Bott-Samelson resolutions $\mcal{X}^{I_w}\to G/B$ in $T$-equivariant algebraic cobordism (see \cite{CZZ3}). These Schubert classes form a basis for $\Omega_T(G/B)$ over $\Omega_T(\mrm{point})$, and the basis \textit{depends} on the choices of reduced words $I_w$. Under the ring homomorphism $\Omega_{T}(G/B)\to \mcal{F}(R,F)$, the image $\xi_{I_w}^{(R,F)}$ of the Schubert class $\xi_{I_w}$ in the ring $\mcal{F}(R,F)$ is \textit{independent} of the choice of reduced word $I_w$ for $w$ if and only if $F(x,y)=x+y-\beta xy$ for some $\beta\in R$. Thus, the Schubert basis $\left\{\xi_{I_w}^{(\mbb{Z}[\beta],F_C)}\right\}_{w\in W}$ in $\mcal{F}(\mbb{Z}[\beta],F_C)$ is independent of the choices of reduced words $I_w$, so we may as well use the notation $\xi_w:=\xi_{I_w}^{(\mbb{Z}[\beta],F_C)}$. There are natural maps from $\mcal{F}(\mbb{Z}[\beta],F_C)$ to both $\mcal{F}(\mbb{Z},F_A)$ and $\mcal{F}(\mbb{Z},F_M)$, which are induced by evaluating the parameter $\beta$ at $0$ and $1$, respectively. Under the $\beta=0$ (resp. $\beta=1$) evaluation, the equivariant ``connective" Schubert class $\xi_w$ in $\mcal{F}(\mbb{Z}[\beta],F_C)$ coincides with the usual equivariant Schubert class $[X^w]$ (resp. $[\mcal{O}_{X^w}]$) in $CH_T(G/B)$ (resp. $K_T(G/B)$). We record these evaluations in the diagram below:	
	\begin{equation}\label{eq:equivariant-ck}\begin{tikzcd}
			& \xi_\lambda \in \mcal{F}(\mbb{Z}[\beta],F_C)  \arrow [dl,rightsquigarrow, swap,"\beta \to 0"] \arrow[dr,rightsquigarrow, "\beta \to 1"] &\\
			\left[X_\lambda\right]\in CH_T(G/B)  & &\left[\mcal{O}_{X_\lambda}\right] \in K_T(G/B)
		\end{tikzcd}
	\end{equation}
	Notably, after the localization of $\mcal{F}(\mbb{Z}[\beta],F_C)$ at $\beta$, the equivariant Schubert classes in $\mcal{F}(\mbb{Z},F_C)$ can be identified with the $\beta$-\textbf{homogenizations} $\beta^{-\ell(w)}[\mcal{O}_{X^w}]$ of the equivariant Schubert classes in $K_T(G/B)\otimes_{\mbb{Z}}\mbb{Z}[\beta,\beta^{-1}]$. From this point of view, the basis of Schubert classes in $\mcal{F}(\mbb{Z}[\beta],F_C)$ is \textit{not} an interesting combinatorial object, as the structure constants in this basis are simply \textit{homogenizations} of the structure constants in the Schubert basis of $K_T(G/B)\otimes_{\mbb{Z}}\mbb{Z}[\beta,\beta^{-1}]$. 
	
	Recall that the Schubert classes in the connective $K$-ring $CK^*(G/B)$ specialize to the Schubert classes in $CH^*(G/B) $ and $K^0(G/B)$ by specializing the parameter $\beta$ to $0$ or $1$, respectively. The theory of \textit{equivariant algebraic connective $K$-theory} was developed by Karpenko and Merkurjev in \cite{KM22}. As the Schubert classes in $\mcal{F}(\mbb{Z}[\beta],F_C)$ specialize to the equivariant Schubert classes in $H_T^*(G/B)$ and $K_T(G/B)$ by specializing the parameter $\beta$ at $0$ or $1$, respectively, and as the inclusions (\ref{eqn:inclusions}) hold, we expect that $\mcal{F}(\mbb{Z}[\beta],F_C)$ contains the $T$-equivariant algebraic connective $K$-ring $CK_T(G/B)$ of $G/B$ as a subring, and that  $CK_T(G/B)$ is generated over $CK_T(\mrm{pt})$ by the equivariant Schubert classes $\xi_w$ in $\mcal{F}(\mbb{Z}[\beta],F_C)$. This intuition, as well as work in progress on the construction of a ``connective" motivic Chern transformation for equivariant algebraic connective $K$-theory, is explained in detail in \cref{rmk:Karpenko}.

	\subsection{Research results}\label{subsection:research-results}

	Recall that the cotangent bundle $T^*(G/P)$ is a smooth quasi-projective variety with the algebraic action of $\widehat{T}:=T\times \mbb{C}^*$, where the last $\mbb{C}^*$ factor acts by scaling cotangent fibres. In Subsection \ref{subsection:LR-rule}, we explained that localizations of the rings $CH_{\widehat{T}}(T^*(G/P))$ and $K_{\widehat{T}}(T^*(G/P))$ have natural bases indexed by $W^P$. The basis elements in the localizations $CH_{\widehat{T}}^{\text{loc}}(T^*(G/P))$ and  $K_{\widehat{T}}^{\text{loc}}(T^*(G/P))$ are called the \textit{SSM classes} and \textit{mS classes}, respectively. In Subsection \ref{subsection:LR-rule}, we denoted the basis elements in both rings uniformly by $S_w$. Let $\widehat{\mcal{F}}(\mbb{Z}[\beta],F_C)$ be the subring of $\bigoplus_{w\in W^P}\mbb{Z}[\beta]\llbracket \Lambda\rrbracket_{F_C}\llbracket x_q\rrbracket$ satisfying the \textit{GKM conditions}; see \cref{dfn:GKM-conditions}.

	The result below is motivated by the diagram (\ref{eq:equivariant-ck}):
	\begin{theo}(\cref{subsection:second-construction}, \cref{theo:geometric-connective-stable})\label{theo:universal-connective}	
		There are elements $\mrm{S}_w$ in $ \widehat{\mcal{F}}(\mbb{Z}[\beta],F_C)^{\text{loc}}$ (where $ \widehat{\mcal{F}}(\mbb{Z}[\beta],F_C)^{\text{loc}}$ is a localization of $ \widehat{\mcal{F}}(\mbb{Z}[\beta],F_C)$) that map to the corresponding elements $S_w$ in both $K_{\widehat{T}}^{\text{loc}}(T^*(G/P))$ and $CH_{\widehat{T}}^{\text{loc}}(T^*(G/P))$:
		\[\begin{tikzcd}
			&S_w \in \widehat{\mcal{F}}(\mbb{Z}[\beta],F_C)^{\text{loc}} \arrow [dl,rightsquigarrow, swap,"\beta \to 0"] \arrow[dr,rightsquigarrow, "\beta \to 1"] &\\
			S_w\in CH_{\widehat{T}}^{\text{loc}}(T^*(G/P))  & &S_w \in K_{\widehat{T}}^{\text{loc}}(T^*(G/P))
		\end{tikzcd}\]
		The elements $S_w\in \widehat{\mcal{F}}(\mbb{Z}[\beta],F_C)^{\text{loc}}$ can be constructed from canonical elements $\mrm{St}_w$ which lie in a quotient of the $\widehat{T}$-equivariant algebraic cobordism ring of $T^*(G/P)$ and satisfy a GKM type condition.
	\end{theo}

	In fact, we will show in future work that there is an analogue of the motivic Chern natural transformation for equivariant algebraic connective $K$-theory (in the sense of Karpenko--Merkurjev \cite{KM22}), and the elements $\mrm{St}_w$ of \cref{theo:universal-connective} can be realized ``geometrically" as the images of Schubert cells of $G/P$ in the $\widehat{T}$-equivariant algebraic connective $K$-ring of $T^*(G/P)$ under this natural transformation. See \cref{rmk:Karpenko} for details on work in progress in this direction, and see \cref{subsection:motivic-Chern-classes} for exposition on the motivic Chern natural transformation for the usual equivariant $K$-theory.
	
	We will call the elements $S_w$ in $\widehat{\mcal{F}}(\mbb{Z}[\beta],F_C)^{\text{loc}}$ of Theorem \ref{theo:universal-connective} the \textbf{connective motivic Segre classes}. We saw in \cref{subsection:generalized-cohomology-theories} that Schubert calculus in ${\mcal{F}}(\mbb{Z}[\beta],F_C)$ boils down to Schubert calculus in $K_T(G/P)$, since the $T$-equivariant Schubert classes in $\mcal{F}(\mbb{Z}[\beta],F_C)$ can be identified with $\beta$-homogenizations of the $T$-equivariant Schubert classes in $K_T(G/P)$. It is natural to ask whether the \textit{connective} motivic Segre classes are identified with the $\beta$-\textit{homogenizations} of the motivic Segre classes $S_w$ in $K_{\widehat{T}}^{\text{loc}}(T^*(G/P))\otimes_{\mbb{Z}}\mbb{Z}[\beta,\beta^{-1}]$. We address this below:
	\begin{lem}(\cref{egg:connective:beta})\label{theo:non-trivial-Schubert}
		The connective motivic Segre classes \underline{do \textit{not}} correspond to $\beta$-homogenizations of the motivic Segre classes in $K_{\widehat{T}}^{\text{loc}}(T^*(G/P))\otimes_{\mbb{Z}}\mbb{Z}[\beta,\beta^{-1}]$. 
	\end{lem}
	As a consequence of Lemma \ref{theo:non-trivial-Schubert}, Schubert calculus in $\widehat{\mcal{F}}(\mbb{Z}[\beta],F_C)^{\text{loc}}$ is strictly more general (and more interesting) than Schubert calculus in $K_{\widehat{T}}^{\text{loc}}(T^*(G/P))$. One might now ask whether there are rational function 
	representatives for the connective motivic Segre classes, and whether there is a combinatorial formula for the structure constants for this basis in terms of Knutson-Tao puzzles. We address this in the case where $G/P=\mrm{Gr}(k,n)$:
	\begin{prop}(\cref{prop:charact})
		Assume $G/P=\mrm{Gr}(k,n)$.	There is a solvable lattice model that defines rational function representatives for the connective motivic Segre classes in $\widehat{\mcal{F}}(\mbb{Z}[\beta],F_C)^{\text{loc}}$.
	\end{prop}
	\begin{theo}(\cref{thm:main}, \cref{rmk:positive}, \cref{thm:main2})
		Assume $G/P=\mrm{Gr}(k,n)$. There is a combinatorial formula for the structure constants in the basis of connective motivic Segre classes in $\widehat{\mcal{F}}(\mbb{Z}[\beta],F_C)^{\text{loc}}$ given in terms of Knutson-Tao puzzles.
	\end{theo}
	
	The puzzle formula of Knutson and Zinn-Justin in \cite{KZJ17}, \cite{KZJ21} computes the structure constants for the motivic Segre classes in $K_{\widehat{T}}^{\text{loc}}(T^*(\mrm{Gr}(k,n)))$. To prove the formula, the authors constructed a lattice model from the weights of the puzzle pieces, and this lattice model simultaneously computes the structure constants and the rational function representatives for the motivic Segre classes in $K_{\widehat{T}}^{\text{loc}}(T^*(\mrm{Gr}(k,n)))$. The weights of the puzzle pieces were arranged into a matrix, and the authors showed that the matrix of puzzle weights is an $R$-matrix for an evaluation representation of the quantum affine algebra $U_q(\widehat{\mathfrak{a}}_2)$ (see Appendix \ref{section:quantum-group}). The $R$-matrices of $U_q(\widehat{\mathfrak{a}}_2)$ satisfy matrix equations called ``Yang-Baxter equations", which allow one to equate different products of $R$-matrices. The utility of the Yang-Baxter equations is that the left-hand-side and right-hand-side of Equation (\ref{eqn:stable-constants}) in $\widehat{K}_{\widehat{T}}^{\text{loc}}(T^*(\mrm{Gr}(k,n)))$ could be equated through a sequence of moves arising from the Yang-Baxter equations:
	\begin{equation}\label{eqn:stable-constants}
		[S_\lambda]\cdot [S_\mu]=\sum_{\nu\in\Gamma_{k,n}}(c_{\lambda,\mu}^\nu)^{K_{\widehat{T}}}\cdot  [S_\nu],\quad (c_{\lambda,\mu}^\nu)^{K_{\widehat{T}}}\in \mrm{Frac}({K}_{\widehat{T}}(\mrm{point}))
	\end{equation}
	We highlight this deep connection:
	\[\{\text{Puzzle formula in $\widehat{K}_{\widehat{T}}^{\text{loc}}(T^*(\mrm{Gr}(k,n))$}\}\leftrightarrow \{\text{Rep. theory of the quantum affine algebra of type $\hat{\mathfrak{a}}_2$}\}\]
	This leads to a natural question: is there a quantum group attached to the proof of Theorem \ref{thm:main}? Indeed, the proof of the puzzle formula of Theorem \ref{thm:main} relies heavily (and unexpectedly) on the representation theory of the \textit{multi-parameter quantum group} of type $\widehat{\mathfrak{a}}_2$ (see \cite[Definition 7]{PHR10}):
	\begin{theo}(\cref{theo:R-representations})\label{prop:R-matrix-intro}
		The weights of the puzzle pieces that define the puzzle rule in Theorem \ref{thm:main} can be arranged into an $R$-matrix for the multi-parameter quantum group of type $\widehat{\mathfrak{a}}_2$ of \cite[Definition 7]{PHR10}.
	\end{theo}
	
	\cref{prop:R-matrix-intro} suggests a deep relationship between multi-parameter quantum groups and equivariant connective $K$-rings of Nakajima quiver varieties. The author is not aware of any action of these quantum groups on connective $K$-rings in the literature, and plans to explore this in future work (see \cref{problem:convolution-structure}).
	
	\subsection{Organization}\label{subsection:organization} In \cref{section:preliminaries}, we establish notation that will be used throughout this paper, and we recall and summarize various notions such as $d$-step flag varieties and their cotangent bundles, Chern-Schwartz-MacPherson classes, motivic Chern classes, divided difference operators, equivariant localization, and Knutson-Tao puzzles. In \cref{section:deformation}, we define the $\beta$-deformations of the motivic Segre classes of Schubert cells that will be our main object of study, \textit{in the special case that $\beta$ is nonzero}. (The uniform construction of these classes for all $\beta$ will be handled in \cref{section:interpolating}). In \cref{section:solvable-lattice-model}, we define a solvable lattice model that generates rational functions which represent our $\beta$-deformed motivic Segre classes in the special case when $d=1$. When $\beta=1$, our lattice model specializes to the one considered in \cite{KZJ17,KZJ21}. In \cref{section:puzzle-formula}, we prove a \textit{positive} puzzle formula for the structure constants in the basis of our $\beta$-deformed classes when $d=1$, using the technique of \cite{KZJ17,KZJ21}. Our proof involves $R$, $U$, and $D$ matrices involving a parameter $\beta$ that specialize to the matrices considered in \cite{KZJ17,KZJ21} when $\beta=1$. In \cref{section:multi-parameter}, we show that our $\beta$-deformed $R$, $U$, and $D$ matrices are intertwiners for evaluation representations of the multi-parameter quantum group of type $\widehat{\mathfrak{a}}_2$. In \cref{section:interpolating}, we recall the formal group algebra of \cite{CPZ}, and we repeat the construction of the $\beta$-deformed motivic Segre classes of \cref{section:deformation} in the full generality, allowing $\beta$ to be zero or non-zero. We show that the $\beta=0$ specialization of our $\beta$-deformed classes recovers the Segre-Schwartz-MacPherson classes of Schubert cells and that the specialization of $\beta$ to any nonzero element recovers the classes defined in \cref{section:deformation}. In \cref{section:localization-techniques}, we prove that our $\beta$-deformed motivic Segre classes arise from canonical elements in a quotient of the equivariant algebraic cobordism ring of the cotangent bundle of the $d$-step flag variety. In \cref{section:future-work}, we describe future directions to take this work and outline strategies to accomplish this future work. In Appendix \ref{section:YBE}, we list the equations that were verified in SageMath \cite{G25} in the course of the proof of our puzzle formula, and in Appendix \ref{section:quantum-group}, we recall the definition of the quantum affine algebra of type $\widehat{a}_n$.

	\subsection*{Acknowledgements}
	The author is grateful to his Ph.D. advisor Allen Knutson for proposing the idea that eventually led to this paper and for his attention, guidance, and insights throughout the course of this project. We thank Kirill Zainoulline for introducing the author to formal group laws and connective $K$-theory; we thank Rui Xiong for an enlightening conversation regarding Schubert calculus in connective $K$-theory; we thank Timothy Miller and Travis Scrimshaw for explaining that the lattice model in this paper should simultaneously compute rational function representatives for the classes \textit{and} structure constants; we thank Marcelo Aguiar for pointing out that the Hopf algebra obtained in \cref{section:multi-parameter} is a multi-parameter quantum group; we thank Changlong Zhong for explaining that the analogues of the motivic Chern classes defined via an arbitrary formal group law satisfy the GKM conditions; and we thank Paul Zinn--Justin for help with debugging the code used to verify several computations in this paper.
	This work was partially supported by a Postgraduate Scholarship - Doctoral - from the Natural Sciences and Engineering Research Council of Canada, and by Cornell University.
	
	\section{Preliminaries and notation}\label{section:preliminaries}
	
	In this section, we establish notation and recall several concepts that will be used throughout this paper. In \cref{subsection:flag}, we discuss flag varieties and Schubert varieties. In \cref{subsection:Chern-Schwartz-MacPherson-classes}, we recall the definition of equivariant Chern-Schwartz-MacPherson classes. In \cref{subsection:motivic-Chern-classes}, we recall the definition of equivariant motivic Chern classes. In \cref{subsection:cotangent-bundles-of-flag}, we discuss cotangent bundles of flag varieties and the relationships between Chern-Schwartz-MacPherson classes of Schubert cells (resp. motivic Chern classes of Schubert cells) and cohomological stable envelopes (resp. $K$-theoretic stable envelopes) in the cotangent bundle of the flag variety. In \cref{subsection:localization}, we discuss divided difference operators and the combinatorial construction of Segre-Schwartz-MacPherson classes and motivic Segre classes of Schubert cells. In \cref{subsection:puzzles}, we discuss Littlewood-Richardson coefficients and Knutson-Tao puzzles.
	
	\subsection{Flag varieties}\label{subsection:flag} In this subsection, we will discuss flag varieties, Schubert varieties, and Schubert classes in the (equivariant) cohomology ring and (equivariant) $K$-ring of the flag variety. 
	
	Let $G=\mrm{GL}_n(\mathbb{C})$. Fix the Borel subgroup $B$ in $G$ consisting of upper-triangular matrices in $G$, and fix the maximal torus $T$ of $B$ consisting of diagonal matrices in $G$. The weight lattice of $T$ is $\Lambda=\mrm{Hom}(T,\mathbb{C}^\times)$, and the root system of $T$ is $\Sigma=A_{n-1}$. We denote the set of simple roots of $\Sigma$ by $\Delta=\{\alpha_1,\dotsc,\alpha_{n-1}\}$. For a subset $\Theta=\{\alpha_{i_1},\dotsc,\alpha_{i_d}\}$, $i_1<i_2<\dotsb<i_d$, of the simple roots $\Delta$, we denote by $P_\Theta$ a parabolic subgroup of $G$ that contains $B$ and has (negative) simple roots $\Delta\setminus \Theta$. The $d$-step flag variety $G/P_\Theta$ is the variety of (partial) flags of vector spaces,
	\[0\subsetneq V_1\subsetneq V_2\subsetneq \dotsb\subsetneq V_d\subsetneq \mbb{C}^n, \quad \mrm{dim}(V_j)= i_j.\]
	In the special case $\Theta=\emptyset$, the flag variety $G/P_\Theta$ is the complete flag variety $G/B$. In the special case $\Theta=\Delta\setminus\{\alpha_k\}$, the group $P_\Theta$ is a maximal parabolic subgroup of $G$, and $G/P_\Theta$ is the Grassmannian $\mrm{Gr}(k,n)$ of $k$-planes in $\mathbb{C}^n$. Let $\Sigma_\Theta^+$ (resp. $\Sigma_\Theta^-$) be the set of positive (resp. negative) roots of $G/P_\Theta$, and let $\Sigma_\Theta:=\Sigma_\Theta^+\sqcup \Sigma_\Theta^-$ be the root system of $G/P_\Theta$. We will denote the Weyl group $S_n$ of $G$ by $W$. Let $r_i$ be the simple permutation in $W$ that swaps $i$ and $i+1$ and fixes all other $j$. The subgroup $W_{\Theta}$ of $W$ generated by $\langle r_j \rangle_{\alpha_j\in\Theta}$ is the Weyl group of $P_\Theta$. Let $W^\Theta$ be the set of minimal coset representatives of $W/W_\Theta$ in $W$. We will henceforth identify the set $W^\Theta$ with the quotient $W/W_\Theta$. Define $p_j:=i_j-i_{j-1}$. The set $W^\Theta$ is in bijection with the set $\Gamma$ of lists of length $n$ in the alphabet $\{0,\dotsc,d\}$, where $d$ appears $p_1$ times, $d-1$ appears $p_{2}$ times, and so on. Under this identification, the longest word $w_0$ in $W^\Theta$ is identified with the list $d^{p_1} (d-1)^{p_{2}} \dotsb 1^{p_d} 0^{p_{d+1}}$, and $w\cdot w_0\in W^\Theta$ is identified with the list $w(d^{p_1} (d-1)^{p_{2}} \dotsb 1^{p_d} 0^{p_{d+1}})$.
	
	For each $w\in W^\Theta$, there is a Schubert cell $X_w^\circ=B_-wP_\Theta/P_\Theta$ in $G/P_\Theta$, where $B_-$ is the Borel subgroup of $G$ consisting of lower-triangular matrices in $G$. The Schubert cells form a cell decomposition for $G/P_\Theta$. The closures of the Schubert cells $X_w:=\overline{X_w^\circ}$ are Schubert varieties. Let $\mrm{h}^*(-)$ be one of the following (generalized) cohomology theories: singular cohomology $H^*(-)$, $T$-equivariant cohomology $H_T(-)$, $K$-theory $K(-)$, and $T$-equivariant $K$-theory $K_T(-)$. Let $\mrm{pt}$ be a point, so that $H^*(\mrm{pt})\simeq \mbb{Z}$, $H_T(\mrm{pt})\simeq \mbb{Z}[y_1,\dotsc,y_n]$, $K(\mrm{pt})\simeq \mbb{Z}$, and $K_T(\mrm{pt})\simeq \mbb{Z}[e^{\pm y_1},\dotsc,e^{\pm y_n}]$. In $\mrm{h}^*(G/P_\Theta)$ there are natural classes $S_w$, $w\in W^\Theta$, known as Schubert classes, that form a basis for $\mrm{h}^*(G/P_\Theta)$, when we view $\mrm{h}^*(G/P_\Theta)$ as a module over $\mrm{h}^*(\mrm{pt})$. In $K(G/P_\Theta)$ and $K_T(G/P_\Theta)$, the Schubert class $S_w$ is the class of the structure sheaf $[\mcal{O}_{X_w}]$ of $X_w$. In $H^*(G/P_\Theta)$ and $H_T(G/P_\Theta)$, the Schubert class $S_w$ is the class $[X_w]$ defined by the ($T$-invariant) Schubert variety $X_w$.

	\subsection{Chern-Schwartz-MacPherson classes}\label{subsection:Chern-Schwartz-MacPherson-classes} 
	In this subsection, we will recall the axiomatic definition of the (equivariant) Chern-Schwartz-MacPherson class (see \cref{dfn:CSM-class}).
	
	Let $A$ be a reductive complex linear algebraic group. Let $\mcal{F}^A$ be the (covariant) functor from the category $\mcal{C}^A$ whose objects are ``$A$-varieties", i.e., complex algebraic varieties with an algebraic action of $A$, and whose morphisms are $A$-equivariant \textit{proper} morphisms of algebraic varieties, to the category $\mrm{Ab}$ whose objects are abelian groups, such that for any $A$-variety $Z$, the abelian group $\mcal{F}^A(Z)$ consists of $A$-equivariant constructible functions on $Z$ under pointwise addition. Recall that a function on $Z$ is called \textit{constructible} if it is a finite $\mbb{C}$-linear combination of characteristic functions $1_Y$ of closed subvarieties $Y$ in $Z$. If $Z_1$ and $Z_2$ are $A$-varieties and $f\colon Z_1\to Z_2$ is an $A$-equivariant proper morphism of varieties, then the functorial \textit{push-forward} map $f_*\colon \mcal{F}^A(Z_1)\to \mcal{F}^A(Z_2)$ is defined on generators by $(f_*(1_Y))(z):=\chi(f^{-1}(z)\cap Y)$ for any $A$-invariant locally closed subvariety $Y$ of $Z_1$, and extended by $\mbb{C}$-linearity to all of $\mcal{F}^A(Z_1)$. Let $H_*^A\colon \mcal{C}^A\to \mrm{Ab}$ be the $A$-equivariant singular homology functor that sends an $A$-variety $Z$ to the $A$-equivariant singular homology group $H_*^A(Z)$. If $A=\{1\}$ is the trivial group, then the category $\mcal{C}:=\mcal{C}^A$ is the category whose objects consist of \textit{all} complex algebraic varieties and whose morphisms consist of \textit{all} proper morphisms of varieties. In this case, the abelian group $\mcal{F}(Z):=\mcal{F}^A(Z)$ is the group of \textit{all} constructible functions on $Z$, and the group $H_*(Z)=H_*^A(Z)$ is the ordinary singular homology group of $Z$. 
	
	First let us assume $A=\{1\}$, so that $\mcal{F}=\mcal{F}^A$, $\mcal{C}=\mcal{C}^A$, and $H_*=H_*^A$. In \cite{M74}, Robert MacPherson proved the existence of a unique natural transformation $C_*\colon \mcal{F}\to H_{*}$, such that if $Z$ is smooth, then $C_*({1}_Z)=c(T_Z)\cap [Z]$, where $c(T_Z)\in H_*(Z)$ is the total Chern class of the tangent bundle $T_Z$ of $Z$ and $[Z]\in H_*(Z)$ is the fundamental class $Z$. The existence of $C_*$ was conjectured earlier by Grothendieck and Deligne (see \cite{S69}), and the work of MacPherson settled this conjecture. 
	
	Suppose now that $A$ is \textit{any} complex reductive linear algebraic group. An $A$-\textit{equivariant} version $C_*^A$ of $C_*$ was constructed by Ohmoto in \cite{Oh06}. That is, Ohmoto proves in \cite[Theorem 1.1]{Oh06} that there is a unique natural transformation $C_*^A\colon \mcal{F}^A\to H^A_*$, such that if $Z$ is smooth, then $C_*^A({1}_Z)=c^A(T_Z)\cap [Z]$, where $c^A(T_Z)\in H_*^A(Z)$ is the $A$-equivariant total Chern class of the tangent bundle $T_Z$ of $Z$ and $[Z]\in H_*^A(Z)$ is the $A$-equivariant fundamental class $Z$.

	\begin{dfn}\label{dfn:CSM-class}
		For a complex algebraic variety $Z$ with an algebraic action of $A$ and an $A$-invariant constructible subset $Y$ of $Z$, the class $\mrm{csm}^{A}(Y):=C_*^A(1_Y)\in H_*(Z)$ is called the \textbf{$A$-equivariant Chern-Schwartz-MacPherson (CSM) class} of $Y$ in $H_*^A(Z)$. The \textbf{Segre-Schwartz-MacPherson (SSM) class} of $Y$ in the fraction field $\mrm{Frac}(H_*^A(Z))$ of $H_*^A(Z)$ is the quotient $\mrm{ssm}^A(Y):=\frac{\mrm{csm}^A(Y)}{c^A(T_Z)}$. 
	\end{dfn}
	\begin{rmk}
		In the notation of \cref{dfn:CSM-class}, when $A=\{1\}$, the class $\mrm{csm}(Y):=\mrm{csm}^A(Y)\in  H_*(Z)$ is the ordinary (non-equivariant) \textbf{Chern-Schwartz-MacPherson} class of $Y$ in $H_*(Z)$, and $\mrm{ssm}(Y):=\frac{\mrm{csm}(Y)}{c(T_Z)}$ is the ordinary (non-equivariant) \textbf{Segre-Schwartz-MacPherson} class of $Y$ in the fraction field $\mrm{Frac}(H_*(Z))$ of $H_*(Z)$.
	\end{rmk}

	\subsection{Motivic Chern classes}\label{subsection:motivic-Chern-classes}
	
	In this subsection, we will recall the axiomatic definition of the (equivariant) motivic Chern class (see \cref{dfn:mC-class}).
	
	Let $A$ be a reductive complex linear algebraic group. Consider the category $\mcal{F}^A$ whose objects are \textit{smooth} $A$-varieties, and whose morphisms are $A$-equivariant \textit{proper} morphisms of algebraic varieties, and the category $\mrm{Ab}$ whose objects are abelian groups and whose morphisms are group homomorphisms. Let $\mrm{Var}^A(-)\colon \mcal{F}^A\to\mrm{Ab}$ be the (covariant) functor that sends a smooth $A$-variety $Z$ to the quotient of the free abelian group generated by isomorphism classes $[Y\to Z]$ of $A$-equivariant morphisms of $A$-varieties $Y\to Z$, modulo the subgroup generated by the ``scissor relations":
	\begin{equation}\label{equation:scissor}
		[Y\to Z]-([Y'\to Y\to Z]+[Y\setminus Y'\to Y\to Z]),
	\end{equation}
	where $Y'$ is a closed $A$-invariant algebraic subvariety of $Y$ and $Y'\to Y$ and $Y\setminus Y'\to Y$ are both $A$-equivariant morphisms of algebraic varieties. Let $K(-)[y]\colon \mcal{F}^A\to \mrm{Ab}$ be the (covariant) functor that sends an $A$-variety $Z$ to the group $K_0^A(Z)\otimes \mbb{Z}[y]$, where $K_0^A(Z)$ is the Grothendieck group of $A$-equivariant locally free coherent sheaves on $Z$. If $A=\{1\}$ is the trivial group, then the category $\mcal{F}:=\mcal{F}^A$ is the category whose objects consist of \textit{all} smooth complex algebraic varieties and whose morphisms consist of \textit{all} proper morphisms between smooth varieties. In this case, the abelian group $\mrm{Var}(Z):=\mrm{Var}^A(Z)$ is the quotient of the free abelian group generated by isomorphism classes $[Y\to Z]$, where  $Y\to Z$ is \textit{any} morphism of algebraic varieties, modulo the subgroup generated by the scissor relations (\ref{equation:scissor}). Additionally, the group $K_0^A(Z)\otimes \mbb{Z}[y]$ is just $K_0(Z)\otimes \mbb{Z}[y]$, where $K_0(Z)$ is the ordinary $K$-group of locally free coherent sheaves on $Z$.
	
	First let us assume $A=\{1\}$ is the trivial group, so that $\mcal{F}=\mcal{F}^A$, $\mrm{Var}^A(-)=\mrm{Var}(-)$, and $K_0^A(-)=K_0(-)$. In \cite{BSY10}, the authors proved the existence of a unique natural transformation $mC\colon \mrm{Var}(-)\to K_0(-)\otimes\mbb{Z}[y]$ satisfying
	(i) $mC([\mrm{id}_Z])=\sum_{i=0}^{\mrm{rank}(T_Z)}[\Lambda^i T_Z]y^i$, where $\Lambda^i T_Z$ is the $i$-th exterior power of the tangent bundle $T_Z$ of $Z$, and (ii) for any proper map $f\colon Z\to Z'$, we have $\mrm{mC}([f\circ g\colon Y\to Z'])=f_*(\mrm{mC}([g\colon Y\to Z]))$.
	
	Suppose now that $A$ is \textit{any} complex reductive linear algebraic group. In \cite{FRW21}, the authors proved the existence of a unique natural transformation $mC^A\colon \mrm{Var}^A(-)\to K_0^A(-)\otimes\mbb{Z}[y]$ satisfying (i) $mC^A([\mrm{id}_Z])=\sum_{i=0}^{\mrm{rank}(T_Z)}[\Lambda^i T_Z]y^i$, where $\Lambda^i T_Z$ is the $i$-th exterior power of the tangent bundle $T_Z$ of $Z$, and (ii) for any $A$-equivariant proper map $f\colon Z\to Z'$, we have $\mrm{mC}^A([f\circ g\colon Y\to Z'])=f_*(\mrm{mC}^A([g\colon Y\to Z]))$. 
	
	\begin{dfn}\label{dfn:mC-class}
		For a smooth $A$-variety $Z$, together with an $A$-equivariant morphism of varieties $Y\to Z$, the class $\mrm{mC}^{A}(Y):=mC^A([Y\to Z])\in K_0^A(Z)\otimes\mbb{Z}[y]$ is called the \textbf{$A$-equivariant motivic Chern (mC) class} of $Y$ in $K_0^A(Z)\otimes\mbb{Z}[y]$. The \textbf{motivic Segre (mS) class} of $Y$ in the fraction field $\mrm{Frac}(K_0^A(Z)\otimes\mbb{Z}[y])$ of $K_0^A(Z)\otimes\mbb{Z}[y]$ is the quotient $\mrm{mS}^A(Y):=\frac{\mrm{mC}^A(Y)}{\lambda_y(T_Z)}$, where $\lambda_y(Z):=\sum_{i=0}^{\mrm{rank}(T_Z)}[\Lambda^i T_Z]y^i$ is called the \textbf{$\lambda_y$-genus} of $Z$.
	\end{dfn}
	
	\begin{rmk}
		In the notation of \cref{dfn:mC-class}, when $A=\{1\}$, the class $\mrm{mC}(Y):=\mrm{mC}^A(Y)$ is the ordinary (non-equivariant) \textbf{motivic Chern class} of $Y$ in $K_0(Z)\otimes\mbb{Z}[y]$, and $\mrm{mS}(Y):=\frac{\mrm{mC}(Y)}{\lambda_y(T_Z)}$ is the ordinary (non-equivariant) \textbf{motivic Segre class} of $Y$ in the fraction field $\mrm{Frac}(K_0(Z)\otimes\mbb{Z}[y])$ of $K_0(Z)\otimes\mbb{Z}[y]$.
	\end{rmk}

	\subsection{Cotangent bundles of flag varieties}\label{subsection:cotangent-bundles-of-flag}
	
	In this subsection, we will recall several facts regarding the cotangent bundles of flag varieties. We will also state the relationship between cohomological (resp. $K$-theoretic) stable envelopes for cotangent bundles of flag varieties and Chern-Schwartz-MacPherson classes (resp. motivic Chern classes) of Schubert cells (see \cref{theo:CSM-classes} and \cref{theo:motivic-Chern-classes}).
	
	Let $\widehat{T}:=T\times\mbb{C}^\times$. The torus $\widehat{T}$ acts on $G/P_\Theta$, where $T$ acts on $G/P_\Theta$ in the usual way and the additional $\mbb{C}^\times$-factor acts trivially. The torus $\widehat{T}$ also acts on the cotangent bundle $T^*(G/P_\Theta)$, where $(t,z)\cdot (x,y):=(t\cdot x,(t^{-1})^*(z\cdot y))$ for all $(t,z)\in \widehat{T}$ and $(x,y)\in T^*_x(G/P_\Theta)$. The inclusion of the zero section $\iota\colon G/P_\Theta\hookrightarrow T^*(G/P_\Theta)$ is $\widehat{T}$-equivariant, and as $\iota(G/P_\Theta)$ is $\widehat{T}$-equivariantly homotopic to $T^*(G/P_\Theta)$, it follows that the pullback map $\iota^*\colon H_{\widehat{T}}(T^*(G/P_\Theta))\to H_{\widehat{T}}(G/P_\Theta)$ is an isomorphism of $H_{\widehat{T}}(\mrm{pt})$-algebras, and the pullback map $\iota^*\colon K_{\widehat{T}}(T^*(G/P_\Theta))\to K_{\widehat{T}}(G/P_\Theta)$ is an isomorphism of $K_{\widehat{T}}(\mrm{pt})$-algebras. Write $H_{\widehat{T}}(\mrm{pt})\simeq H_T(\mrm{pt})\otimes_{\mbb{Z}}\mbb{Z}[\hbar]$ and $K_{\widehat{T}}(\mrm{pt})\simeq K_T(\mrm{pt})\otimes_{\mbb{Z}}\mbb{Z}[q^{\pm}]$, where $\hbar$ and $q$ are used to denote the equivariant parameter coming from the dliation action of $\widehat{T}$ on the cotangent fibres of $T^*(G/P_\Theta)$.

	For $w\in W^\Theta$, recall the notation $X_w^\circ$ for the Schubert cell inside the $d$-step flag variety $G/P_\Theta$. As $X_w^\circ$ is a locally closed $\widehat{T}$-invariant subvariety of $G/P_\Theta$, one can consider its $\widehat{T}$-equivariant Chern-Schwartz-MacPherson class $\mrm{csm}^{\widehat{T}}(X_w^\circ)$ inside of $H_{\widehat{T}}^*(G/P_\Theta)$. We remark that, although we have defined $\mrm{csm}^{\widehat{T}}(X_w^\circ)$ to be an element of the $\widehat{T}$-equivariant singular \textit{homology} ring of $G/P_\Theta$, as $G/P_\Theta$ is a smooth, projective variety, we simply identify $\mrm{csm}^{\widehat{T}}(X_w^\circ)$ with its Poincar\'e dual in $H_{\widehat{T}}^*(G/P_\Theta)$ without losing any information. In \cite{MO19}, Davesh Maulik and Andrei Okounkov defined a natural basis $\{\mrm{stab}^H_-(w)\}_{w\in W^\Theta}$ inside of $H_{\widehat{T}}(T^*(G/P_\Theta))$ corresponding to the \textit{negative Weyl chamber} for $W$, i.e., corresponding to the subset of $\Sigma_\Theta$ consisting of \textit{negative} roots $\Sigma_\Theta^-$ (in fact, there is a stable basis for each Weyl chamber in $W$, though we will not consider the other Weyl chambers in this paper), and the elements of this basis correspond to $\widehat{T}$-invariant cycles in $H_{\widehat{T}}(T^*(G/P_\Theta))$. 
	
	\begin{theo}(\cite[Prop. 9.7(b) and Prop. 9.9]{AMSS23}, \cite[\S 8]{RWV18})\label{theo:CSM-classes}
		There is an identity \[\iota^*(\mrm{stab}^H_-(w))=(-1)^{\mrm{dim}(G/P_\Theta)}\mrm{csm}^{\widehat{T}}(X_w^\circ).\]
	\end{theo}
	
	In \cite{O17}, Maulik and Okounkov defined natural bases inside of $K_{\widehat{T}}(T^*(G/P_\Theta))$, which are analogues of the stables bases in $H_{\widehat{T}}(T^*(G/P_\Theta))$ and depend on more data than the cohomological stable bases. Although we will not go into detail on this construction, we mention the following result:
	
	\begin{theo}(\cite[Prop. 2.9, and the Remark above \S 3]{KZJ21}, \cite[Theorem 3.5]{SZZ20}, \cite[Remark 5.6]{FRW21})\label{theo:motivic-Chern-classes}
		Let $\{\mrm{stab}_{-}^K(w)\}_{w\in W^\Theta}$ be the stable basis in $K_{\widehat{T}}(T^*(G/P_\Theta))$ corresponding to the polarization $T^*(G/P_\Theta)$, and to a line bundle in the negative Weyl alcove closest to the origin. There is an identity 
		\[\iota^*(\mrm{stab}_-^K(w))=\mrm{mC}^{\widehat{T}}(X_w^\circ).\]
	\end{theo}
	
	In light of \cref{theo:CSM-classes} and \cref{theo:motivic-Chern-classes}, we will denote the class $\mrm{csm}^{\widehat{T}}(X_w^\circ)$ in $H_{\widehat{T}}(T^*(G/P_\Theta))$ and the class $\mrm{mC}^{\widehat{T}}(X_w^\circ)$ in $K_{\widehat{T}}(T^*(G/P_\Theta))$ uniformly by $\mrm{St}_w$. Consider the localizations $H_{\widehat{T}}^{\mrm{loc}}(\mrm{pt})$ and $K_{\widehat{T}}^{\mrm{loc}}(\mrm{pt})$ of $H_{\widehat{T}}(\mrm{pt})$ and $K_{\widehat{T}}(\mrm{pt})$ at the elements $\{\hbar+\alpha\}_{\alpha\in\Sigma}$ and $\{1-q^2e^{\alpha}\}_{\alpha\in\Sigma}$, respectively. Set $H_{\widehat{T}}^{\mrm{loc}}(T^*(G/P_\Theta)):=H_{\widehat{T}}^{\mrm{loc}}(\mrm{pt})\otimes_{H_{\widehat{T}}(\mrm{pt})} H_{\widehat{T}}(T^*(G/P_\Theta))$ and $K_{\widehat{T}}^{\mrm{loc}}(T^*(G/P_\Theta)):=K_{\widehat{T}}^{\mrm{loc}}(\mrm{pt})\otimes_{K_{\widehat{T}}(\mrm{pt})} K_{\widehat{T}}(T^*(G/P_\Theta))$. Let $\kappa$ be the class of the zero section of $T^*(G/P_\Theta)\to G/P_\Theta$ in either $H_{\widehat{T}}(T^*(G/P_\Theta))$ or $K_{\widehat{T}}(T^*(G/P_\Theta))$. The elements $S_\lambda=\tfrac{\mrm{csm}^{\widehat{T}}(X_w^\circ)}{\kappa}$ and $S_\lambda=\tfrac{\mrm{mC}^{\widehat{T}}(X_w^\circ)}{\kappa}$ live in the localization $H_{\widehat{T}}^{\mrm{loc}}(T^*(G/P_\Theta))$ or the localization $K_{\widehat{T}}^{\mrm{loc}}(T^*(G/P_\Theta))$, and $\left\{S_\lambda\right\}_{\lambda\in W^\Theta}$ forms a basis for $H_{\widehat{T}}^{\mrm{loc}}(T^*(G/P_\Theta))$ or $K_{\widehat{T}}^{\mrm{loc}}(T^*(G/P_\Theta))$ over $H_{\widehat{T}}^{\text{loc}}(\mrm{pt})$ or $K_{\widehat{T}}^{\text{loc}}(\mrm{pt})$, respectively. As explained in \cite[\S 2 and \S 5]{KZJ21}, the elements $S_\lambda$ in $H_{\widehat{T}}^{\mrm{loc}}(T^*(G/P_\Theta))$ agree with the \textit{Segre-Schwartz-MacPherson (SSM) classes} of the Schubert cells $X_\lambda^\circ$, and the elements $S_\lambda$ in $K_{\widehat{T}}^{\mrm{loc}}(T^*(G/P_\Theta))$ agree with the \textit{motivic Segre classes} of the Schubert cells $X_\lambda^\circ$. The class $\kappa$ is discussed further in \cref{subsection:second-construction}.
	
	\subsection{Equivariant localization and divided difference operators}\label{subsection:localization}
	
	In this subsection, we will recall divided difference operators, and we will construct the Schubert classes, Segre-Schwartz-MacPherson classes of Schubert cells, and motivic Segre classes of Schubert cells using these operators.

	Let $X$ be a smooth projective complex variety, and suppose $E$ is a complex torus acting on $X$ algebraically with finitely many fixed points $X^E$. For both $h_E(-)=H_E(-)$ and $h_E(-)=K_E(-)$, the localization map
	\begin{equation}\label{equation:injective-fixed-point}
		\iota\colon h_E(X)\to \bigoplus_{f\in X^E}h_E(\mrm{pt}),
	\end{equation}
	defined by restricting $h_E(X)$ to the corresponding fixed point $f$ on each component, is injective.
	We will consider the case $X=G/P_\Theta$ and $E=T$, or the case $X=G/P_\Theta$ and $E=\widehat{T}=T\times \mbb{C}^\times$, where in the latter case the additional $\mbb{C}^\times$-factor acts trivially on $G/P_\Theta$. Thus, in both cases, the $E$-fixed points of $X$ are indexed by $W^\Theta$. The inclusion of the zero section $G/P_\Theta\hookrightarrow T^*(G/P_\Theta)$ induces a pullback map $\mrm{h}_{\widehat{T}}(T^*(G/P_\Theta))\to \mrm{h}_{\widehat{T}}(G/P_\Theta)$, which is a ring isomorphism, and the localization map $\iota\colon h_{\widehat{T}}(T^*(G/P_\Theta))\to \bigoplus_{f\in X^{\widehat{T}}}h_{\widehat{T}}(\mrm{pt})$ agrees with the injective localization map $\iota\colon h_{\widehat{T}}(G/P_\Theta)\hookrightarrow \bigoplus_{f\in X^{\widehat{T}}}h_{\widehat{T}}(\mrm{pt})$ under this isomorphism. We will henceforth identify the ring homomorphisms $\iota\colon h_{\widehat{T}}(T^*(G/P_\Theta))\hookrightarrow \bigoplus_{f\in X^{\widehat{T}}}h_{\widehat{T}}(\mrm{pt})$ and $\iota\colon h_{\widehat{T}}(G/P_\Theta)\hookrightarrow \bigoplus_{f\in X^{\widehat{T}}}h_{\widehat{T}}(\mrm{pt})$.

	We will now describe how to obtain the images of the classes $S_\lambda$ in the rings $H_T(G/P_\Theta)$, $H_{\widehat{T}}^{\text{loc}}(T^*(G/P_\Theta))$, $K_T(G/P_\Theta)$, and $K_{\widehat{T}}^{\text{loc}}(T^*(G/P_\Theta))$ under the localization map. For each $\alpha_i\in\Delta$, there is a $\mathbb{Z}$-linear operator $\delta_i$ that acts on $\bigoplus_{w\in W^\Theta}h_E(\mrm{pt})$, which we will define explicitly later in this subsection. Any $w\in W^\Theta$ can be expressed as the reduced product of simple transpositions $w=r_{i_1}\dotsb r_{i_k}$. The operator $\delta_w:=\delta_{i_1}\circ\dotsb\circ \delta_{i_k}$ also acts on $\bigoplus_{w\in W^\Theta}h_E(\mrm{pt})$ and is independent of the choice of reduced expression for $w$.  We have $\iota(S_{w\cdot w_0})=\delta_w(\iota(S_{w_0}))$, where $w_0$ is the longest word in $W^\Theta$. We will now explicitly describe the restrictions $S_{w_0}|_\lambda$ of $S_{w_0}$ to each fixed point $\lambda\in W^\Theta$, and we will give an explicit formula defining the operator $\delta_i$ in each of the four cases $H_T(G/P_\Theta)$, $H_{\widehat{T}}^{\text{loc}}(T^*(G/P_\Theta))$, $K_T(G/P_\Theta)$, and $K_{\widehat{T}}^{\text{loc}}(T^*(G/P_\Theta))$. We can view the root $\alpha_i$ as the element $\alpha_i=y_i-y_{i+1}$ in the weight lattice $\Lambda\simeq \mbb{Z}[y_1,\dotsc,y_n]$ 
	
	\begin{enumerate}
		\item (\cite{D73}) $H_T(G/P_\Theta)$. Below are the formulas for $S_{w_0}$ restricted to fixed points $\lambda\in W^\Theta$:	
		\[S_{w_0}|_\lambda=\begin{cases}
			\prod_{\alpha\in \Sigma_\Theta^+}\alpha,& \lambda=w_0;\\
			0,& \lambda\neq w_0.
		\end{cases}\]		
		There is a natural action of $W^\Theta$ on the ring $H_T(\mrm{pt})\simeq \mbb{Z}[y_1,\dotsc,y_n]$, defined on generators by $w(y_i):=y_{w(i)}$ for all $w\in W^\Theta$. This $W^\Theta$-action extends to the ring $\bigoplus_{\lambda\in W^\Theta}H_T(\mrm{pt})$ by $w(f_\lambda)_{\lambda}:=(w(f_\lambda))_{w(\lambda)}$ for all $w\in W^\Theta$. 
		Consider the $\mbb{Z}$-linear operators on $H_T(\mrm{pt})$:
		\[\partial_{i}(f)=\frac{f-r_i(f)}{\alpha_i},\quad f\in H_T(\mrm{pt}),\text{ } \alpha_i\in\Delta.\]
		The operators $\delta_i$ naturally define $\mbb{Z}$-linear operators on $\bigoplus_{\lambda\in W^\Theta}H_T(\mrm{pt})$, and they are often called \textbf{divided difference operators} for $H_T(G/P_\Theta)$. 
		\item (\cite[\S 5]{KZJ21}, \cite{AM16}) $H_{\widehat{T}}^{\text{loc}}(T^*(G/P_\Theta))$. Below are the formulas for $S_{w_0}$ restricted to fixed points $\lambda\in W^\Theta$:	
		\[S_{w_0}|_\lambda=\begin{cases}
			\prod_{\alpha\in \Sigma_\Theta^+} \tfrac{\alpha}{\alpha+\hbar},& \lambda=w_0;\\
			0,& \lambda\neq w_0.
		\end{cases}\]		
		There is a natural action of $W^\Theta$ on the ring $H_{\widehat{T}}^{\text{loc}}(\mrm{pt})\simeq \mbb{Z}[y_1,\dotsc,y_n,\hbar]$, defined on generators by $w(y_i):=y_{w(i)}$ and $w(\hbar)=\hbar$ for all $w\in W^\Theta$. This $W^\Theta$-action extends  to the ring $\bigoplus_{\lambda\in W^\Theta}H_{\widehat{T}}^{\text{loc}}(\mrm{pt})$ by $w(f_\lambda)_{\lambda}:=(w(f_\lambda))_{w(\lambda)}$ for all $w\in W^\Theta$. 
		Consider the $\mbb{Z}[\hbar]$-linear operators on $H_{\widehat{T}}^{\text{loc}}(\mrm{pt})$:
		\[\delta_{i}(f)=\frac{\hbar}{\alpha_i}f+\frac{\alpha_i-\hbar}{\alpha_i}r_i(f),\quad f\in H_{\widehat{T}}(\mrm{pt}),\text{ } \alpha_i\in\Delta.\]
		The operators $\delta_i$ naturally define $\mbb{Z}[\hbar]$-linear operators on $\bigoplus_{\lambda\in W^\Theta}H_{\widehat{T}}^{\text{loc}}(\mrm{pt})$, and we will call them \textbf{Demazure-Lusztig operators} for $H_{\widehat{T}}(T^*(G/P_\Theta))$.
		\item\label{itm:3} (\cite{D74}) $K_T(G/P_\Theta)$. Below are the formulas for $S_{w_0}$ restricted to fixed points $\lambda\in W^\Theta$:	
		\[S_{w_0}|_\lambda=\begin{cases}
			\prod_{\alpha\in\Sigma_\Theta^+} (1-e^{-\alpha}),& \lambda=w_0;\\
			0,&\lambda\neq w_0.
		\end{cases}\]		
		There is a natural action of $W^\Theta$ on the ring $K_T(\mrm{pt})\simeq \mbb{Z}[e^{\pm y_1},\dotsc,e^{\pm y_n}]$, defined on generators by $w(e^{\pm y_i}):=e^{\pm y_{w(i)}}$ for all $w\in W^\Theta$. This $W^\Theta$-action extends to the ring $\bigoplus_{\lambda\in W^\Theta}K_T(\mrm{pt})$ by $w(f_\lambda)_{\lambda}:=(w(f_\lambda))_{w(\lambda)}$ for all $w\in W^\Theta$. 
		Consider the $\mbb{Z}$-linear operators on $K_T(\mrm{pt})$:
		\[\delta_{i}(f)=\frac{f}{1-e^{-\alpha_i}}+\frac{r_i(f)}{1-e^{\alpha_i}},\quad f\in K_T(\mrm{pt}),\text{ } \alpha_i\in\Delta.\]
		The operators $\delta_i$ naturally define $\mbb{Z}$-linear operators on $\bigoplus_{\lambda\in W^\Theta}K_T(\mrm{pt})$, and they are often called \textbf{Demazure operators} for $K_T(G/P_\Theta)$.
		\item\label{itm:4} (\cite[\S 2.5]{KZJ21}) $K_{\widehat{T}}^{\text{loc}}(T^*(G/P_\Theta))$. Below are the formulas for $S_{w_0}$ restricted to fixed points $\lambda\in W^\Theta$:	
		\[S_{w_0}|_\lambda=\begin{cases}
			\prod_{\alpha\in\Sigma_\Theta^+} \tfrac{1-e^{-\alpha}}{1-q^2e^{-\alpha}},& \lambda=w_0;\\
			0,&\lambda\neq w_0.
		\end{cases}\]		
		There is a natural action of $W^\Theta$ on the ring $K_{\widehat{T}}^{\text{loc}}(\mrm{pt})\simeq \mbb{Z}[e^{\pm y_1},\dotsc,e^{\pm y_n},q^{\pm 1}]$, defined on generators by $w(e^{\pm y_i}):=e^{\pm y_{w(i)}}$ and $w(q^{\pm 1})=q^{\pm1 }$ for all $w\in W^\Theta$. This $W^\Theta$-action extends to the ring $\bigoplus_{\lambda\in W^\Theta}K_{\widehat{T}}^{\text{loc}}(\mrm{pt})$ by $w(f_\lambda)_{\lambda}:=(w(f_\lambda))_{w(\lambda)}$ for all $w\in W^\Theta$. 
		Consider the $\mbb{Z}[q^{\pm1}]$-linear operators on $K_{\widehat{T}}^{\text{loc}}(\mrm{pt})$:
		\[\delta_{i}(f)=\frac{1-q^2}{1-e^{-\alpha_i}}f+\frac{1-q^2e^{\alpha_i}}{1-e^{\alpha_i}}r_i(f),\quad f\in K_{\widehat{T}}^{\text{loc}}(\mrm{pt}),\text{ } \alpha_i\in\Delta.\]
		The operators $\delta_i$ naturally define $\mbb{Z}[q^{\pm 1}]$-linear operators on $\bigoplus_{\lambda\in W^\Theta}K_{\widehat{T}}^{\text{loc}}(\mrm{pt})$, and they are often called \textbf{Demazure-Lusztig operators} for $K_{\widehat{T}}(T^*(G/P_\Theta))$.
	\end{enumerate}

	\subsection{Structure constants and puzzles}\label{subsection:puzzles}
	In this subsection, we will recall the Littlewood-Richardson rule for multiplying motivic Segre classes of Schubert cells in terms of Knutson-Tao puzzles, which is proven in \cite{KZJ17}, \cite{KZJ21}. See the introduction of \cite{K23} for a detailed history on Littlewood-Richardson rules.

	Apart from several special cases, general combinatorial formulas for the structure constants $c_{\lambda,\mu}^\nu$ in the Schubert basis $S_\lambda$ in the rings $H^*(G/P_\Theta)$, $H_T(G/P_\Theta)$, $H_{\widehat{T}}^{\mrm{loc}}(T^*(G/P_\Theta))$, $K(G/P_\Theta)$, $K_T(G/P_\Theta)$, $K_{\widehat{T}}^{\mrm{loc}}(T^*(G/P_\Theta))$ are unknown:
	\[S_\lambda\cdot S_\mu=\sum_{\nu} c_{\lambda,\mu}^\nu S_\nu,\quad \lambda,\mu,\nu\in W^\Theta.\]
	There are combinatorial formulas for the structure constants $c_{\lambda,\mu}^\nu$ in $H^*(G/P_\Theta)$ and $K(G/P_\Theta)$ in terms of Knutson-Tao puzzles for $d=1,2,3$. However, equivariantly in $H_T(G/P_\Theta), K_T(G/P_\Theta)$, $H_{\widehat{T}}^{\mrm{loc}}(T^*(G/P_\Theta)),$ and  $K_{\widehat{T}}^{\mrm{loc}}(T^*(G/P_\Theta))$ there are only combinatorial formulas for $d=1,2$ (see \cite[\S 6]{KZJ21}).

	We will explain the puzzle formula for multiplying motivic Segre classes of Schubert cells in the special case $d=1$. Consider three lists $\lambda$, $\mu$, $\nu$ in $W^\Theta$. A \textbf{puzzle} with side labels $\lambda$, $\mu$, $\nu$ is a triangle with side labels $\lambda$, $\mu$, $\nu$ that is tiled by the following \textbf{valid} puzzle pieces with edge labels\footnote{We note that an equal sign between puzzle pieces means that the puzzle pieces have the same \textit{fugacity}, which is a rational function in the $z$ variable that appears to the right of some of the puzzles (we will explain precisely what we mean by ``fugacity" shortly). An equal sign between puzzle pieces does \textit{not} mean that the puzzle pieces themselves are actually equal.}:
	
	\[
	\begin{gathered}
		\rh0000=
		\rh1111=
		\rh0101=
		\rh1{10}1{10}=
		\rh{10}0{10}0=
		\rh{10}{10}{10}{10}=1
	\end{gathered}
	\]
	\[
	\begin{gathered}
		\rh1{10}00=
		\rh{10}011=\frac{1-q^2}{1-q^2z}
	\end{gathered} \quad\quad
	\begin{gathered}
		\rh11{10}0=
		\rh001{10}=\frac{(1-q^2)z}{1-q^2z}
	\end{gathered}
	\]
	\[
	\begin{gathered}
		\rh{10}1{10}1=
		\rh0{10}0{10}= 
		\rh1010=\frac{q(1-z)}{1-q^2z}
	\end{gathered}
	\]	
	\[	
	\begin{gathered}
		\rh{10}{10}01=\frac{ q(q^2-1)}{1-q^2z}
	\end{gathered}	\quad \quad
	\begin{gathered}
		\rh01{10}{10}=\frac{(q^2-1)z}{q(1-q^2z)}
	\end{gathered}
	\]
	\[
	\uptri000=
	\uptri111=
	\uptri01{10}=
	\uptri{10}01=1
	\]

	Below is an example of a puzzle with edge labels $(\lambda,\mu,\nu)=(011,011,110)$.
	
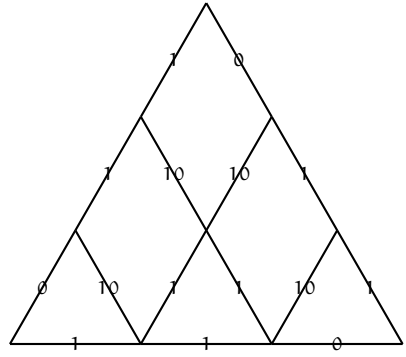
\[\begin{tikzpicture}[math mode,nodes={\mcol},x={(-0.577cm,-1cm)},y={(0.577cm,-1cm)},scale=1.5]
		\draw[thick] (1,0) -- node[pos=0.5] {\ss 0} ++(0,1); \draw[thick] (1,0) -- node[pos=0.5] {\ss 1} ++(1,0); 
		\draw[thick] (1,1) -- node[pos=0.5] {\ss 1} ++(0,1); \draw[thick] (1,1) -- node[pos=0.5] {\ss 10} ++(1,0); 
		\draw[thick] (1,2) -- node[pos=0.5] {\ss 1} ++(0,1); \draw[thick] (1,2) -- node[pos=0.5] {\ss 10} ++(1,0); \draw[thick] (1+1,2) -- node {\ss 0} ++(-1,1); 
		\draw[thick] (2,0) -- node[pos=0.5] {\ss10} ++(0,1); \draw[thick] (2,0) -- node[pos=0.5] {\ss 1} ++(1,0); 
		\draw[thick] (2,1) -- node[pos=0.5] {\ss 1} ++(0,1); \draw[thick] (2,1) -- node[pos=0.5] {\ss 1} ++(1,0); \draw[thick] (2+1,1) -- node {\ss 1} ++(-1,1); 
		\draw[thick] (3,0) -- node[pos=0.5] {\ss 10} ++(0,1); \draw[thick] (3,0) -- node[pos=0.5] {\ss 0} ++(1,0); \draw[thick] (3+1,0) -- node {\ss 1} ++(-1,1); 
	\end{tikzpicture}\]
	Triangle tiles are allowed only along the bottom row of a puzzle. Each puzzle piece has a \textbf{fugacity}, which is a function from $\Lambda$ to $\mrm{Frac}(K_{
		\widehat{T}}(\mrm{pt}))\simeq \mrm{Frac}( \mbb{Z}[e^{\pm y_1},\dotsc,e^{\pm y_n},q^{\pm 1}])$. The fugacities of the puzzle pieces are the rational functions in the $z$ variable that appear next to the puzzle pieces above, and these fugacities can also be found in \cite[\S 4.1]{KZJ21}. The \textbf{fugacity of a puzzle} is the product of the fugacities of the rhombi puzzles pieces that tile it; within a puzzle, the fugacity of a triangle tile is constant and equal to $1$, whereas the fugacity of a rhombus tile depends on its position in the puzzle. A rhombus tile that lies on the $i$-th southwest-to-northeast diagonal and the $j$-th northwest-to-southeast diagonal depends on $z=e^{y_j-y_i}$. Let $\mrm{fug}(P)$ denote the fugacity of a puzzle $P$. In this $d=1$ case, the combinatorial formula for the structure constants $c_{\lambda,\mu}^\nu$ in the Schubert basis $S_\lambda$ in the ring $K_{\widehat{T}}^{\mrm{loc}}(T^*(G/P_\Theta))$ is
	\[S_\lambda\cdot S_\mu=\sum_{\text{puzzles $P$ with side lengths $\lambda,\mu,\nu$}}\left(\mrm{fug}(\text{P})\right)S_\nu.\]

	In \cref{section:deformation}, we will deform the motivic Segre classes $S_\lambda$ by incorporating a parameter $\beta$. As we will explain in \cref{section:deformation}, our deformed motivic Segre classes are non-trivial $\beta$-deformations of the usual motivic Segre classes described in (\ref{itm:4}); they are not simply $\beta$-homogenizations of the usual motivic Segre classes. In the $d=1$ case, we will give a positive formula for the structure constants for these $\beta$-deformed classes in terms of Knutson-Tao puzzles (see \cref{thm:main}). This formula will be positive in the sense of \cref{rmk:positive}. In \cref{section:interpolating}, we will explain that our puzzle formula involving the $\beta$ parameter interpolates between the puzzle formula for multiplying motivic Segre classes of Schubert cells (the $\beta=1$ specialization) and the puzzle formula for multiplying Segre-Schwartz-MacPherson classes of Schubert cells (the $\beta=0$ specialization). See \cref{thm:main2} for the precise formulation of this statement. The proof of the puzzle formula for the deformed classes involves the representation theory of the multi-parameter quantum group of type $\widehat{\mathfrak{a}}_2$, as we will explain in \cref{subsection:R} and \cref{section:multi-parameter}.

	\section{Deformation of the motivic Segre classes}\label{section:deformation}
	
	In this section, we define the $\beta$-deformations of the motivic Segre classes of Schubert cells, which will be the main objects of study in this paper. Our deformed motivic Segre classes are non-trivial $\beta$-deformations of the usual motivic Segre classes described in (\ref{itm:4}), and they are not simply $\beta$-homogenizations of the usual motivic Segre classes. This construction works for all $d$, and the $\beta=1$ specialization recovers the usual motivic Segre classes of Schubert cells. We will \textit{not} explain how to specialize the parameter $\beta$ in order to recover the Segre-Schwartz-MacPherson classes of Schubert cells in this section: this will be left until \cref{section:interpolating}, where we discuss the connective formal group law. In \cref{subsection:deformed-ddo}, we will define divided difference operators that $\beta$-deform the $K$-theoretic Demazure-Lusztig operators. In \cref{subsection:motive-Segre}, we will construct our deformed motivic Segre classes using the deformed divided difference operators. In \cref{section:localization-techniques}, we will see that the deformed motivic Segre classes arise from canonical elements in a quotient of the $\widehat{T}$-equivariant algebraic cobordism ring of $T^*(G/P_\Theta)$.

	\subsection{Deformed divided difference operators}\label{subsection:deformed-ddo}
	
	In this subsection, we will define a one-parameter deformation of the $K$-theoretic Demazure-Lusztig operators. Our deformed operators, which we define in \cref{dfn:deformed-operator}, are not simply the $\beta$-homogenizations $\beta^{-1}\delta_i^{K_{\widehat{T}}}$ of the Demazure-Lusztig operators $\delta_i^{K_{\widehat{T}}}$ defined in (\ref{itm:4}); they are non-trivial $\beta$-deformation of the operators in (\ref{itm:4}).
	
	We will use the operators $\delta_i:=\frac{1}{1-e^{-\alpha_i}}+\frac{1}{1-e^{\alpha_i}}r_{i}$ of (\ref{itm:3}) that act on $K_{T}(\mrm{pt})$ to define linear operators on $K_{\widehat{T}}(\mrm{pt})[\beta,\beta^{-1}]$. Though the operators $r_i$ and $\delta_i$ act naturally on $K_{T}(\mrm{pt})$, we will extend them $\mbb{Z}[q^{\pm 1},\beta^{\pm 1}]$-linearly to act on $K_{\widehat{T}}(\mrm{pt})[\beta,\beta^{-1}]$. For a positive integer $k$, set $[k]:=\{1,\dotsc,k\}$.
	
	\begin{dfn}\label{dfn:Q(b)}
		Let $Q(\beta,q):=q^2+\beta-q^2\beta$ in $\mbb{Z}[q,\beta]$. 
	\end{dfn}
	\begin{dfn}\label{dfn:deformed-operator}
		For each $i\in[n-1]$, define the following $\mbb{Z}[q^{\pm 1},\beta^{\pm 1}]$-linear operator on $K_{\widehat{T}}(\mrm{pt})[\beta,\beta^{-1}]$:
		\[\partial_i:=\beta(1-q^2)\delta_{i}+q^2r_{i}=\frac{\beta(1-q^2)}{1-e^{-\alpha_i}}+\frac{Q(\beta,q)-q^2e^{\alpha_i}}{1-e^{\alpha_i}}r_{i}.\]
	\end{dfn}
	
	When $\beta=1$, the operator $\partial_i$ is identical to that of (\ref{itm:4}). We will see in \cref{lem:independent} that the $\partial_i$ satisfy the braid relations, which does not trivially follow from the fact that the operators $\delta_i^{K_{\widehat{T}}}$ of (\ref{itm:4}) satisfy the braid relations, as $\partial_i$ is not the homogenization $\beta^{-1}\delta_i^{K_{\widehat{T}}}$ of the Demazure-Lusztig operator $\delta_i^{K_{\widehat{T}}}$ defined in (\ref{itm:4}).
	\begin{rmk}
		That $\partial_i$ is indeed a linear operator on $K_{\widehat{T}}(\mrm{pt})[\beta,\beta^{-1}]$ follows from the fact that $\delta_i$ and $r_i$ are linear operators on $K_{\widehat{T}}(\mrm{pt})[\beta,\beta^{-1}]$.
	\end{rmk}
	\begin{rmk}
		The motivation for our definition of $\partial_i$ in \cref{dfn:deformed-operator} comes the \textit{connective formal group law}. See \cref{section:interpolating} for further details.
	\end{rmk}

	\begin{lem}\label{lem:independent}
		The operators $\partial_{i}$ satisfy the following relations:
		\begin{enumerate}
			\item $\partial_{i}\circ \partial_{j}=\partial_{j}\circ \partial_{i}$ whenever $|i-j|>1$.
			\item $\partial_{i}\circ \partial_{{i+1}}\circ\partial_{i}=\partial_{{i+1}}\circ\partial_{{i}}\circ\partial_{{i+1}}$ for all $i\in[n-2]$.
			\item $(\partial_{i}+q^2)\circ(\partial_{i}-Q(\beta,q))=0$ for all $i\in[n-1]$.
		\end{enumerate}
	\end{lem}
	\begin{proof}
		This is a straightforward calculation.
	\end{proof}
	
	\begin{rmk}\label{rmk:inverse-relations}
		The operator $\partial_i$ has an inverse
		\[\widetilde{\partial}_i=\frac{1}{q^2Q(\beta,q)}\left(\frac{\beta(q^2-1)}{1-e^{\alpha_i}}+\frac{Q(\beta,q)-q^2 e^{\alpha_i}}{1-e^{\alpha_i}}r_i\right)=\frac{1}{q^2Q(\beta,q)}\left(\beta\delta_i+q^2 r_i+\beta(q^2-1)\right)\]
		in the sense that $\partial_i\circ \widetilde{\partial}_i=\widetilde{\partial}_i\circ\partial_i=1$. These inverse operators $\widetilde{\partial}_i$ satisfy the following relations:
		\begin{enumerate}
			\item $\widetilde{\partial}_{i}\circ \widetilde{\partial}_{j}=\widetilde{\partial}_{j}\circ \widetilde{\partial}_{i}$ whenever $|i-j|>1$.
			\item $\widetilde{\partial}_{i}\circ \widetilde{\partial}_{{i+1}}\circ\widetilde{\partial}_{i}=\widetilde{\partial}_{{i+1}}\circ\widetilde{\partial}_{i}\circ\widetilde{\partial}_{{i+1}}$ for all $i\in[n-2]$.
			\item $\left(\widetilde{\partial}_{i}+\tfrac{1}{q^2}\right)\circ\left(\widetilde{\partial}_{i}-\tfrac{1}{Q(\beta,q)}\right)=0$ for all $i\in[n-1]$.
		\end{enumerate}
	\end{rmk}

	Write $w\in W^\Theta$ as a reduced product of simple transpositions $w=r_{i_1}\cdots r_{i_k}$. We will use the notation $\partial_w:=\partial_{{i_1}}\circ\dotsb\circ \partial_{{i_k}}$; it follows from \cref{lem:independent} that $\partial_w$ is independent of the choice of reduced word for $w$. We also use the notation $\widetilde{\partial}_w:=\widetilde{\partial}_{{i_1}}\circ\dotsb\circ \widetilde{\partial}_{{i_k}}$, which is well-defined by Remark \ref{rmk:inverse-relations}.

	\subsection{Deformed motivic Segre classes}\label{subsection:motive-Segre}
	In this subsection, we will use the deformed divided difference operators of \cref{subsection:deformed-ddo} to define our deformed motivic Segre classes. In \cref{section:localization-techniques}, we will interpret the deformed motivic Segre classes geometrically, as arising from canonical elements in a quotient of the $\widehat{T}$-equivariant algebraic cobordism ring of $T^*(G/P_\Theta)$ which satisfy a GKM-like condition (see \cref{prop:GKM}).

	Let $K_{\widehat{T}}^{\text{loc}}(\mrm{pt})[\beta,\beta^{-1}]$ be the localization of $K_{\widehat{T}}(\mrm{pt})[\beta,\beta^{-1}]$ at the elements $\left\{Q(\beta,q)-q^2e^\lambda\right\}_{\lambda\in \Sigma}$. Consider the rings 
	\[\widetilde{R}_\Lambda:=\bigoplus_{\lambda\in W^\Theta}K_{\widehat{T}}(\mrm{pt})[\beta,\beta^{-1}]\quad\text{and}\quad \widetilde{R}_\Lambda^{\text{loc}}:=\bigoplus_{\lambda\in W^\Theta}K_{\widehat{T}}^{\text{loc}}(\mrm{pt})[\beta,\beta^{-1}],\]
	each with direct sum addition and multiplication. There is a natural action of $W^\Theta$ on $\widetilde{R}_\Lambda$, given by $w\cdot(f_\lambda)_\lambda:=(w(f_\lambda))_{w(\lambda)}$ for all $w\in W^\Theta$. Moreover, the action of $W^\Theta$ on $K_{\widehat{T}}(\mrm{pt})[\beta,\beta^{-1}]$ preserves the set of elements $\left\{Q(\beta,q)-q^2e^\lambda\right\}_{\lambda\in \Sigma}$. Therefore, the action of $W^\Theta$ on $\widetilde{R}_\Lambda$ induces an action of $W^\Theta$ on $\widetilde{R}_\Lambda^\mrm{loc}$. In turn, we get an action of $\partial_w$ on $\widetilde{R}_\Lambda^\mrm{loc}$ for all $w\in W^\Theta$. Consider the element $S_{w_0}\in\widetilde{R}_\Lambda^\mrm{loc}$ defined by 
	\[S_{w_0}|_\lambda=\begin{cases}
		\prod_{\alpha\in\Sigma_\Theta^+} \frac{1-e^{-
				\alpha}}{Q(\beta,q)-q^2e^{-\alpha}},& \text{if $\lambda=w_0$};\\
		0,&\text{otherwise.}
	\end{cases}\] 
	Define $S_{w\cdot w_0}:=\partial_w(S_{w_0})$. The set $\{S_\lambda\}_{\lambda\in W^\Theta}$ forms a basis for $\widetilde{R}_\Lambda^\mrm{loc}$ over $K_{\widehat{T}}^{\text{loc}}(\mrm{pt})[\beta,\beta^{-1}]$, and we will call the elements $S_\lambda$ \textbf{deformed motivic Segre classes}. Observe that, since $\partial_i$ is not simply the $\beta$-homogenization of the Demazure-Lusztig operator defined in (\ref{itm:4}), the element 
	$S_\lambda$ is not simply a $\beta$-homogenization $\beta^{-\ell(\lambda)}S_\lambda^{K_{\widehat{T}}}$ of the usual motivic Segre class $S_\lambda^{K_{\widehat{T}}}$ that we recalled in (\ref{itm:4}). The following definition will be useful in Proposition \ref{prop:divided-difference-recursion} and \cref{section:puzzle-formula}.
	\begin{dfn}\label{dfn:homogenization}
		The \textbf{homogenization} of $S_\lambda$ by a factor $q$ is $q^{\ell(\lambda)}S_\lambda$, where $\ell(\lambda)$ is the length of $\lambda\in W^\Theta$. 
	\end{dfn}
	We will now explain in \cref{egg:K} how to specialize the parameter $\beta$ to recover the motivic Segre classes of Schubert cells. In \cref{egg:H}, we will explain how the Segre-Schwartz-MacPherson classes of Schubert cells can be recovered from the deformed motivic Segre classes.
	\begin{egg}\label{egg:K}
		At $\beta=1$, the operator $\partial_i$ and class $S_{w_0}$ are the following:
		\[\partial_i=\frac{1-q^2}{1-e^{-\alpha_i}}+\frac{1-q^2e^{\alpha_i}}{1-e^{\alpha_i}}r_i
		\quad\text{and}\quad S_{w_0}|_\lambda=\begin{cases}
			\prod_{\alpha\in\Sigma_\Theta^+} \frac{1-e^{-\alpha}}{1-q^2e^{-\alpha}},& \text{if $\lambda=w_0$};\\
			0,&\text{otherwise.}
		\end{cases}
		\]
		Therefore, $\{S_\lambda\}_{\lambda\in W^\Theta}$ is the set of images of motivic Segre classes of Schubert cells in $K_{\widehat{T}}^{\text{loc}}(T^*(G/P_\Theta))$ under the localization map $\iota$ of \cref{subsection:localization}.
	\end{egg}

	\section{A solvable lattice model for Grassmannian deformed classes}\label{section:solvable-lattice-model}
	
	In this section, we derive a solvable lattice model for the homogenizations $q^{\ell(\lambda)}S_\lambda$ of the classes $S_\lambda$ from \cref{section:deformation} in the case $d=1$. Our technique closely follows \cite{KZJ21}. In \cref{subsection:the-lattice-model}, we define a new family of rational functions in terms of a solvable lattice model that depends on parameters $q$ and $\beta$. In \cref{subsection:stabclass}, we show that the rational functions of \cref{subsection:the-lattice-model} are rational function representatives for homogenizations of the deformed motivic Segre classes $S_\lambda$ defined in \cref{subsection:motive-Segre}.
	
	\begin{assumption}
		We will assume that $d=1$ for the rest of this section.
	\end{assumption}
	
	\begin{rmk}
		Note that our lattice model for $d=1$ generalizes the lattice model in \cite{KZJ21} by incorporating an additional parameter $\beta$. When $\beta=1$, we recover the lattice model of \cite{KZJ21}. The paper \cite{KZJ21} also defines lattice models for the motivic Segre classes in the cases $d=2,3,4$. We expect that our lattice model can be extended to the cases $d=2,3,4$ as well. We leave this for future work (see \cref{section:future-work}).
	\end{rmk}
	
	\subsection{The lattice model}\label{subsection:the-lattice-model} In this subsection, we define and study a lattice model that will be used to construct rational function representatives for the deformed motivic Segre classes of \cref{subsection:deformed-ddo}. 
	
	Recall the element $Q(\beta,q)=q^2+\beta-q^2\beta$ defined in \cref{subsection:deformed-ddo}. Consider the matrix $\widehat{R}(\beta,z_2/z_1)$ below: 
	\begin{equation}\label{eqn:latticemodel}
		\widehat{R}(\beta,z_2/z_1) :=\begin{tikzpicture}[baseline=(current  bounding  box.center),scale=0.75]
			\draw[orange,arrow=0.17,arrow=0.67,rounded corners,d]  (-0.5,0.5) node[black,above] {$\ss z_1$} -- (0.5,-0.5)  (0.5,0.5) node[black,above] {$\ss z_2$} -- (-0.5,-0.5);
		\end{tikzpicture}=\bordermatrix{
			&1\slash\backslash1 & 1\slash\backslash0 & 0\slash\backslash1 & 0\slash\backslash0\cr
			1\backslash\slash1&1 & 0 & 0 & 0\cr
			1\backslash\slash0&0 & \frac{\beta(1-q^2)(z_2/z_1)}{Q(\beta,q)-q^2(z_2/z_1)}  & \frac{qQ(\beta,q)(1-(z_2/z_1))}{Q(\beta,q)-q^2(z_2/z_1)}   & 0 \cr
			0\backslash\slash1&0 & \frac{q(1-(z_2/z_1))}{Q(\beta,q)-q^2(z_2/z_1)} & \frac{\beta(1-q^2)}{Q(\beta,q)-q^2(z_2/z_1)}  & 0  \cr
			0\backslash\slash0&0 & 0 & 0 & 1
		}.
	\end{equation}
	The parameters $z_1$ and $z_2$ that appear in the definition of $\widehat{R}(\beta,z_2/z_1)$ are called \textbf{spectral parameters}. Observe that the matrix $\widehat{R}(\beta,z_2/z_1)$ depends only on the quotient $z_2/z_1$.
	\begin{prop}\label{prop:single} (cf. \cite[Prop. 2.1]{KZJ21})
		The matrix $\widehat{R}(\beta,z_2/z_1)$ satisfies all equations below.
		\begin{itemize}
			\item Yang--Baxter equation:
			\begin{equation}\label{eq:ybe0}
				\begin{tikzpicture}[baseline=-3pt,y=2cm]
					\draw[orange,d,arrow=0.1,arrow=0.4,arrow=0.7,rounded corners=4mm] (-0.5,0.5) node[black,above] {$\ss z_1$} -- (0.75,0) -- (1.5,-0.5) (0.5,0.5) node[black,above] {$\ss z_2$} -- (0.2,0) -- (0.5,-0.5) (1.5,0.5) node[black,above] {$\ss z_3$} -- (0.75,0) -- (-0.5,-0.5);
				\end{tikzpicture}
				=
				\begin{tikzpicture}[baseline=-3pt,y=2cm]
					\draw[orange,d,arrow=0.1,arrow=0.4,arrow=0.7,rounded corners=4mm] (-0.5,0.5) node[black,above] {$\ss z_1$} -- (0.25,0) -- (1.5,-0.5) (0.5,0.5) node[black,above] {$\ss z_2$} -- (0.8,0) -- (0.5,-0.5)  (1.5,0.5) node[black,above] {$\ss z_3$} -- (0.25,0) -- (-0.5,-0.5) ;
				\end{tikzpicture}
			\end{equation}
			\item Unitarity equation:
			\begin{equation}\label{eq:unit0}
				\begin{tikzpicture}[baseline=-3pt]
					\draw[orange,arrow=0.07,arrow=0.57,rounded corners,d] (-0.5,1) node[black,above] {$\ss z_1$} -- (0.5,0) -- (-0.5,-1)  (0.5,1) node[black,above] {$\ss z_2$} -- (-0.5,0) -- (0.5,-1);
				\end{tikzpicture}
				=
				\begin{tikzpicture}[baseline=-3pt]
					\draw[orange,arrow=0.1,arrow=0.6,rounded corners,d] (-0.5,1) node[black,above] {$\ss z_1$} -- (-0.5,-1)  (0.5,1) node[black,above] {$\ss z_2$} -- (0.5,-1);
				\end{tikzpicture}
			\end{equation}
			\item Value at equal spectral parameters:
			\begin{equation}\label{eq:equal0}
				\begin{tikzpicture}[baseline=-3pt,yscale=1.5]
					\draw[orange,invarrow=0.3,d] (-0.5,-0.5)   -- node[left,pos=0.3] {} (0.5,0.5) node[black,above] {$\ss z_1$};
					\draw[orange,invarrow=0.3,d] (0.5,-0.5) -- node[right,pos=0.3] {} (-0.5,0.5) node[black,above] {$\ss z_1$};
				\end{tikzpicture}
				=
				\begin{tikzpicture}[baseline=-3pt,yscale=1.5]
					\draw[orange,invarrow=0.3,rounded corners=4mm,d] (-0.5,-0.5) -- node[left] {} (0,0) -- (-0.5,0.5) node[black,above] {$\ss z_1$};
					\draw[orange,invarrow=0.3,rounded corners=4mm,d] (0.5,-0.5) -- node[right] {} (0,0) -- (0.5,0.5) node[black,above] {$\ss z_1$};
				\end{tikzpicture}
			\end{equation}
		\end{itemize}
	\end{prop}
	\begin{proof}
		This is a straightforward calculation, verified using SageMath \cite{G25}. The equations we have verified are as follows, where $\mrm{id}_3$ is the $3\times 3$ identity matrix:
		
		\textit{Yang-Baxter equation:}
		\begin{align*}
			(\widehat{R}(\beta,z_2/z_1)\otimes \mrm{id}_3)\circ (\mrm{id}_3\otimes \widehat{R}(\beta,z_3/z_1))&\circ(\widehat{R}(\beta,z_3/z_2)\otimes\mrm{id}_3)\\&
			=(\mrm{id}_3\otimes \widehat{R}(\beta,z_3/z_2))\circ(\widehat{R}(\beta,z_3/z_1)\otimes \mrm{id}_3)\circ (\mrm{id}_3\otimes \widehat{R}(\beta,z_2/z_1)).
		\end{align*}
		
		\textit{Unitarity equation:}
		\begin{align*}
			\widehat{R}(\beta,z_2/z_1)\circ \widehat{R}(\beta,z_1/z_2)=\mrm{id}_3\otimes \mrm{id}_3.
		\end{align*}
		
		\textit{Value at equal spectral parameters:}
		\begin{align*}
			\widehat{R}(\beta,1)=\mrm{id}_3\otimes\mrm{id}_3.
		\end{align*}
	\end{proof}
	
	Let $\omega$ be the list $0^{n-i_1}1^{i_1}$. Given $\lambda\in W^\Theta$, we use the $R$-matrices above to define a rational function $\overline{S}_\lambda$ in variables $x_1,\ldots,x_n, z_1,\ldots,z_n,\beta,t$ below. We will show in Proposition \ref{prop:divided-difference-recursion} that our rational functions represent the classes defined in \cref{section:deformation}.
	The rational function $\overline{S}_\lambda$ is defined to be the partition function of a wiring diagram constructed from the solvable lattice model with respect to the  $R$-matrices:

	\begin{equation}\label{eq:defS}
		\overline{S}_\lambda
		:=
		\begin{tikzpicture}[baseline=1.8cm,scale=0.75]
			\foreach\x/\t in {1/z_1,2/z_2,3/\cdots,4/z_n}
			\draw[orange,d,invarrow=0.9] (\x,0) node[black,below] {$\m 1$} -- node[black,right=-1mm,pos=0.1] {$\ss\t$} ++(0,5);
			\foreach\y/\t in {1/x_1,2/x_2,3/\vdots,4/x_n}
			\draw[orange,d,invarrow=0.9] (0,5-\y) -- node[black,above=-1mm,pos=0.1] {$\ss\t$} ++(5,0) node[black,right] {$\m 1$};
			\draw[decorate,decoration=brace] (-0.3,1) -- node[black,left] {$\m\omega$} (-0.3,4);			\draw[decorate,decoration=brace] (1,5.3) -- node[black,above] {$\m\lambda$} (4,5.3);
		\end{tikzpicture}
	\end{equation}
	The way to read the wiring diagram is as follows (see Example \ref{ex:d1} for explicit examples). We view the wiring diagram as a graph, where each boundary point and internal lattice point in the wiring diagram is considered a vertex, and each line segment connecting two vertices is an edge. A \textbf{state} in the wiring diagram is an assignment of a $0$ or $1$ label to each edge in the wiring diagram (apart from the labels of edges that are adjacent to boundary points, which have already been fixed). The \textbf{weight} of an internal vertex for a given state is the matrix entry in $\widehat{R}(\beta,x_i/z_j)$ corresponding to the edge labels of the four edges adjacent to the vertex. The \textbf{Boltzmann weight} of a state is the product over the weights of all internal vertices in the wiring diagram. The \textbf{partition function} $\overline{S}_\lambda$ of the wiring diagram above is the sum over the Boltzmann weights over all states of the wiring diagram. 
	
	\begin{egg}\label{ex:d1}
		We compute $\overline{S}_\lambda$ for $\lambda=01,10$:
		\begin{align*}
			\overline{S}_{01}&=
			\begin{tikzpicture}[baseline=(current  bounding  box.center),scale=0.75]
				\foreach\x/\t in {1/z_1,2/z_2}
				\draw[orange,d,invarrow=0.9] (\x,0) -- node[black,right=-1mm,pos=0.1] {$\ss\t$} ++(0,3);
				\foreach\y/\t in {1/x_1,2/x_2}
				\draw[orange,d,invarrow=0.9] (0,3-\y) -- node[above=-1mm,pos=0.05] {$\ss\t$} ++(3,0);
				\foreach\c in {(.5,1),(.9,.5),(1.9,.5),(2.5,1),(2.5,2),(1.9,2.5),
					(.9,1.5),(1.9,1.5),(1.5,1),(1.5,2)} { \node at \c {$1$}; }
				\foreach\c in {(.5,2),(.5,2),(.9,2.5)} { \node at \c {$0$}; }
			\end{tikzpicture}
			=\frac{\beta(1-q^2)}{Q(\beta,q)-q^2(x_1/z_1)}
			\\
			\overline{S}_{10}&=
			\begin{tikzpicture}[baseline=(current  bounding  box.center),scale=0.75]
				\foreach\x/\t in {1/z_1,2/z_2}
				\draw[orange,d,invarrow=0.9] (\x,0) -- node[black,right=-1mm,pos=0.1] {$\ss\t$} ++(0,3);
				\foreach\y/\t in {1/x_1,2/x_2}
				\draw[orange,d,invarrow=0.9] (0,3-\y) -- node[black,above=-1mm,pos=0.05] {$\ss\t$} ++(3,0);
				\foreach\c in {(.5,1),(.9,.5),(1.9,.5),(2.5,1),(2.5,2),(.9,2.5),
					(.9,1.5),(1.9,1.5),(1.5,1)} { \node at \c {$1$}; }
				\foreach\c in {(.5,2),(.5,2),(1.9,2.5),(1.5,2)} { \node at \c {$0$}; }
			\end{tikzpicture}
			=\frac{q(1-x_1/z_1)}{Q(\beta,q)-q^2(x_1/z_1)}\cdot \frac{\beta(1-q^2)}{Q(\beta,q)-q^2(x_1/z_2)}
		\end{align*}
	\end{egg}

	\begin{lem}\label{lem:Winv}
		$\overline{S}_\lambda$ is invariant under the action of $W^\Theta$ on the variables $x_i$.
	\end{lem}
	\begin{proof}
		The argument is identical to that of \cite[Lemma 3.11]{KZJ17}.
	\end{proof}

	\begin{lem}\label{lem:spec}
		Write $\bar\sigma$ for $\sigma^{-1}$. Then $\overline{S}_\lambda$ is well-defined at every specialization $x_i=z_{\bar\sigma(i)}$,
		$\sigma\in W^\Theta$, and given by
		\begin{equation}\label{eq:defSfp}
			\overline{S}_\lambda|_\sigma = 
			\begin{tikzpicture}[baseline=-3pt,scale=0.9]
				\draw (0,-1) rectangle (5,1); \node at (2.5,0) {$\sigma$};
				\foreach\x/\t/\u in {1/z_1/z_{\bar\sigma(1)},2/z_2/z_{\bar\sigma(2)},3/\cdots/\cdots,4/z_n/z_{\bar\sigma(n)}}
				{
					\draw[orange,d,invarrow=0.5] (\x,1) -- node[black,right,pos=0.5] {$\ss\t$} ++(0,1);
					\draw[orange,d,invarrow=0.5] (\x,-2) -- node[black,right=-1mm,pos=0.4] {$\ss\u$} ++(0,1);
				}
				\draw[decorate,decoration=brace] (4,-2.3) -- node[below] {$\m\omega$} (1,-2.3);
				\draw[decorate,decoration=brace] (1,2.3) -- node[above] {$\m\lambda$} (4,2.3);
			\end{tikzpicture}
		\end{equation}
		where the rectangle labeled $\sigma$ is any wiring diagram of $\sigma$, each crossing being an $R$-matrix. Furthermore, $\overline{S}_\lambda|_{\sigma}$ only depends on the class of $\sigma$ in $W^\Theta$.
	\end{lem}
	\begin{proof}
		As Proposition \ref{prop:single} holds, the proof is identical to those of \cite[Lemma 3.12]{KZJ17}, \cite[Lemma 2.4]{KZJ21}.
	\end{proof}
	
	\begin{egg}\label{ex:d1b}
		We compute $\overline{S}_\lambda|_{\sigma}$ for $\lambda=01,10$ and $\sigma=01,10$, where we identify the identity permutation with $01$
		and the nontrivial permutation with $10$:
		\begin{align*}
			\overline{S}_{01}|_{01} &=
			\begin{tikzpicture}[baseline=(current  bounding  box.center),scale=0.75]
				\draw[orange,d,invarrow=.45] (1,0) -- node[black,xshift=-.15cm,pos=0.7] {$0$} node[black,xshift=-.15cm,pos=0.3] {$0$} node[black,right,pos=.9] {$\ss z_1$} (1,2);
				\draw[orange,d,invarrow=.45] (2,0) -- node[black,xshift=-.15cm,pos=0.7] {$1$} node[black,xshift=-.15cm,pos=0.3] {$1$} node[black,right,pos=.9] {$\ss z_2$} (2,2);
			\end{tikzpicture}
			=1
			&
			\overline{S}_{10}|_{01} &=
			\begin{tikzpicture}[baseline=(current  bounding  box.center),scale=0.75]
				\draw[orange,d,invarrow=.45] (1,0) -- node[black,xshift=-.15cm,pos=0.7] {$1$} node[black,xshift=-.15cm,pos=0.3] {$0$} node[black,right,pos=.9] {$\ss z_1$} (1,2);
				\draw[orange,d,invarrow=.45] (2,0) -- node[black,xshift=-.15cm,pos=0.7] {$0$} node[black,xshift=-.15cm,pos=0.3] {$1$} node[black,right,pos=.9] {$\ss z_2$} (2,2);
			\end{tikzpicture}
			=0
			\\
			\overline{S}_{01}|_{10} &=
			\begin{tikzpicture}[baseline=(current  bounding  box.center),scale=0.75]
				\draw[orange,d,invarrow=.6] (2,0) -- node[black,xshift=-.15cm,pos=0.7] {$0$} node[black,xshift=.15cm,pos=0.3] {$1$} node[black,right,pos=.9] {$\ss z_1$} (1,2);
				\draw[orange,d,invarrow=.6] (1,0) -- node[black,xshift=.15cm,pos=0.7] {$1$} node[black,xshift=-.15cm,pos=0.3] {$0$} node[black,right,pos=.9] {$\ss z_2$} (2,2);
			\end{tikzpicture}
			=\frac{\beta(1-q^2)}{Q(\beta,q)-q^2(z_2/z_1)}
			&
			\overline{S}_{10}|_{10} &=
			\begin{tikzpicture}[baseline=(current  bounding  box.center),scale=0.75]
				\draw[orange,d,invarrow=.6] (2,0) -- node[black,xshift=-.15cm,pos=0.7] {$1$} node[black,xshift=.15cm,pos=0.3] {$1$} node[black,right,pos=.9] {$\ss z_1$} (1,2);
				\draw[orange,d,invarrow=.6] (1,0) -- node[black,xshift=.15cm,pos=0.7] {$0$} node[black,xshift=-.15cm,pos=0.3] {$0$} node[black,right,pos=.9] {$\ss z_2$} (2,2);
			\end{tikzpicture}
			=\frac{q(1-z_2/z_1)}{Q(\beta,q)-q^2(z_2/z_1)}
		\end{align*}
		
		Compare with Example~\ref{ex:d1}.
	\end{egg}

	\subsection{Characterization of the deformed motivic Segre classes}\label{subsection:stabclass}
	In this subsection, we will characterize the homogenizations of the deformed motivic Segre classes defined in \cref{subsection:deformed-ddo} in terms of the rational functions $\overline{S}_\lambda$ defined by the lattice model of \cref{subsection:the-lattice-model}. 
	
	In \cref{prop:charact} below, the transposition $r_i$ fixes both $\beta$ and $q$, and acts on the set of variables $z_1,\dotsc,z_n$ by swapping the indices $i$ and $i+1$ and fixing all other indices:
	
	\begin{prop}\label{prop:charact}
		The classes $\overline{S}_\lambda$ are entirely determined by the following properties:
		\begin{enumerate}
			\item {\em Triangularity:} $\overline{S}_\lambda|_\sigma=0$ unless $\sigma\geq \lambda$ in the Bruhat order.
			\item {\em Diagonal entries:}
			\[
			\overline{S}_\lambda|_\lambda
			=\prod_{i<j: \lambda_i>\lambda_j}
			\frac{q(1-z_{j}/z_{i})}{Q(\beta,q)-q^2(z_{j}/z_{i})}
			\]
			\item {\em Exchange relation:}
			\[
			r_i \overline{S}_\lambda|_{\sigma r_i}
			=\begin{cases}
				\frac{\beta(1-q^2)}{Q(\beta,q)-q^2(z_{i}/z_{i+1})}
				\overline{S}_\lambda|_\sigma 
				+ 
				\frac{q Q(\beta,q)(1-z_{i}/z_{i+1})}{Q(\beta,q)-q^2(z_{i}/z_{i+1})}
				\overline{S}_{\lambda r_i}|_\sigma
				&\lambda_i<\lambda_{i+1}
				\\
				\overline{S}_\lambda|_\sigma
				& \lambda_i=\lambda_{i+1}
				\\
				\frac{\beta (1-q^2)(z_{i}/z_{i+1})}{Q(\beta,q)-q^2(z_{i}/z_{i+1})}
				\overline{S}_\lambda|_\sigma 
				+ 
				\frac{q(1-z_{i}/z_{i+1})}{Q(\beta,q)-q^2(z_{i}/z_{i+1})}\overline{S}_{\lambda r_i}|_\sigma
				& \lambda_i>\lambda_{i+1}
			\end{cases}
			\]
		\end{enumerate}
	\end{prop}
	\begin{proof}
		The proof is identical to that of \cite[Prop. 2.7]{KZJ21}.
	\end{proof}
	
	\begin{rmk}
		When $\beta=1$, the recursion in \cref{prop:charact} is exactly the recursion for the motivic Segre classes of the Schubert cells in \cite[Prop. 2.7]{KZJ21}. In this sense, our $\beta$-deformed classes can be viewed as a $\beta$-deformation of motivic Segre classes in the extended ring $\mrm{Frac}(K_{\widehat{T}}(T^*(G/P_\Theta))\otimes_{\mbb{Z}}\mbb{Z}[\beta,\beta^{-1}])$.
	\end{rmk}
	
	\begin{rmk}\label{rmk:clarity}
		In the recursion of Proposition \ref{prop:charact} above, the classes $\overline{S}_\lambda$ are in fact entirely defined by $(1), (2),$ and the case $\lambda_i> \lambda_{i+1}$ of $(3)$. To see this, begin with the class $\overline{S}_{w_0}$. Then $(2)$ gives us the formula for $\overline{S}_{w_0}|_{w_0}$ and $(1)$ tells us that $\overline{S}_{w_0}|_\lambda=0$ for all $\lambda<w_0$. Now we apply the case $\lambda_i> \lambda_{i+1}$ of $(3)$ recursively to obtain the remaining classes. 
	\end{rmk}
	
	Recall the homogenization of a $\beta$-deformed class by a factor $q$, defined in Definition \ref{dfn:homogenization}.
	
	\begin{prop}\label{prop:divided-difference-recursion}
		The $\overline{S}_\lambda$ defined in this section are rational function representatives for the homogenizations of the $\beta$-deformed classes $q^{\ell(\lambda)}S_\lambda$ defined in \cref{subsection:motive-Segre}, in the sense that the $q^{\ell(\lambda)}S_\lambda$ are characterized by the properties (1)--(3) of \cref{prop:charact}.
	\end{prop}
	\begin{proof}
		Let $\{q^{\ell(\lambda)}S_\lambda\}_{\lambda\in W^\Theta}$ be the homogenizations of the $\beta$-deformed classes defined in \cref{section:deformation}. Then
		\[q^{\ell(w_0)}S_{w_0}|_\lambda=\begin{cases}
			\prod_{\alpha\in\Sigma_\Theta^+} \frac{q(1-e^{-
					\alpha})}{Q(\beta,q)-q^2e^{-\alpha}},& \text{if $\lambda=w_0$};\\
			0,&\text{otherwise.}
		\end{cases}\]
		Moreover, given $\lambda\in W^\Theta$ with $\lambda_i>\lambda_{i+1}$, we have $q^{\ell(r_i(\lambda))}S_{r_i(\lambda)}=\tfrac{1}{q}\partial_{i}(q^{\ell(\lambda)}S_\lambda)$. In particular,
		\begin{align*}
			&(q^{\ell(r_i(\lambda))}S_{r_i(\lambda)})|_{\sigma}=\frac{\beta(1-q^2)}{q(1-e^{-\alpha_i})} ((q^{\ell(\lambda)}S_{\lambda})|_\sigma)+\frac{Q(\beta,q)-q^2e^{\alpha_i}}{q(1-e^{\alpha_i})}r_i((q^{\ell(\lambda)}S_\lambda)|_{r_i(\sigma)})\\
			&\implies r_i((q^{\ell(\lambda)}S_\lambda)|_{r_i(\sigma)})=\frac{q(1-e^{\alpha_i})}{Q(\beta,q)-q^2e^{\alpha_i}} ((q^{\ell(r_i(\lambda))}S_{r_i(\lambda)})|_\sigma)-\frac{1-e^{\alpha_i}}{1-e^{-\alpha_i}}\frac{\beta(1-q^2)}{Q(\beta,q)-q^2e^{\alpha_i}}((q^{\ell(\lambda)} S_\lambda)|_{\sigma}).
		\end{align*}
		Observe that $\tfrac{1-e^{\alpha_i}}{1-e^{-\alpha_i}}=-e^{\alpha_i}$. Setting $z_{i}:=e^{y_{i}}$ and $z_{i+1}:=e^{y_{i+1}}$, we see that the recursion for $\lambda_{i}>\lambda_{i+1}$ in $(3)$ is satisfied by the classes $q^{\ell(\lambda)}S_\lambda$. It is straightforward to verify that the classes $q^{\ell(\lambda)}S_\lambda$ satisfy $(1)$ and $(2)$ of Proposition \ref{prop:charact}. Therefore, by Remark \ref{rmk:clarity}, the homogenizations of the $\beta$-deformed classes defined in \cref{section:deformation} coincide with the rational functions $\overline{S}_\lambda$ defined in this section.
	\end{proof}

	\section{The puzzle rule for multiplying Grassmannian deformed classes}\label{section:puzzle-formula}
	
	In this section, we will prove a \textit{positive} formula for the structure constants in the basis of deformed motivic Segre classes $S_\lambda$ when $d=1$ in terms of Knutson-Tao puzzles. As the $S_\lambda$ are not simply the $\beta$-homogenizations of the motivic Segre classes of Schubert cells defined in (\ref{itm:4}), our puzzle formula for the structure constants in the $\beta$-deformed basis is a non-trivial $\beta$-deformation of the puzzle formula for the structure constants in the basis of usual motivic Segre classes described in \cref{subsection:puzzles}. As we will show in \cref{section:localization-techniques}, the deformed motivic Segre classes of Schubert cells arise from canonical elements in a quotient of the $\widehat{T}$-equivariant algebraic cobordism ring of $T^*(G/P_\Theta)$. Although our deformed classes have this geometric description, we currently do not have a geometric understanding for why the structure constants in the $\beta$-deformed basis are positive. The only proof that we have for the positivity of the structure constants is the combinatorial proof given by \cref{thm:main}. In the proof of the puzzle formula, we introduce $R$, $U$, and $D$ matrices that satisfy various Yang-Baxter type equations. We will show in \cref{section:multi-parameter} that these matrices are intertwiners for representations of the multi-parameter quantum group of type $\widehat{a}_2$. That these matrices are related to the representation theory of the multi-parameter quantum group is unexpected.
	
	In \cref{subsection:fugacities}, we define the fugacities of the rhombus and triangle puzzle pieces which are required for the puzzle formula. In \cref{subsection:R}, we define several matrices, which are constructed from the fugacities in \cref{subsection:fugacities}, and we show that these matrices satisfy several Yang-Baxter type equations. In \cref{subsection:proof}, we prove the puzzle formula in \cref{thm:main} and show that the formula is suitably positive in \cref{rmk:positive}. We also provide several examples of puzzle computations.
	
	\begin{assumption}
		We will assume that $d=1$ for the rest of this section.
	\end{assumption}
	
	\subsection{Fugacities of the puzzle pieces}\label{subsection:fugacities} 
	
	In this subsection, we will present the rhombus and triangle tiles that are used to tile the puzzles for multiplying the deformed motivic Segre classes, as well as the fugacities of these tiles. The tiles that appear are the same as those that appear for multiplying usual motivic Segre classes when $d=1$ (see \cref{subsection:puzzles}). Recall the definition of the Knutson-Tao puzzle for $d=1$ given in \cref{subsection:puzzles}. The same definition applies in our case, the only difference being that our fugacities are more general, as they involve a parameter $\beta$. When $\beta=1$, the fugacities are the same as those that appear for multiplying the usual motivic Segre classes (see \cref{subsection:puzzles} or \cite[\S 4.1]{KZJ21}). 
	
	We will denote the structure constants of the homogenizations $q^{\ell(\lambda)}S_\lambda$ by $\overline{c}_{\lambda,\mu}^\nu$. By definition, we have
	\[(q^{\ell(\lambda)}S_\lambda)(q^{\ell(\mu)}S_\mu)=\sum_{\nu}\overline{c}_{\lambda,\mu}^\nu (q^{\ell(\nu)}S_\nu).\]
	Thus, the relationship between the Littlewood-Richardson coefficient $c_{\lambda,\mu}^\nu$ and the coefficient $\overline{c}_{\lambda,\mu}^\nu$ is
	\begin{equation}\label{eq:relationship}
		c_{\lambda,\mu}^\nu=q^{\ell(\nu)-\ell(\lambda)-\ell(\mu)}\overline{c}_{\lambda,\mu}^\nu.
	\end{equation}
	
	Below we list the fugacities of the rhombus tiles. 	The main result of this section is \cref{thm:main}, which implies that the structure constants in the $\beta$-deformed basis are sums of products of the fugacities of these rhombus puzzle pieces. These fugacities also appear in the $R_{g,r}$ matrix in \cref{subsection:R}.
	
	\[
	\begin{gathered}
		\rh0000=
		\rh1111=
		\rh0101=
		\rh1{10}1{10}=
		\rh{10}0{10}0=1
	\end{gathered}\quad\quad
	\begin{gathered}
		\rh{10}{10}{10}{10}=Q(\beta,q)
	\end{gathered}
	\]
	\[
	\begin{gathered}
		\rh1{10}00=
		\rh{10}011=\frac{\beta(1-q^2)}{Q(\beta,q)-q^2z}
	\end{gathered}\quad 
	\begin{gathered}
		\rh1010=\frac{q(1-z)}{Q(\beta,q)-q^2z}
	\end{gathered}
	\]
	\[ 
	\begin{gathered}
		\rh11{10}0=
		\rh001{10}=\frac{\beta (1-q^2)z}{Q(\beta,q)-q^2z}
	\end{gathered}\quad\quad
	\begin{gathered}
		\rh{10}{10}01=\frac{\beta q(q^2-1)}{Q(\beta,q)-q^2z}
	\end{gathered}
	\]	
	\[
	\begin{gathered}
		\rh0{10}0{10}=
		\rh{10}1{10}1=\frac{qQ(\beta,q)(1-z)}{Q(\beta,q)-q^2z}
	\end{gathered}\quad\quad
	\begin{gathered}
		\rh01{10}{10}=\frac{\beta Q(\beta,q)(q^2-1)z}{q(Q(\beta,q)-q^2z)}
	\end{gathered}
	\]
	
	For $i,j,k,\ell\in\{0,1,10\}$, will use the notation 
	
	\[f_{i,j}^{k,\ell}(z):=\mrm{fug}\left(\begin{gathered}
		\rh ij{\ell}{k}
	\end{gathered}\right).\]

	Below we list the fugacities of the triangle tiles, which appear in the $U(\beta)$ and $D(\beta)$ matrices of \cref{subsection:R}, and are used to prove a key result, \cref{prop:ybe}:
	
	\[
	\uptri000=
	\uptri111=
	\uptri01{10}=
	\uptri{10}01=
	\uptri1{10}0=1 \quad\quad
	\uptri{10}{10}{10}=\frac{-Q(\beta,q)}{q}
	\]
	
	\[
	\downtri000=
	\downtri111=
	\downtri01{10}=
	\downtri{10}01=
	\downtri1{10}0=1\quad\quad 
	\downtri{10}{10}{10}= - q
	\]
	Note that the $\uptri1{10}0$ and $\uptri{10}{10}{10}$ puzzle pieces and the upside-down triangle puzzle pieces are never used to tile a Knutson-Tao puzzle. However, these puzzle pieces and their fugacities are necessary for the proof of the puzzle rule for multiplying the elements $q^{\ell(\lambda)} S_\lambda$, as we will explain in \cref{subsection:R} and \cref{subsection:proof}.

	\subsection{$R$, $U$, and $D$ matrices}\label{subsection:R}
	
	In this subsection, we will define the $R$, $U$, and $D$ matrices, which will be used to prove the puzzle formula for multiplying the $q^{\ell(\lambda)}S_\lambda$ in \cref{subsection:proof}. We will also prove \cref{prop:ybe} in this subsection, which is the key result used to prove our puzzle formula, \cref{thm:main}. We will prove in \cref{section:multi-parameter} that the $R$, $U$, and $D$ matrices defined in this subsection are intertwiners for representations of the \textit{multi-parameter quantum group} of type $\widehat{\mathfrak{a}}_2$.

	Recall the solvable lattice model of \cref{section:solvable-lattice-model}. In Proposition \ref{prop:divided-difference-recursion}, we showed that the rational functions defined by the solvable lattice model are rational function representatives for the \textit{homogenization} $\{q^{\ell(\lambda)}S_\lambda\}_{\lambda\in W^\Theta}$ of the basis $\{S_\lambda\}_{\lambda\in W^\Theta}$. Puzzles compute the structure constants for the basis $\{S_\lambda\}_{\lambda\in W^\Theta}$ if and only if they compute the structure constants for the basis $\{q^{\ell(\lambda)}S_\lambda\}_{\lambda\in W^\Theta}$ (see Equation (\ref{eq:relationship})). We will prove that puzzles compute the structure constants for the basis $\{q^{\ell(\lambda)}S_\lambda\}_{\lambda\in W^\Theta}$ using the solvable lattice model of \cref{section:solvable-lattice-model}. The matrix $R_{g,r}$ below records the fugacities of the rhombus tiles from \cref{subsection:fugacities}. The other matrices are used to prove Theorem \ref{thm:main}.
	All $9\times 9$ matrices below are nonsingular at every specialization of $\beta$ except for $Q(\beta,q)=0$. 
	
	Set $z:=z_2/z_1$ to reduce the clutter in the matrix entries below.
	\begin{align*}
		&R_{g,r}(\beta,z/q^2)=
		\begin{tikzpicture}[baseline=(current  bounding  box.center),scale=0.75]
			\draw[green, thick,arrow=.27,rounded corners,d]  (-0.5,0.5) node[black,above] {$\ss z_1$} -- (0.5,-0.5) ;
			\draw[red, thick,arrow=.27,rounded corners,d] (0.5,0.5) node[black,above] {$\ss z_2$} -- (-0.5,-0.5) ;
		\end{tikzpicture}=
		\\&
		\bordermatrix{
			&1\slash\backslash1 & 1\slash\backslash0 & 1\slash\backslash10 & 0\slash\backslash1 & 0\slash\backslash0 & 0\slash\backslash10 & 10\slash\backslash1 & 10\slash\backslash0 & 10\slash\backslash10\cr
			1\backslash\slash1&1 & 0 & 0 & 0 & 0 & f_{1,1}^{0,10}(z) & 0 & 0 & 0\cr
			1\backslash\slash0&0 & 0 & 0 & f_{1,0}^{0,1}(z) & 0 & 0 & 0 & 0 & 0\cr
			1\backslash\slash10&0 & 0 & 0 & 0 & f_{1,10}^{0,0}(z) & 0 & 1 & 0 & 0\cr
			0\backslash\slash1&0 & 1 & 0 & 0 & 0 & 0 & 0 & 0 & f_{0,1}^{10,10}(z)\cr
			0\backslash\slash0&0 & 0 & 0 & 0 & 1 & 0 & f_{0,0}^{10,1}(z) & 0 & 0\cr
			0\backslash\slash10&0 & 0 & 0 & 0 & 0 & 0 & 0 & f_{0,10}^{10,0}(z) & 0\cr
			10\backslash\slash1&0 & 0 & f_{10,1}^{1,10}(z) & 0 & 0 & 0 & 0 & 0 & 0\cr
			10\backslash\slash0& f_{10,0}^{1,1}(z) & 0 & 0 & 0 & 0 & 1 & 0 & 0 & 0\cr
			10\backslash\slash10&0 & f_{10,10}^{1,0}(z) & 0 & 0 & 0 & 0 & 0 & 0 & f_{10,10}^{10,10}(z)
		}
	\end{align*}
	\begin{align*}
		&R_{b,r}(\beta,z/q)=
		\begin{tikzpicture}[baseline=(current  bounding  box.center),scale=0.75]
			\draw[blue, thick,arrow=.27,rounded corners,d]  (-0.5,0.5) node[black,above] {$\ss z_1$} -- (0.5,-0.5) ;
			\draw[red, thick,arrow=.27,rounded corners,d] (0.5,0.5) node[black,above] {$\ss z_2$} -- (-0.5,-0.5) ;
		\end{tikzpicture}=
		\\&
		\bordermatrix{
			&1\slash\backslash1 & 1\slash\backslash0 & 1\slash\backslash10 & 0\slash\backslash1 & 0\slash\backslash0 & 0\slash\backslash10 & 10\slash\backslash1 & 10\slash\backslash0 & 10\slash\backslash10\cr
			1\backslash\slash1&1 & 0 & 0 & 0 & 0 & 0 & 0 & 0 & 0\cr
			1\backslash\slash0&0 & 0 & 0 & 1 & 0 & 0 & 0 & 0 & 0\cr
			1\backslash\slash10&0 & f_{1,1}^{0,10}(1/z) & 0 & 0 & 0 & (z\cdot f_{0,1}^{10,10}(1/z)) & f_{10,1}^{1,10}(1/z)& 0 & 0\cr
			0\backslash\slash1&0 & f_{1,0}^{0,1}(1/z) & 0 & 0 & 0 &  f_{10,0}^{1,1}(1/z) & f_{10,0}^{1,1}(1/z) & 0 & 0\cr
			0\backslash\slash0&0 & 0 & 0 & 0 & 1 & 0 & 0 & 0 & 0\cr
			0\backslash\slash10&0 & 0 & 0 & 0 & 0 & 0 & 0 & 1 & 0\cr
			10\backslash\slash1&0 & 0 &  1  & 0 & 0 & 0 & 0 & 0 & 0\cr
			10\backslash\slash0&0 & f_{1,1}^{0,10}(1/z) & 0 & 0 & 0 & f_{10,1}^{1,10}(1/z) & (\tfrac{1}{z}\cdot f_{10,10}^{1,0}(1/z)) & 0 & 0\cr
			10\backslash\slash10&0 & 0 & 0 & 0 & 0 & 0 & 0 & 0 & Q(\beta,q)
		}
	\end{align*}
	\begin{align*}
		&R_{g,b}(\beta,z/q)=
		\begin{tikzpicture}[baseline=(current  bounding  box.center),scale=0.75]
			\draw[green, thick,arrow=.27,rounded corners,d]  (-0.5,0.5) node[black,above] {$\ss z_1$} -- (0.5,-0.5) ;
			\draw[blue, thick,arrow=.27,rounded corners,d] (0.5,0.5) node[black,above] {$\ss z_2$} -- (-0.5,-0.5) ;
		\end{tikzpicture}=
		\\&
		\bordermatrix{
			&1\slash\backslash1 & 1\slash\backslash0 & 1\slash\backslash10 & 0\slash\backslash1 & 0\slash\backslash0 & 0\slash\backslash10 & 10\slash\backslash1 & 10\slash\backslash0 & 10\slash\backslash10\cr
			1\backslash\slash1&1 & 0 & 0 & 0 & 0 & 0 & 0 & 0 & 0\cr
			1\backslash\slash0&0 & 0 & 0 & 1 & 0 & 0 & 0 & 0 & 0\cr
			1\backslash\slash10&0 & f_{1,1}^{0,10}(1/z) & 0 & 0 & 0 & (\tfrac{1}{z}\cdot f_{10,10}^{1,0}(1/z)) & f_{10,1}^{1,10}(1/z) & 0 & 0\cr
			0\backslash\slash1&0 & f_{1,0}^{0,1}(1/z) & 0 & 0 & 0 &  f_{10,0}^{1,1}(1/z) & f_{10,0}^{1,1}(1/z) & 0 & 0\cr
			0\backslash\slash0&0 & 0 & 0 & 0 & 1 & 0 & 0 & 0 & 0\cr
			0\backslash\slash10&0 & 0 & 0 & 0 & 0 & 0 & 0 & 1 & 0\cr
			10\backslash\slash1&0 & 0 &  1  & 0 & 0 & 0 & 0 & 0 & 0\cr
			10\backslash\slash0&0 & f_{1,1}^{0,10}(1/z) & 0 & 0 & 0 & f_{10,1}^{1,10}(1/z) & (z\cdot f_{0,1}^{10,10}(1/z))& 0 & 0\cr
			10\backslash\slash10&0 & 0 & 0 & 0 & 0 & 0 & 0 & 0 & Q(\beta,q)
		}
	\end{align*}

	\begin{align*}
		&R_{r,g}(\beta,q^2z):=
		\begin{tikzpicture}[baseline=(current  bounding  box.center),scale=0.75]
			\draw[red, thick,arrow=.27,rounded corners,d]  (-0.5,0.5) node[black,above] {$\ss z_1$} -- (0.5,-0.5) ;
			\draw[green, thick,arrow=.27,rounded corners,d] (0.5,0.5) node[black,above] {$\ss z_2$} -- (-0.5,-0.5) ;
		\end{tikzpicture} = R_{g,r}(\beta,1/(q^2z))^{-1}\quad
		&R_{r,b}(\beta,qz):=
		\begin{tikzpicture}[baseline=(current  bounding  box.center),scale=0.75]
			\draw[red, thick,arrow=.27,rounded corners,d]  (-0.5,0.5) node[black,above] {$\ss z_1$} -- (0.5,-0.5) ;
			\draw[blue, thick,arrow=.27,rounded corners,d] (0.5,0.5) node[black,above] {$\ss z_2$} -- (-0.5,-0.5) ;
		\end{tikzpicture} = R_{r,b}(\beta,1/(qz))^{-1}
	\end{align*}
	\[ R_{b,g}(\beta,qz):=
	\begin{tikzpicture}[baseline=(current  bounding  box.center),scale=0.75]
		\draw[blue, thick,arrow=.27,rounded corners,d]  (-0.5,0.5) node[black,above] {$\ss z_1$} -- (0.5,-0.5) ;
		\draw[green, thick,arrow=.27,rounded corners,d] (0.5,0.5) node[black,above] {$\ss z_2$} -- (-0.5,-0.5) ;
	\end{tikzpicture} = R_{g,b}(\beta,1/(qz))^{-1}
	\]
	
	\begin{align*}
		&R_{\text{g,g}}(\beta,z)=
		\begin{tikzpicture}[baseline=(current  bounding  box.center),scale=0.75]
			\draw[green, thick,arrow=.27,rounded corners,d]  (-0.5,0.5) node[black,above] {$\ss z_1$} -- (0.5,-0.5) ;
			\draw[green, thick,arrow=.27,rounded corners,d] (0.5,0.5) node[black,above] {$\ss z_2$} -- (-0.5,-0.5) ;
		\end{tikzpicture}= 
		\\&
		\bordermatrix{
			&1\slash\backslash1 & 1\slash\backslash0 & 1\slash\backslash10 & 0\slash\backslash1 & 0\slash\backslash0 & 0\slash\backslash10 & 10\slash\backslash1 & 10\slash\backslash0 & 10\slash\backslash10\cr
			1\backslash\slash1& 1 & 0 & 0 & 0 & 0 & 0 & 0 & 0 & 0\cr
			1\backslash\slash0&0 & f_{1,1}^{0,10}(1/z) & 0 & f_{10,1}^{1,10}(1/z)& 0 & 0 & 0 & 0 & 0\cr
			1\backslash\slash10&0 & 0 & f_{1,1}^{0,10}(1/z)  & 0 & 0 & 0 & f_{10,1}^{1,10}(1/z) & 0 & 0\cr
			0\backslash\slash1&0 & f_{1,0}^{0,1}(1/z) & 0 & f_{10,0}^{1,1}(1/z) & 0 & 0 & 0 & 0 & 0 \cr
			0\backslash\slash0&0 & 0 & 0 & 0 & 1 & 0 & 0 & 0 & 0\cr
			0\backslash\slash10&0 & 0 & 0 & 0 & 0 & f_{1,1}^{0,10}(1/z) & 0 & f_{1,0}^{0,1}(1/z) & 0\cr
			10\backslash\slash1&0 & 0 & f_{1,0}^{0,1}(1/z) & 0 & 0 & 0 & f_{10,0}^{1,1}(1/z) & 0 & 0\cr
			10\backslash\slash0&0 & 0 & 0 & 0 & 0 & f_{10,1}^{1,10}(1/z) & 0 & f_{10,0}^{1,1}(1/z) & 0\cr
			10\backslash\slash10&0 & 0 & 0 & 0 & 0 & 0 & 0 & 0 & 1
		}
	\end{align*}
	\begin{align*}
		&R_{\text{r,r}}(\beta,z)=
		\begin{tikzpicture}[baseline=(current  bounding  box.center),scale=0.75]
			\draw[red, thick,arrow=.27,rounded corners,d]  (-0.5,0.5) node[black,above] {$\ss z_1$} -- (0.5,-0.5) ;
			\draw[red, thick,arrow=.27,rounded corners,d] (0.5,0.5) node[black,above] {$\ss z_2$} -- (-0.5,-0.5) ;
		\end{tikzpicture}= 
		\\&
		\bordermatrix{
			&1\slash\backslash1 & 1\slash\backslash0 & 1\slash\backslash10 & 0\slash\backslash1 & 0\slash\backslash0 & 0\slash\backslash10 & 10\slash\backslash1 & 10\slash\backslash0 & 10\slash\backslash10\cr
			1\backslash\slash1& 1 & 0 & 0 & 0 & 0 & 0 & 0 & 0 & 0\cr
			1\backslash\slash0&0 & f_{1,1}^{0,10}(1/z) & 0 & f_{10,1}^{1,10}(1/z)& 0 & 0 & 0 & 0 & 0\cr
			1\backslash\slash10&0 & 0 & f_{10,0}^{1,1}(1/z) & 0 & 0 & 0 & f_{10,1}^{1,10}(1/z)& 0 & 0\cr
			0\backslash\slash1&0 & f_{1,0}^{0,1}(1/z) & 0 & f_{10,0}^{1,1}(1/z) & 0 & 0 & 0 & 0 & 0 \cr
			0\backslash\slash0&0 & 0 & 0 & 0 & 1 & 0 & 0 & 0 & 0\cr
			0\backslash\slash10&0 & 0 & 0 & 0 & 0 & f_{10,0}^{1,1}(1/z) & 0 & f_{1,0}^{0,1}(1/z) & 0\cr
			10\backslash\slash1&0 & 0 & f_{1,0}^{0,1}(1/z) & 0 & 0 & 0 & f_{1,1}^{0,10}(1/z) & 0 & 0\cr
			10\backslash\slash0&0 & 0 & 0 & 0 & 0 & f_{10,1}^{1,10}(1/z) & 0 & Q(f_{1,1}^{0,10}(1/z))& 0\cr
			10\backslash\slash10&0 & 0 & 0 & 0 & 0 & 0 & 0 & 0 & 1
		}
	\end{align*}
	\begin{align*}
		&R_{\text{b,b}}(\beta,z)=
		\begin{tikzpicture}[baseline=(current  bounding  box.center),scale=0.75]
			\draw[blue, thick,arrow=.27,rounded corners,d]  (-0.5,0.5) node[black,above] {$\ss z_1$} -- (0.5,-0.5) ;
			\draw[blue, thick,arrow=.27,rounded corners,d] (0.5,0.5) node[black,above] {$\ss z_2$} -- (-0.5,-0.5) ;
		\end{tikzpicture}=
		\\&
		\bordermatrix{
			&1\slash\backslash1 & 1\slash\backslash0 & 1\slash\backslash10 & 0\slash\backslash1 & 0\slash\backslash0 & 0\slash\backslash10 & 10\slash\backslash1 & 10\slash\backslash0 & 10\slash\backslash10\cr
			1\backslash\slash1& 1 & 0 & 0 & 0 & 0 & 0 & 0 & 0 & 0\cr
			1\backslash\slash0&0 & f_{1,1}^{0,10}(1/z) & 0 & f_{10,1}^{1,10}(1/z)& 0 & 0 & 0 & 0 & 0\cr
			1\backslash\slash10&0 & 0 & f_{1,1}^{0,10}(1/z)& 0 & 0 & 0 & f_{10,1}^{1,10}(1/z) & 0 & 0\cr
			0\backslash\slash1&0 & f_{1,0}^{0,1}(1/z)& 0 & f_{10,0}^{1,1}(1/z) & 0 & 0 & 0 & 0 & 0 \cr
			0\backslash\slash0&0 & 0 & 0 & 0 & 1 & 0 & 0 & 0 & 0\cr
			0\backslash\slash10&0 & 0 & 0 & 0 & 0 & f_{10,0}^{1,1}(1/z) & 0 & f_{1,0}^{0,1}(1/z) & 0\cr
			10\backslash\slash1&0 & 0 & f_{1,0}^{0,1}(1/z) & 0 & 0 & 0 & f_{10,0}^{1,1}(1/z)& 0 & 0\cr
			10\backslash\slash0&0 & 0 & 0 & 0 & 0 & f_{10,1}^{1,10}(1/z) & 0 & f_{1,1}^{0,10}(1/z) & 0\cr
			10\backslash\slash10&0 & 0 & 0 & 0 & 0 & 0 & 0 & 0 & 1
		}
	\end{align*}
	
	\begin{rmk}\label{rmk:submatrix}
		Observe that the matrix (\ref{eqn:latticemodel}) defining the lattice model for the $\overline{S}_\lambda$ of \cref{section:solvable-lattice-model} is a submatrix of the $R_{g,g}(\beta,z)$, $R_{r,r}(\beta,z)$, $R_{b,b}(\beta,z)$ matrices.
	\end{rmk}
	
	Finally, consider the following two matrices, which appear in \cref{lem:factorization}. These matrices record the fugacities of the triangle tiles.
	\begin{align*}
		D(\beta)=\begin{tikzpicture}[baseline=(current  bounding  box.center),scale=0.75]
			\draw[red, thick,arrow=.67,rounded corners,d] (0,0) to (-0.5,-0.5) node[black,left] {$\ss q^{-1}z$};
			\draw[green, thick,arrow=.67,rounded corners,d] (0,0) to (0.5,-0.5) node[black,right] {$\ss qz$} ;
			\draw[blue, thick,arrow=.67,rounded corners,d] (0,0.5) node[black,above] {$\ss z$} to (0,0);
		\end{tikzpicture}=
		\bordermatrix{
			&1- & 0- & 10-\cr
			1\backslash\slash1&1 & 0 & 0 \cr
			1\backslash\slash0&0 & 0 & 0\cr
			1\backslash\slash10&0 & 1 & 0 \cr
			0\backslash\slash1&0 & 0 & 1 \cr
			0\backslash\slash0&0 & 1 & 0 \cr
			0\backslash\slash10&0 & 0 & 0 \cr
			10\backslash\slash1&0 & 0 & 0\cr
			10\backslash\slash0&1 & 0 & 0\cr
			10\backslash\slash10&0 & 0 & -q
		}
	\end{align*}
	\begin{align*}
		U(\beta)=\begin{tikzpicture}[baseline=(current  bounding  box.center),scale=0.75]
			\draw[red, thick,arrow=.67,rounded corners,d] (0.5,0.5) node[black,right] {$\ss  q^{-1}z$} to (0,0);
			\draw[green, thick,arrow=.67,rounded corners,d] (-0.5,0.5) node[black, left] {$\ss qz$} to (0,0);
			\draw[blue, thick,arrow=.67,rounded corners,d]  (0,0) to (0,-0.5) node[black,below] {$\ss z$};
		\end{tikzpicture}=
		\bordermatrix{
			&1\slash\backslash1 & 1\slash\backslash0 & 1\slash\backslash10 & 0\slash\backslash1 & 0\slash\backslash0 & 0\slash\backslash10 & 10\slash\backslash1 & 10\slash\backslash0 & 10\slash\backslash10\cr
			1-&1 & 0 & 0 & 0 & 0 & 1 & 0 & 0 & 0\cr
			0-&0 & 0 & 0 & 0 & 1 & 0 & 1 & 0 & 0\cr
			10-&0 & 1 & 0 & 0 & 0 & 0 & 0 & 0 & \tfrac{-Q(\beta,q)}{q}\cr
		}
	\end{align*}

	The main result of this section will be \cref{thm:main}, which says that the Littlewood-Richardson coefficient $\overline{c}_{\lambda,\mu}^\nu$ is a sum of products of the fugacities of the \textit{rhombus} puzzle pieces defined in \cref{subsection:fugacities}. The key tool in proving the main theorem is \cref{prop:ybe}, which says that the $R$, $U$, and $D$ matrices defined in this section satisfy various Yang-Baxeter type equations. The following result, \cref{lem:factorization}, together with \cref{thm:main} imply that the evaluation of the coefficient $\overline{c}_{\lambda,\mu}^\nu$, by setting all $e^{\pm y_i}=0$, is a sum of products of the fugacities of the \textit{triangle} puzzle pieces defined in \cref{subsection:fugacities}. In particular, $\overline{c}_{\lambda,\mu}^\nu$ evaluated at all $e^{\pm y_i}=0$ can be computed as the sum over the fugacities of Knutson-Tau puzzles tiled solely by \textit{triangle} puzzles pieces. When $\beta=1$, this means that, non-equivariantly, the structure constants in the Schubert basis in $K(\mrm{Gr}(k,n))$ can be computed using only the triangle puzzle pieces to tile a Knutson-Tao puzzle.
	
	\begin{lem}\label{lem:factorization}
		The matrix $R_{{\text{g,r}}}(\beta,0)$ has rank $3$, and there is a factorization $R_{{\text{g,r}}}(\beta,0)=D(\beta)U(\beta)$.
	\end{lem}
	\begin{proof}
		This is a straightforward calculation.
	\end{proof}

	\begin{prop}\label{prop:ybe} (cf. \cite[Prop. 3.4]{KZJ21})
		In the pictures below, black lines can have arbitrary (independent)
		colors. Moreover, all green lines have spectral parameter $q z_i$, all red lines have spectral parameter $q^{-1}z_i$, and all blue lines have spectral parameter $z_i$.
		\begin{itemize}
			\item Yang--Baxter equations:
			\begin{equation}\label{eq:ybe}
				\begin{tikzpicture}[baseline=-3pt,y=2cm]
					\draw[d,arrow=0.1,arrow=0.4,arrow=0.7,rounded corners=4mm] (-0.5,0.5) node[black,above] {} -- (0.75,0) -- (1.5,-0.5)  (0.5,0.5) node[black,above] {} -- (0.2,0) -- (0.5,-0.5) (1.5,0.5) node[black,above] {} -- (0.75,0) -- (-0.5,-0.5);
				\end{tikzpicture}
				=
				\begin{tikzpicture}[baseline=-3pt,y=2cm]
					\draw[d,arrow=0.1,arrow=0.4,arrow=0.7,rounded corners=4mm] (-0.5,0.5)node[black,above] {} -- (0.25,0) -- (1.5,-0.5) (0.5,0.5) node[black,above] {} -- (0.8,0) -- (0.5,-0.5) (1.5,0.5) node[black,above] {} -- (0.25,0) -- (-0.5,-0.5);
				\end{tikzpicture}
			\end{equation}
			\item Bootstrap equations: 
			\begin{align}\label{eq:qtri}
				&\tikz[baseline=0]{
					\draw[invarrow=0.5,dg] (0,0) -- (-1,1) node[black,right] {$\ss qz_1$};
					\draw[invarrow=0.5,dr]  (0,0) -- (1,1) node[black,left] {$\ss q^{-1}z_1$};
					\draw[invarrow=0.5,db] (0,-1.5) node[black,below] {$\ss z_1$} -- (0,0)   node[triv] {};
					\draw[invarrow=0.5,rounded corners,d] (-1,-1.5) -- (-1,-0.5) -- (1.5,0) -- (1.5,1) node[black,above] {};
				}
				=
				\tikz[baseline=0]{
					\draw[invarrow=0.6,dg]  (0,-0.5) -- (-1,1) node[black,right] {$\ss qz_1$};
					\draw[invarrow=0.75,dr] (0,-0.5) -- (1,1) node[black,left] {$\ss q^{-1}z_1$};
					\draw[invarrow=0.6,db] (0,-1.5) node[black,below] {$\ss z_1$} -- (0,-0.5) node[triv] {};
					\draw[invarrow=0.6,rounded corners,d] (-1,-1.5) -- (-1,0) -- (1.5,0.5) -- (1.5,1) node[black,above] {};
				}
				&
				&\tikz[baseline=0,xscale=-1]{
					\draw[invarrow=0.5,dr] (0,0) -- (-1,1) node[black,left] {$\ss q^{-1}z_2$};
					\draw[invarrow=0.5,dg] (0,0) -- (1,1) node[black,right] {$\ss qz_2$};
					\draw[invarrow=0.5,db] (0,-1.5) node[black,below] {$\ss z_2$} -- (0,0) node[triv] {};
					\draw[invarrow=0.5,rounded corners,d] (-1,-1.5) -- (-1,-0.5) -- (1.5,0) -- (1.5,1) node[black,above] {};
				}
				=
				\tikz[baseline=0,xscale=-1]{
					\draw[invarrow=0.6,dr]  (0,-0.5) -- (-1,1) node[black,left] {$\ss q^{-1}z_2$};
					\draw[invarrow=0.75,dg] (0,-0.5) -- (1,1) node[black,right] {$\ss qz_2$};
					\draw[invarrow=0.6,db] (0,-1.5) node[black,below] {$\ss z_2$} -- (0,-0.5) node[triv] {};
					\draw[invarrow=0.6,rounded corners,d] (-1,-1.5) -- (-1,0) -- (1.5,0.5) -- (1.5,1) node[black,above] {};
				}
				\\[4mm]\label{eq:qtrirev}
				&\tikz[baseline=0,scale=-1]{
					\draw[arrow=0.5,dg]  (0,0) -- (-1,1) node[black,left] {$\ss qz_1$};
					\draw[arrow=0.5,dr] (0,0) -- (1,1) node[black,right] {$\ss q^{-1}z_1$};
					\draw[arrow=0.5,db] (0,-1.5) node[black,above] {$\ss z_1$} -- (0,0) node[triv] {};
					\draw[arrow=0.5,rounded corners,d] (-1,-1.5) node[black,above] {} -- (-1,-0.5) -- (1.5,0) -- (1.5,1) ;
				}
				=
				\tikz[baseline=0,scale=-1]{
					\draw[arrow=0.6,dg] (0,-0.5) -- (-1,1) node[black,left] {$\ss z_1$};
					\draw[arrow=0.75,dr] (0,-0.5) -- (1,1) node[black,right] {$\ss q^{-1}z_1$};
					\draw[arrow=0.6,db] (0,-1.5) node[black,above] {$\ss z_1$} -- (0,-0.5)  node[triv] {};
					\draw[arrow=0.6,rounded corners,d] (-1,-1.5) node[black,above] {} -- (-1,0) -- (1.5,0.5) -- (1.5,1);
				}&
				&\tikz[baseline=0,yscale=-1]{
					\draw[arrow=0.5,dr] (0,0) -- (-1,1) node[black,right] {$\ss q^{-1}z_2$};
					\draw[arrow=0.5,dg] (0,0) -- (1,1) node[black,left] {$\ss qz_2$};
					\draw[arrow=0.5,db] (0,-1.5) node[black,above] {$\ss z_2$} -- (0,0) node[triv] {};
					\draw[arrow=0.5,rounded corners,d] (-1,-1.5) node[black,above] {} -- (-1,-0.5) -- (1.5,0) -- (1.5,1);
				}
				=
				\tikz[baseline=0,yscale=-1]{
					\draw[arrow=0.6,dr] (0,-0.5) -- (-1,1) node[black,right] {$\ss q^{-1}z_2$};
					\draw[arrow=0.75,dg] (0,-0.5) -- (1,1) node[black,left] {$\ss qz_2$};
					\draw[arrow=0.6,db] (0,-1.5) node[black,above] {$\ss z_2$} -- (0,-0.5) node[triv] {};
					\draw[arrow=0.6,rounded corners,d] (-1,-1.5) node[black,above] {} -- (-1,0) -- (1.5,0.5) -- (1.5,1);
				}
			\end{align}
			\item Unitarity equations: 
			\begin{equation}\label{eq:unit}
				\begin{tikzpicture}[baseline=-3pt]
					\draw[arrow=0.07,arrow=0.57,rounded corners,d] (-0.5,1) node[black,above] {} -- (0.5,0) -- (-0.5,-1) (0.5,1) node[black,above] {} -- (-0.5,0) -- (0.5,-1);
				\end{tikzpicture}
				=
				\begin{tikzpicture}[baseline=-3pt]
					\draw[arrow=0.1,arrow=0.6,rounded corners,d] (-0.5,1) node[black,above] {} -- (-0.5,-1) (0.5,1) node[black,above] {}-- (0.5,-1);
				\end{tikzpicture}
			\end{equation}
			\item Value at equal spectral parameters: (here, the lines {\em must}\/ have the same color
			for the equality to make sense)
			\begin{equation}\label{eq:equal}
				\begin{tikzpicture}[baseline=-3pt,yscale=1.5]
					\draw[invarrow=0.3,d] (-0.5,-0.5) -- node[left,pos=0.3] {} (0.5,0.5);
					\draw[invarrow=0.3,d] (0.5,-0.5) -- node[right,pos=0.3] {} (-0.5,0.5);
				\end{tikzpicture}
				=
				\begin{tikzpicture}[baseline=-3pt,yscale=1.5]
					\draw[invarrow=0.3,rounded corners=4mm,d] (-0.5,-0.5) -- node[left] {} (0,0) -- (-0.5,0.5);
					\draw[invarrow=0.3,rounded corners=4mm,d] (0.5,-0.5) -- node[right] {} (0,0) -- (0.5,0.5);
				\end{tikzpicture}
			\end{equation}
		\end{itemize}
	\end{prop}
	\begin{proof}
		These are a straightforward calculations, verified using SageMath \cite{G25}. See \cref{subsection:the-equations} for the list of all the equations that were verified.
	\end{proof}

	\subsection{Proof of the puzzle formula}\label{subsection:proof}
	In this subsection, we will prove the puzzle formula for multiplying the classes $q^{\ell(\lambda)}S_\lambda$, closely following the strategy used in \cite{KZJ17}, \cite{KZJ21}. We will show the formula is positive in \cref{rmk:positive}, and we will give several examples of computations of structure constants in terms of puzzles in this subsection.
	
	We will use the notation $\tikz[scale=1,baseline=0cm]{\uptri{\lambda}{\nu}{\mu}}$ to mean the sum of the fugacities over all possible puzzles with the prescribed boundary labels. For example, for
	$\lambda = 0101$, $\mu = 0101$, $\nu = 0110$, we have
	\[\tikz[scale=1,baseline=0cm]{\uptri{0101}{0110}{0101}}:=\sum_{L_1,\dots,L_{12}\in \{0,1,10\}}
	\begin{tikzpicture}[math mode,nodes={\mcol},x={(-0.577cm,-1cm)},y={(0.577cm,-1cm)},baseline=(current  bounding  box.center)]
		\draw[thick] (0,0) -- node[pos=0.5] {\ss 0} ++(0,1); \draw[thick] (0,0) -- node[pos=0.5] {\ss 1} ++(1,0);
		\draw[thick] (0,1) -- node[pos=0.5] {\ss 1} ++(0,1); \draw[thick] (0,1) -- node[pos=0.5] {\ss L_1} ++(1,0); 
		\draw[thick] (0,2) -- node[pos=0.5] {\ss 0} ++(0,1); \draw[thick] (0,2) -- node[pos=0.5] {\ss L_2} ++(1,0); 
		\draw[thick] (0,3) -- node[pos=0.5] {\ss 1} ++(0,1); \draw[thick] (0,3) -- node[pos=0.5] {\ss L_3} ++(1,0); \draw[thick] (0+1,3) -- node {\ss 0} ++(-1,1); 
		\draw[thick] (1,0) -- node[pos=0.5] {\ss L_4} ++(0,1); \draw[thick] (1,0) -- node[pos=0.5] {\ss 0} ++(1,0);
		\draw[thick] (1,1) -- node[pos=0.5] {\ss L_5} ++(0,1); \draw[thick] (1,1) -- node[pos=0.5] {\ss L_6} ++(1,0);
		\draw[thick] (1,2) -- node[pos=0.5] {\ss L_7} ++(0,1); \draw[thick] (1,2) -- node[pos=0.5] {\ss L_8} ++(1,0); \draw[thick] (1+1,2) -- node {\ss 1} ++(-1,1); 
		\draw[thick] (2,0) -- node[pos=0.5] {\ss L_9} ++(0,1); \draw[thick] (2,0) -- node[pos=0.5] {\ss 1} ++(1,0);
		\draw[thick] (2,1) -- node[pos=0.5] {\ss L_{10}} ++(0,1); \draw[thick] (2,1) -- node[pos=0.5] {\ss L_{11}} ++(1,0); \draw[thick] (2+1,1) -- node {\ss 1} ++(-1,1); 
		\draw[thick] (3,0) -- node[pos=0.5] {\ss L_{12}} ++(0,1); \draw[thick] (3,0) -- node[pos=0.5] {\ss 0} ++(1,0); \draw[thick] (3+1,0) -- node {\ss 0} ++(-1,1); 
	\end{tikzpicture}
	=
	\begin{tikzpicture}[math mode,nodes={\mcol},x={(-0.577cm,-1cm)},y={(0.577cm,-1cm)},baseline=(current  bounding  box.center)]
		\draw[dg,arrow=0.07] (0.5,0) node[above] {\ss 1} -- node[left,pos=0.07] {\ss q z_1} (0.5,3.5);
		\draw[dg,arrow=0.1] (1.5,0) node[above] {\ss 0} -- node[left,pos=0.1] {\ss qz_2} (1.5,2.5);
		\draw[dg,arrow=0.17] (2.5,0) node[above] {\ss 1} -- node[left,pos=0.17] {\ss qz_3} (2.5,1.5);
		\draw[dg,arrow] (3.5,0) node[above] {\ss 0} -- node[left] {\ss q z_4} (3.5,0.5);
		\draw[dr,arrow=0.07] (0,0.5) node[above] {\ss 0} -- node[right,pos=0.07] {\ss q^{-1} z_4} (3.5,0.5);
		\draw[dr,arrow=0.1] (0,1.5) node[above] {\ss 1} -- node[right,pos=0.1] {\ss q^{-1} z_3} (2.5,1.5);
		\draw[dr,arrow=0.17] (0,2.5) node[above] {\ss 0} -- node[right,pos=0.17] {\ss q^{-1} z_2} (1.5,2.5);
		\draw[dr,arrow] (0,3.5) node[above] {\ss 1} -- node[right] {\ss q^{-1} z_1} (0.5,3.5);
		\draw[db,arrow] (0.5,3.5) node[triv] {} -- node[right] {\ss z_1}  ++(0.25,0.25) node[below] {\ss 0};
		\draw[db,arrow] (1.5,2.5) node[triv] {} -- node[right] {\ss z_2} ++(0.25,0.25) node[below] {\ss 1};
		\draw[db,arrow] (2.5,1.5) node[triv] {} -- node[right] {\ss z_3} ++(0.25,0.25) node[below] {\ss 1};
		\draw[db,arrow] (3.5,0.5) node[triv] {} -- node[right] {\ss z_4} ++(0.25,0.25) node[below] {\ss 0};
	\end{tikzpicture}
	\]
	
	\begin{prop}(cf. \cite[Prop. 3.4]{KZJ17})\label{prop:unique-puzzle}
		\begin{enumerate}
			\item Single-color crossings preserve single numbers, i.e.
			\[\text{if }\begin{tikzpicture}[baseline=-3pt,yscale=1.5]
				\draw[invarrow=0.3,d] (-0.5,-0.5) node[black,below] {$\ss i$} -- node[left,pos=0.3] {} (0.5,0.5) node[black,above] {$\ss l$};
				\draw[invarrow=0.3,d] (0.5,-0.5) node[black,below] {$\ss j$} -- node[right,pos=0.3] {} (-0.5,0.5) node[black,above] {$\ss k$};
			\end{tikzpicture}\neq 0, \text{ then}\quad  i,j\in\{0,1\}\iff k,l\in\{0,1\}, \]
			where the lines have arbitrary, but identical, colors.
			\item For $i\in\{0,1\}$,
			\[\text{if }\begin{tikzpicture}[baseline=-3pt,yscale=1.5]
				\draw[invarrow=0.3,d] (-0.5,-0.5) node[black,below] {$\ss i$} -- node[left,pos=0.3] {} (0.5,0.5) node[black,above] {$\ss k$};
				\draw[invarrow=0.3,d] (0.5,-0.5) node[black,below] {$\ss i$} -- node[right,pos=0.3] {} (-0.5,0.5) node[black,above] {$\ss j$};
			\end{tikzpicture}\neq 0 \quad \text{   or } \quad \begin{tikzpicture}[baseline=-3pt,yscale=1.5]
				\draw[invarrow=0.3,d] (-0.5,-0.5) node[black,below] {$\ss j$} -- node[left,pos=0.3] {} (0.5,0.5) node[black,above] {$\ss i$};
				\draw[invarrow=0.3,d] (0.5,-0.5) node[black,below] {$\ss k$} -- node[right,pos=0.3] {} (-0.5,0.5) node[black,above] {$\ss i$};
			\end{tikzpicture}\neq0 \quad \implies \quad i=j=k,\]
			where again the lines have identical colors.
			\item Recall the weakly-increasing string $\omega=0^{n-i_1}1^{i_1}$. We have \[\uptri\lambda\omega\mu\neq 0\implies \lambda=\mu=\omega\]
			for all single-number strings $\lambda$, $\mu$. Moreover, the puzzle $\uptri\omega\omega\omega$ is unique and has fugacity $1$.	
		\end{enumerate}
	\end{prop}
	\begin{proof}
		By inspection, the $(i,j)$-entry in a single-color $R$-matrix, $R_{g,g}(\beta,z)$, $R_{r,r}(\beta,z)$, or $R_{b,b}(\beta,z)$, is nonzero if and only if the $(i,j)$-entry in the respective single-color $R$-matrix evaluated at $\beta=1$, $R_{g,g}(1,z)$, $R_{r,r}(1,z)$, $R_{b,b}(1,z)$, is nonzero. This result holds when $\beta=1$ by \cite[Prop. 3.4]{KZJ17}, so the result for general $\beta$ holds as well.
	\end{proof}
	
	We can now state the main result of this section:
	\begin{theo}\label{thm:main}
		The product of two classes $S_\lambda$ and $S_\mu$ in $\widetilde{R}_\Lambda^{\mrm{loc}}$ is given by the ``puzzle'' formula
		\begin{equation}\label{eq:main}
			(q^{\ell(\lambda)}S_\lambda)(q^{\ell(\mu)} S_\mu)=
			\sum_\nu \tikz[scale=1.8,baseline=0.5cm]{\uptri{\lambda}{\nu}{\mu}}
			\ (q^{\ell(\nu)}S_\nu)  
		\end{equation}
	\end{theo}
	\begin{proof}
		Since Proposition \ref{prop:ybe} and Proposition \ref{prop:unique-puzzle} hold, the proof of this theorem is identical to that of \cite[Theorem 3.8]{KZJ21}. 
	\end{proof}
	
	\begin{rmk}\label{rmk:positive}
		Following \cite[\S 6]{KZJ21}, we define a \textbf{positivity monoid} to be a submonoid $M$ of an abelian group, such that $M\cap (-M)=0$. Consider the submonoid $M$ of $K_{\widehat{T}}^{\text{loc}}(\mrm{pt})[\beta,\beta^{-1}]$ under addition, defined as the set of sums of products of the factors
		\[-q^{\pm}\qquad Q(\beta,q)\qquad  e^{-\alpha}\qquad \tfrac{\beta(1-q^2)}{Q(\beta,q)-q^2e^{-\alpha}}\qquad -\tfrac{1-e^{-\alpha}}{Q(\beta,q)-q^2e^{-\alpha}}\]
		over all $\alpha\in\Sigma_\Theta^+$. Then $M$ is a positivity monoid. To see this, we apply \cite[Lemma 6.1]{KZJ21}: note that after evaluating $\beta=1$, $e^{y_i}=2^i$, and $q=-2^{-n/2}$, every factor is a positive real number. Thus, $M\cap (-M)=0$. 
		
		As the structure constants $\overline{c}_{\lambda,\mu}^\nu$ lie in the positivity monoid $M$, it is in this sense that we consider the structure constants and the puzzle formula for them \textbf{positive}.
	\end{rmk}
	
	\begin{rmk}
		In \cref{section:localization-techniques}, we will show that the elements $S_\lambda$ arise from canonical elements in a quotient of the $\widehat{T}$-equivariant algebraic cobordism ring of $T^*(G/P_\Theta)$. Although \cref{rmk:positive} shows that the structure constants in the $\beta$-deformed basis are positive, we currently do not have a geometric understanding of the positivity from the point of view of algebraic cobordism or connective $K$-theory.
	\end{rmk}

	We will now give several examples of computing the classes $q^{\ell(\lambda)}S_\lambda$ and computing the Littlewood-Richardson coefficients in terms of puzzles. Recall that, by the definition given in \cref{subsection:motive-Segre}, a deformed motivic Segre class $q^{\ell(\lambda)}S_\lambda$ is a sequence in $\oplus_{\lambda\in W^\Theta} K_{\widehat{T}}(\mrm{pt})[\beta,\beta^{-1}]$, and that all classes $q^{\ell(\lambda)}S_\lambda$ can be obtained from $q^{\ell(w_0)}S_{w_0}$ by applying a sequence of deformed divided difference operators to $q^{\ell(w_0)}S_{w_0}$.
	
	\begin{egg}\label{egg:1}
		There is exactly one class for each of $T^*(\mrm{Gr}(0,2))$ and $T^*(\mrm{Gr}(2,2))$. For $T^*(\mrm{Gr}(0,2))$, we have $S_{11}=1$, and for $T^*(\mrm{Gr}(2,2))$, we have $S_{00}=1$. Thus, $S_{00}^2=1\cdot S_{00}$ and $S_{11}^2=1\cdot S_{11}$, and indeed
		\[\def\posa{0.5}\def\posb{0.5}\def\thescale{1}
		\begin{tikzpicture}[math mode,nodes={edgelabel},x={(-0.577cm,-1cm)},y={(0.577cm,-1cm)},scale=\thescale,baseline=(current  bounding  box.center)]
			\draw[thick] (2,0) -- node[pos=0.5] {\ss 0} ++(0,1); \draw[thick] (2,0) -- node[pos=0.5] {\ss 0} ++(1,0); 
			\draw[thick] (2,1) -- node[pos=0.5] {\ss 0} ++(0,1); \draw[thick] (2,1) -- node[pos=0.5] {\ss 0} ++(1,0); \draw[thick] (2+1,1) -- node {\ss 0} ++(-1,1); 
			\draw[thick] (3,0) -- node[pos=0.5] {\ss 0} ++(0,1); \draw[thick] (3,0) -- node[pos=0.5] {\ss 0} ++(1,0); \draw[thick] (3+1,0) -- node {\ss 0} ++(-1,1); 	
		\end{tikzpicture}
		=1\quad\quad\text{and}\quad\quad
		\def\posa{0.5}\def\posb{0.5}\def\thescale{1}
		\begin{tikzpicture}[math mode,nodes={edgelabel},x={(-0.577cm,-1cm)},y={(0.577cm,-1cm)},scale=\thescale,baseline=(current  bounding  box.center)]
			\draw[thick] (2,0) -- node[pos=0.5] {\ss 1} ++(0,1); \draw[thick] (2,0) -- node[pos=0.5] {\ss 1} ++(1,0); 
			\draw[thick] (2,1) -- node[pos=0.5] {\ss 1} ++(0,1); \draw[thick] (2,1) -- node[pos=0.5] {\ss 1} ++(1,0); \draw[thick] (2+1,1) -- node {\ss 1} ++(-1,1); 
			\draw[thick] (3,0) -- node[pos=0.5] {\ss 1} ++(0,1); \draw[thick] (3,0) -- node[pos=0.5] {\ss 1} ++(1,0); \draw[thick] (3+1,0) -- node {\ss 1} ++(-1,1); 	
		\end{tikzpicture}
		=1
		\]
	\end{egg}
	\begin{egg}\label{egg:2}
		We will first compute the classes $q\cdot S_{10}$  and $S_{01}$ in $T^*(\mrm{Gr}(1,2))$. There is exactly one root $\alpha$ in $\Sigma_\Theta^+$. The classes are
		\[q\cdot S_{10}=q\cdot \left[0,\tfrac{1-e^{-\alpha}}{Q(\beta,q)-q^2e^{-\alpha}}\right];\quad\quad\quad S_{01}=(\tfrac{1}{q}\partial_\alpha)(q\cdot S_{10})=\left[1,\tfrac{\beta(1-q^2)}{Q(\beta,q)-q^2e^{-\alpha}}\right].\]
		Next, we will expand the products of the classes in the basis $\{q^{\ell(\lambda)}S_\lambda\}$:
		\[(q\cdot S_{10})^2=\tfrac{q(1-e^{-\alpha})}{Q(\beta,q)-q^2e^{-\alpha}}(q\cdot S_{10});\quad (q\cdot S_{10})S_{01}=S_{01}(q\cdot S_{10})=\tfrac{\beta(1-q^2)}{Q(\beta,q)-q^2e^{-\alpha}}(q\cdot S_{10});	\]
		\[S_{01}^2=S_{01}+\tfrac{\beta q(q^2-1)}{Q(\beta,q)-q^2e^{-\alpha}}(q\cdot S_{10}).\]
		Indeed, puzzles give the same answers:
		\[
		\def\posa{0.5}\def\posb{0.5}\def\thescale{1}
		\begin{tikzpicture}[math mode,nodes={edgelabel},x={(-0.577cm,-1cm)},y={(0.577cm,-1cm)},scale=\thescale,baseline=(current  bounding  box.center)]
			\draw[thick] (2,0) -- node[pos=0.5] {\ss 1} ++(0,1); \draw[thick] (2,0) -- node[pos=0.5] {\ss 0} ++(1,0); 
			\draw[thick] (2,1) -- node[pos=0.5] {\ss 0} ++(0,1); \draw[thick] (2,1) -- node[pos=0.5] {\ss 0} ++(1,0); \draw[thick] (2+1,1) -- node {\ss 0} ++(-1,1); 
			\draw[thick] (3,0) -- node[pos=0.5] {\ss 1} ++(0,1); \draw[thick] (3,0) -- node[pos=0.5] {\ss 1} ++(1,0); \draw[thick] (3+1,0) -- node {\ss 1} ++(-1,1); 	
		\end{tikzpicture}
		=\frac{q(1-e^{-\alpha})}{Q(\beta,q)-q^2e^{-\alpha}}\quad\quad
		\def\posa{0.5}\def\posb{0.5}\def\thescale{1}
		\begin{tikzpicture}[math mode,nodes={edgelabel},x={(-0.577cm,-1cm)},y={(0.577cm,-1cm)},scale=\thescale,baseline=(current  bounding  box.center)]
			\draw[thick] (2,0) -- node[pos=0.5] {\ss 0} ++(0,1); \draw[thick] (2,0) -- node[pos=0.5] {\ss 0} ++(1,0); 
			\draw[thick] (2,1) -- node[pos=0.5] {\ss 1} ++(0,1); \draw[thick] (2,1) -- node[pos=0.5] {\ss 10} ++(1,0); \draw[thick] (2+1,1) -- node {\ss 0} ++(-1,1); 
			\draw[thick] (3,0) -- node[pos=0.5] {\ss 1} ++(0,1); \draw[thick] (3,0) -- node[pos=0.5] {\ss 1} ++(1,0); \draw[thick] (3+1,0) -- node {\ss 1} ++(-1,1); 	
		\end{tikzpicture}
		=\frac{\beta(1-q^2)}{Q(\beta,q)-q^2e^{-\alpha}}\]
		\[
		\def\posa{0.5}\def\posb{0.5}\def\thescale{1}
		\begin{tikzpicture}[math mode,nodes={edgelabel},x={(-0.577cm,-1cm)},y={(0.577cm,-1cm)},scale=\thescale,baseline=(current  bounding  box.center)]
			\draw[thick] (2,0) -- node[pos=0.5] {\ss 1} ++(0,1); \draw[thick] (2,0) -- node[pos=0.5] {\ss 1} ++(1,0); 
			\draw[thick] (2,1) -- node[pos=0.5] {\ss 0} ++(0,1); \draw[thick] (2,1) -- node[pos=0.5] {\ss 0} ++(1,0); \draw[thick] (2+1,1) -- node {\ss 0} ++(-1,1); 
			\draw[thick] (3,0) -- node[pos=0.5] {\ss 10} ++(0,1); \draw[thick] (3,0) -- node[pos=0.5] {\ss 0} ++(1,0); \draw[thick] (3+1,0) -- node {\ss 1} ++(-1,1); 	
		\end{tikzpicture}
		=\frac{\beta(1-q^2)}{Q(\beta,q)-q^2e^{-\alpha}}\quad\quad
		\def\posa{0.5}\def\posb{0.5}\def\thescale{1}
		\begin{tikzpicture}[math mode,nodes={edgelabel},x={(-0.577cm,-1cm)},y={(0.577cm,-1cm)},scale=\thescale,baseline=(current  bounding  box.center)]
			\draw[thick] (2,0) -- node[pos=0.5] {\ss 0} ++(0,1); \draw[thick] (2,0) -- node[pos=0.5] {\ss 1} ++(1,0); 
			\draw[thick] (2,1) -- node[pos=0.5] {\ss 1} ++(0,1); \draw[thick] (2,1) -- node[pos=0.5] {\ss 1} ++(1,0); \draw[thick] (2+1,1) -- node {\ss 1} ++(-1,1); 
			\draw[thick] (3,0) -- node[pos=0.5] {\ss 0} ++(0,1); \draw[thick] (3,0) -- node[pos=0.5] {\ss 0} ++(1,0); \draw[thick] (3+1,0) -- node {\ss 0} ++(-1,1); 	
		\end{tikzpicture}
		=1\quad\quad
		\def\posa{0.5}\def\posb{0.5}\def\thescale{1}
		\begin{tikzpicture}[math mode,nodes={edgelabel},x={(-0.577cm,-1cm)},y={(0.577cm,-1cm)},scale=\thescale,baseline=(current  bounding  box.center)]
			\draw[thick] (2,0) -- node[pos=0.5] {\ss 0} ++(0,1); \draw[thick] (2,0) -- node[pos=0.5] {\ss 1} ++(1,0); 
			\draw[thick] (2,1) -- node[pos=0.5] {\ss 1} ++(0,1); \draw[thick] (2,1) -- node[pos=0.5] {\ss 10} ++(1,0); \draw[thick] (2+1,1) -- node {\ss 0} ++(-1,1); 
			\draw[thick] (3,0) -- node[pos=0.5] {\ss 10} ++(0,1); \draw[thick] (3,0) -- node[pos=0.5] {\ss 0} ++(1,0); \draw[thick] (3+1,0) -- node {\ss 1} ++(-1,1); 	
		\end{tikzpicture}
		=\frac{\beta q(q^2-1)}{Q(\beta,q)-q^2e^{-\alpha}}
		\]
	\end{egg}
	
	\begin{egg}\label{egg:3}
		We first compute the classes $q^2\cdot S_{100}$, $q\cdot S_{010}$, and $S_{001}$ for $T^*(\mrm{Gr}(1,3))$. There are exactly two simple roots $\alpha_1,\alpha_2$ in $\Sigma_\Theta^+.$ The classes are
		\[q^2\cdot S_{100}=q^2\cdot \left[0,0,\tfrac{1-e^{-\alpha_1}}{Q(\beta,q)-q^2e^{-\alpha_1}}\tfrac{1-e^{-\alpha_1-\alpha_2}}{Q(\beta,q)-q^2e^{-\alpha_1-\alpha_2}}\right];\]
		\[q\cdot S_{010}=(\tfrac{1}{q}\partial_{\alpha_1})(q^2\cdot S_{100})=q\cdot \left[0,\tfrac{1-e^{-\alpha_2}}{Q(\beta,q)-q^2e^{-\alpha_2}},\tfrac{\beta(1-q^2)}{Q(\beta,q)-q^2e^{-\alpha_1}}\tfrac{1-e^{-\alpha_1-\alpha_2}}{Q(\beta,q)-q^2e^{-\alpha_1-\alpha_2}}\right];\]\[ S_{001}=(\tfrac{1}{q}\partial_{\alpha_2})(q\cdot S_{010})=\left[1,\tfrac{\beta(1-q^2)}{Q(\beta,q)-q^2e^{-\alpha_2}},\tfrac{\beta(1-q^2)}{Q(\beta,q)-q^2e^{-\alpha_1-\alpha_2}}\right].\]
		Next, we expand the following products of the classes in the basis $\{q^{\ell(\lambda)}S_\lambda\}$:
		\[(q^2\cdot S_{100})S_{001}=S_{001}(q^2\cdot S_{100})=\tfrac{\beta(1-q^2)}{Q(\beta,q)-q^2e^{-\alpha_1-\alpha_2}}(q^2\cdot S_{100});\]
		\[S_{001}(q\cdot S_{010})=(q\cdot S_{010})S_{001}=\tfrac{\beta(1-q^2)}{Q(\beta,q)-q^2e^{-\alpha_2}}(q\cdot S_{010})+\tfrac{\beta^2q(1-q^2)^2(-e^{-\alpha_2})}{(Q(\beta,q)-q^2e^{-\alpha_1-\alpha_2})(Q(\beta,q)-q^2e^{-\alpha_2})}(q^2\cdot S_{100});\]
		\[S_{001}^2=S_{001}+\tfrac{\beta q(q^2-1)}{Q(\beta,q)-q^2e^{-\alpha_2}} (q\cdot S_{010})+\tfrac{\beta q^2Q(\beta,q)(q^2-1) (1-e^{-\alpha_2})}{(Q(\beta,q)-q^2e^{-\alpha_2})(Q(\beta,q)-q^2e^{-\alpha_1-\alpha_2})}(q^2\cdot S_{100}).
		\] 
		Indeed, these structure constants can be computed with puzzles as well:
		\[
		\def\posa{0.5}\def\posb{0.5}\def\thescale{1}
		\begin{tikzpicture}[math mode,nodes={edgelabel},x={(-0.577cm,-1cm)},y={(0.577cm,-1cm)},scale=\thescale,baseline=(current  bounding  box.center)]
			\draw[thick] (1,0) -- node[pos=0.5] {\ss 1} ++(0,1); \draw[thick] (1,0) -- node[pos=0.5] {\ss 1} ++(1,0); 
			\draw[thick] (1,1) -- node[pos=0.5] {\ss 0} ++(0,1); \draw[thick] (1,1) -- node[pos=0.5] {\ss 0} ++(1,0); 
			\draw[thick] (1,2) -- node[pos=0.5] {\ss 0} ++(0,1); \draw[thick] (1,2) -- node[pos=0.5] {\ss 0} ++(1,0); \draw[thick] (1+1,2) -- node {\ss 0} ++(-1,1); 
			\draw[thick] (2,0) -- node[pos=0.5] {\ss 10} ++(0,1); \draw[thick] (2,0) -- node[pos=0.5] {\ss 0} ++(1,0); 
			\draw[thick] (2,1) -- node[pos=0.5] {\ss 0} ++(0,1); \draw[thick] (2,1) -- node[pos=0.5] {\ss 0} ++(1,0); \draw[thick] (2+1,1) -- node {\ss 0} ++(-1,1); 
			\draw[thick] (3,0) -- node[pos=0.5] {\ss 10} ++(0,1); \draw[thick] (3,0) -- node[pos=0.5] {\ss 0} ++(1,0); \draw[thick] (3+1,0) -- node {\ss 1} ++(-1,1); 
		\end{tikzpicture}=\frac{\beta(1-q^2)}{Q(\beta,q)-q^2e^{-\alpha_1-\alpha_2}}\quad \def\posa{0.5}\def\posb{0.5}\def\thescale{1}
		\begin{tikzpicture}[math mode,nodes={edgelabel},x={(-0.577cm,-1cm)},y={(0.577cm,-1cm)},scale=\thescale,baseline=(current  bounding  box.center)]
			\draw[thick] (1,0) -- node[pos=0.5] {\ss 0} ++(0,1); \draw[thick] (1,0) -- node[pos=0.5] {\ss 1} ++(1,0); 
			\draw[thick] (1,1) -- node[pos=0.5] {\ss 1} ++(0,1); \draw[thick] (1,1) -- node[pos=0.5] {\ss 1} ++(1,0); 
			\draw[thick] (1,2) -- node[pos=0.5] {\ss 0} ++(0,1); \draw[thick] (1,2) -- node[pos=0.5] {\ss 0} ++(1,0); \draw[thick] (1+1,2) -- node {\ss 0} ++(-1,1); 
			\draw[thick] (2,0) -- node[pos=0.5] {\ss0} ++(0,1); \draw[thick] (2,0) -- node[pos=0.5] {\ss 0} ++(1,0); 
			\draw[thick] (2,1) -- node[pos=0.5] {\ss 10} ++(0,1); \draw[thick] (2,1) -- node[pos=0.5] {\ss 0} ++(1,0); \draw[thick] (2+1,1) -- node {\ss 1} ++(-1,1); 
			\draw[thick] (3,0) -- node[pos=0.5] {\ss 0} ++(0,1); \draw[thick] (3,0) -- node[pos=0.5] {\ss 0} ++(1,0); \draw[thick] (3+1,0) -- node {\ss 0} ++(-1,1); 
		\end{tikzpicture}=\frac{\beta(1-q^2)}{Q(\beta,q)-q^2 e^{-\alpha_2}}\]
		\[\def\posa{0.5}\def\posb{0.5}\def\thescale{1}
		\begin{tikzpicture}[math mode,nodes={edgelabel},x={(-0.577cm,-1cm)},y={(0.577cm,-1cm)},scale=\thescale,baseline=(current  bounding  box.center)]
			\draw[thick] (1,0) -- node[pos=0.5] {\ss 0} ++(0,1); \draw[thick] (1,0) -- node[pos=0.5] {\ss 1} ++(1,0); 
			\draw[thick] (1,1) -- node[pos=0.5] {\ss 1} ++(0,1); \draw[thick] (1,1) -- node[pos=0.5] {\ss 10} ++(1,0); 
			\draw[thick] (1,2) -- node[pos=0.5] {\ss 0} ++(0,1); \draw[thick] (1,2) -- node[pos=0.5] {\ss 0} ++(1,0); \draw[thick] (1+1,2) -- node {\ss 0} ++(-1,1); 
			\draw[thick] (2,0) -- node[pos=0.5] {\ss10} ++(0,1); \draw[thick] (2,0) -- node[pos=0.5] {\ss 0} ++(1,0); 
			\draw[thick] (2,1) -- node[pos=0.5] {\ss 0} ++(0,1); \draw[thick] (2,1) -- node[pos=0.5] {\ss 0} ++(1,0); \draw[thick] (2+1,1) -- node {\ss 0} ++(-1,1); 
			\draw[thick] (3,0) -- node[pos=0.5] {\ss 10} ++(0,1); \draw[thick] (3,0) -- node[pos=0.5] {\ss 0} ++(1,0); \draw[thick] (3+1,0) -- node {\ss 1} ++(-1,1); 
		\end{tikzpicture}=\frac{\beta^2 q(1-q^2)^2(-e^{-\alpha_2})}{(Q(\beta,q)-q^2e^{-\alpha_1-\alpha_2}) (Q(\beta,q)-q^2e^{-\alpha_2})}\]
		
		\[\def\posa{0.5}\def\posb{0.5}\def\thescale{1}
		\begin{tikzpicture}[math mode,nodes={edgelabel},x={(-0.577cm,-1cm)},y={(0.577cm,-1cm)},scale=\thescale,baseline=(current  bounding  box.center)]
			\draw[thick] (1,0) -- node[pos=0.5] {\ss 1} ++(0,1); \draw[thick] (1,0) -- node[pos=0.5] {\ss 0} ++(1,0); 
			\draw[thick] (1,1) -- node[pos=0.5] {\ss 1} ++(0,1); \draw[thick] (1,1) -- node[pos=0.5] {\ss 0} ++(1,0); 
			\draw[thick] (1,2) -- node[pos=0.5] {\ss 1} ++(0,1); \draw[thick] (1,2) -- node[pos=0.5] {\ss 1} ++(1,0); \draw[thick] (1+1,2) -- node {\ss 1} ++(-1,1); 
			\draw[thick] (2,0) -- node[pos=0.5] {\ss0} ++(0,1); \draw[thick] (2,0) -- node[pos=0.5] {\ss 0} ++(1,0); 
			\draw[thick] (2,1) -- node[pos=0.5] {\ss 0} ++(0,1); \draw[thick] (2,1) -- node[pos=0.5] {\ss 0} ++(1,0); \draw[thick] (2+1,1) -- node {\ss 0} ++(-1,1); 
			\draw[thick] (3,0) -- node[pos=0.5] {\ss 0} ++(0,1); \draw[thick] (3,0) -- node[pos=0.5] {\ss 0} ++(1,0); \draw[thick] (3+1,0) -- node {\ss 0} ++(-1,1); 
		\end{tikzpicture}=1\quad 
		\def\posa{0.5}\def\posb{0.5}\def\thescale{1}
		\begin{tikzpicture}[math mode,nodes={edgelabel},x={(-0.577cm,-1cm)},y={(0.577cm,-1cm)},scale=\thescale,baseline=(current  bounding  box.center)]
			\draw[thick] (1,0) -- node[pos=0.5] {\ss 1} ++(0,1); \draw[thick] (1,0) -- node[pos=0.5] {\ss 0} ++(1,0); 
			\draw[thick] (1,1) -- node[pos=0.5] {\ss 1} ++(0,1); \draw[thick] (1,1) -- node[pos=0.5] {\ss 0} ++(1,0); 
			\draw[thick] (1,2) -- node[pos=0.5] {\ss 10} ++(0,1); \draw[thick] (1,2) -- node[pos=0.5] {\ss 1} ++(1,0); \draw[thick] (1+1,2) -- node {\ss 0} ++(-1,1); 
			\draw[thick] (2,0) -- node[pos=0.5] {\ss0} ++(0,1); \draw[thick] (2,0) -- node[pos=0.5] {\ss 0} ++(1,0); 
			\draw[thick] (2,1) -- node[pos=0.5] {\ss 0} ++(0,1); \draw[thick] (2,1) -- node[pos=0.5] {\ss 10} ++(1,0); \draw[thick] (2+1,1) -- node {\ss 1} ++(-1,1); 
			\draw[thick] (3,0) -- node[pos=0.5] {\ss 0} ++(0,1); \draw[thick] (3,0) -- node[pos=0.5] {\ss 0} ++(1,0); \draw[thick] (3+1,0) -- node {\ss 0} ++(-1,1); 
		\end{tikzpicture}=\frac{\beta q(q^2-1)}{Q(\beta,q)-q^2 e^{-\alpha_2}}\]	
		\[\def\posa{0.5}\def\posb{0.5}\def\thescale{1}
		\begin{tikzpicture}[math mode,nodes={edgelabel},x={(-0.577cm,-1cm)},y={(0.577cm,-1cm)},scale=\thescale,baseline=(current  bounding  box.center)]
			\draw[thick] (1,0) -- node[pos=0.5] {\ss 0} ++(0,1); \draw[thick] (1,0) -- node[pos=0.5] {\ss 1} ++(1,0); 
			\draw[thick] (1,1) -- node[pos=0.5] {\ss 0} ++(0,1); \draw[thick] (1,1) -- node[pos=0.5] {\ss 10} ++(1,0); 
			\draw[thick] (1,2) -- node[pos=0.5] {\ss 1} ++(0,1); \draw[thick] (1,2) -- node[pos=0.5] {\ss 10} ++(1,0); \draw[thick] (1+1,2) -- node {\ss 0} ++(-1,1); 
			\draw[thick] (2,0) -- node[pos=0.5] {\ss10} ++(0,1); \draw[thick] (2,0) -- node[pos=0.5] {\ss 0} ++(1,0); 
			\draw[thick] (2,1) -- node[pos=0.5] {\ss 0} ++(0,1); \draw[thick] (2,1) -- node[pos=0.5] {\ss 0} ++(1,0); \draw[thick] (2+1,1) -- node {\ss 0} ++(-1,1); 
			\draw[thick] (3,0) -- node[pos=0.5] {\ss 10} ++(0,1); \draw[thick] (3,0) -- node[pos=0.5] {\ss 0} ++(1,0); \draw[thick] (3+1,0) -- node {\ss 1} ++(-1,1); 
		\end{tikzpicture}= \frac{\beta q^2Q(\beta,q)(q^2-1)(1-e^{-\alpha_2})}{(Q(\beta,q)-q^2e^{-\alpha_2})(Q(\beta,q)-q^2e^{-\alpha_1-\alpha_2})}\]
	\end{egg}

	\begin{egg}\label{egg:5}
		We first compute the classes for $T^*(\mrm{Gr}(2,4))$. There are exactly three simple roots $\alpha_1,\alpha_2,\alpha_3$ in $\Sigma_\Theta^+$.  The classes are as follows (we will not expand the formulas as in the previous examples, as these formula are more complex than in previous examples):
		\[q^4\cdot S_{1100}=q^4\cdot \left[0,0,0,0,0,\tfrac{1-e^{-\alpha_2}}{Q(\beta,q)-q^2e^{-\alpha_2}}\tfrac{1-e^{-\alpha_1-\alpha_2}}{Q(\beta,q)-q^2e^{-\alpha_1-\alpha_2}}\tfrac{1-e^{-\alpha_2-\alpha_3}}{Q(\beta,q)-q^2e^{-\alpha_2-\alpha_3}}\tfrac{1-e^{-\alpha_1-\alpha_2-\alpha_3}}{Q(\beta,q)-q^2e^{-\alpha_1-\alpha_2-\alpha_3}}\right].\]
		\[q^3\cdot S_{1010}=(\tfrac{1}{q}\partial_{\alpha_2})(q^4\cdot S_{1100});\quad q^2\cdot S_{0110}=(\tfrac{1}{q}\partial_{\alpha_1})(q^3\cdot S_{1010});\quad q^2\cdot S_{1001}=(\tfrac{1}{q}\partial_{\alpha_3})(q^3\cdot S_{1010});\]
		\[q\cdot S_{0101}=(\tfrac{1}{q}\partial_{\alpha_3})(q^2\cdot S_{0110});\quad S_{0011}=(\tfrac{1}{q}\partial_{\alpha_2})(q\cdot S_{0101}).\]
		One can verify, for example, that		
		\begin{align*}
			\overline{c}&_{1001,0011}^{1010}=\def\posa{0.5}\def\posb{0.5}\def\thescale{1}
			\begin{tikzpicture}[math mode,nodes={edgelabel},x={(-0.577cm,-1cm)},y={(0.577cm,-1cm)},scale=\thescale,baseline=(current  bounding  box.center)]
				\draw[thick] (0,0) -- node[pos=0.5] {\ss 0} ++(0,1); \draw[thick] (0,0) -- node[pos=0.5] {\ss 1} ++(1,0); 
				\draw[thick] (0,1) -- node[pos=0.5] {\ss 0} ++(0,1); \draw[thick] (0,1) -- node[pos=0.5] {\ss 1} ++(1,0); 
				\draw[thick] (0,2) -- node[pos=0.5] {\ss 1} ++(0,1); \draw[thick] (0,2) -- node[pos=0.5] {\ss 10} ++(1,0); 
				\draw[thick] (0,3) -- node[pos=0.5] {\ss 1} ++(0,1); \draw[thick] (0,3) -- node[pos=0.5] {\ss 10} ++(1,0); \draw[thick] (0+1,3) -- node {\ss 0} ++(-1,1); 
				\draw[thick] (1,0) -- node[pos=0.5] {\ss 0} ++(0,1); \draw[thick] (1,0) -- node[pos=0.5] {\ss 0} ++(1,0); 
				\draw[thick] (1,1) -- node[pos=0.5] {\ss 10} ++(0,1); \draw[thick] (1,1) -- node[pos=0.5] {\ss 10} ++(1,0); 
				\draw[thick] (1,2) -- node[pos=0.5] {\ss 1} ++(0,1); \draw[thick] (1,2) -- node[pos=0.5] {\ss 1} ++(1,0); \draw[thick] (1+1,2) -- node {\ss 1} ++(-1,1); 
				\draw[thick] (2,0) -- node[pos=0.5] {\ss 1} ++(0,1); \draw[thick] (2,0) -- node[pos=0.5] {\ss 0} ++(1,0); 
				\draw[thick] (2,1) -- node[pos=0.5] {\ss 0} ++(0,1); \draw[thick] (2,1) -- node[pos=0.5] {\ss 0} ++(1,0); \draw[thick] (2+1,1) -- node {\ss 0} ++(-1,1); 
				\draw[thick] (3,0) -- node[pos=0.5] {\ss 1} ++(0,1); \draw[thick] (3,0) -- node[pos=0.5] {\ss 1} ++(1,0); \draw[thick] (3+1,0) -- node {\ss 1} ++(-1,1); 
			\end{tikzpicture}+
			\begin{tikzpicture}[math mode,nodes={edgelabel},x={(-0.577cm,-1cm)},y={(0.577cm,-1cm)},scale=\thescale,baseline=(current  bounding  box.center)]
				\draw[thick] (0,0) -- node[pos=0.5] {\ss 0} ++(0,1); \draw[thick] (0,0) -- node[pos=0.5] {\ss 1} ++(1,0); 
				\draw[thick] (0,1) -- node[pos=0.5] {\ss 0} ++(0,1); \draw[thick] (0,1) -- node[pos=0.5] {\ss 1} ++(1,0); 
				\draw[thick] (0,2) -- node[pos=0.5] {\ss 1} ++(0,1); \draw[thick] (0,2) -- node[pos=0.5] {\ss 10} ++(1,0); 
				\draw[thick] (0,3) -- node[pos=0.5] {\ss 1} ++(0,1); \draw[thick] (0,3) -- node[pos=0.5] {\ss 10} ++(1,0); \draw[thick] (0+1,3) -- node {\ss 0} ++(-1,1); 
				\draw[thick] (1,0) -- node[pos=0.5] {\ss 0} ++(0,1); \draw[thick] (1,0) -- node[pos=0.5] {\ss 0} ++(1,0); 
				\draw[thick] (1,1) -- node[pos=0.5] {\ss 10} ++(0,1); \draw[thick] (1,1) -- node[pos=0.5] {\ss 0} ++(1,0); 
				\draw[thick] (1,2) -- node[pos=0.5] {\ss 1} ++(0,1); \draw[thick] (1,2) -- node[pos=0.5] {\ss 1} ++(1,0); \draw[thick] (1+1,2) -- node {\ss 1} ++(-1,1); 
				\draw[thick] (2,0) -- node[pos=0.5] {\ss 0} ++(0,1); \draw[thick] (2,0) -- node[pos=0.5] {\ss 0} ++(1,0); 
				\draw[thick] (2,1) -- node[pos=0.5] {\ss 1} ++(0,1); \draw[thick] (2,1) -- node[pos=0.5] {\ss 10} ++(1,0); \draw[thick] (2+1,1) -- node {\ss 0} ++(-1,1); 
				\draw[thick] (3,0) -- node[pos=0.5] {\ss 1} ++(0,1); \draw[thick] (3,0) -- node[pos=0.5] {\ss 1} ++(1,0); \draw[thick] (3+1,0) -- node {\ss 1} ++(-1,1); 
			\end{tikzpicture}+
			\begin{tikzpicture}[math mode,nodes={edgelabel},x={(-0.577cm,-1cm)},y={(0.577cm,-1cm)},scale=\thescale,baseline=(current  bounding  box.center)]
				\draw[thick] (0,0) -- node[pos=0.5] {\ss 0} ++(0,1); \draw[thick] (0,0) -- node[pos=0.5] {\ss 1} ++(1,0); 
				\draw[thick] (0,1) -- node[pos=0.5] {\ss 0} ++(0,1); \draw[thick] (0,1) -- node[pos=0.5] {\ss 10} ++(1,0); 
				\draw[thick] (0,2) -- node[pos=0.5] {\ss 1} ++(0,1); \draw[thick] (0,2) -- node[pos=0.5] {\ss 10} ++(1,0); 
				\draw[thick] (0,3) -- node[pos=0.5] {\ss 1} ++(0,1); \draw[thick] (0,3) -- node[pos=0.5] {\ss 10} ++(1,0); \draw[thick] (0+1,3) -- node {\ss 0} ++(-1,1); 
				\draw[thick] (1,0) -- node[pos=0.5] {\ss 10} ++(0,1); \draw[thick] (1,0) -- node[pos=0.5] {\ss 0} ++(1,0); 
				\draw[thick] (1,1) -- node[pos=0.5] {\ss 0} ++(0,1); \draw[thick] (1,1) -- node[pos=0.5] {\ss 1} ++(1,0); 
				\draw[thick] (1,2) -- node[pos=0.5] {\ss 1} ++(0,1); \draw[thick] (1,2) -- node[pos=0.5] {\ss 1} ++(1,0); \draw[thick] (1+1,2) -- node {\ss 1} ++(-1,1); 
				\draw[thick] (2,0) -- node[pos=0.5] {\ss 1} ++(0,1); \draw[thick] (2,0) -- node[pos=0.5] {\ss 0} ++(1,0); 
				\draw[thick] (2,1) -- node[pos=0.5] {\ss 0} ++(0,1); \draw[thick] (2,1) -- node[pos=0.5] {\ss 0} ++(1,0); \draw[thick] (2+1,1) -- node {\ss 0} ++(-1,1); 
				\draw[thick] (3,0) -- node[pos=0.5] {\ss 1} ++(0,1); \draw[thick] (3,0) -- node[pos=0.5] {\ss 1} ++(1,0); \draw[thick] (3+1,0) -- node {\ss 1} ++(-1,1); 
			\end{tikzpicture}\\
			&\\
			&=\frac{q(1-e^{-\alpha_1})}{Q(\beta,q)-q^2e^{-\alpha_1}}\cdot\frac{\beta(1-q^2)}{Q(\beta,q)-q^2e^{-\alpha_1-\alpha_2}}\cdot\frac{\beta Q(\beta,q)(q^2-1)e^{-\alpha_2}}{q(Q(\beta,q)-q^2e^{-\alpha_2})}\cdot\frac{\beta q(q^2-1)}{Q(\beta,q)-q^2 e^{-\alpha_2-\alpha_3}}\\
			&\quad+\frac{\beta(1-q^2)}{Q(\beta,q)-q^2 e^{-\alpha_1}}\cdot\frac{\beta(1-q^2)e^{-\alpha_2}}{Q(\beta,q)-q^2e^{-\alpha_2}}\cdot\frac{\beta q (q^2-1)}{Q(\beta,q)-q^2e^{-\alpha_2-\alpha_3}}\\
			&\quad+\frac{q(1-e^{-\alpha_1})}{Q(\beta,q)-q^2e^{-\alpha_1}}\cdot\frac{\beta(1-q^2)e^{-\alpha_1-\alpha_2}}{Q(\beta,q)-q^2 e^{-\alpha_1-\alpha_2}}\cdot\frac{\beta q(q^2-1)}{Q(\beta,q)-q^2e^{-\alpha_1-\alpha_2-\alpha_3}}\cdot\frac{qQ(\beta,q)(1-e^{-\alpha_2-\alpha_3})}{Q(\beta,q)-q^2e^{-\alpha_2-\alpha_3}}.
		\end{align*}
	\end{egg}

	\section{Connection to the multi-parameter quantum group}\label{section:multi-parameter}
	
	In this section, we will prove that the $R$, $U$, and $D$ matrices defined in \cref{subsection:R} are intertwiners for evaluation representations of the \textit{multi-parameter quantum group} of type $\widehat{\mathfrak{a}}_2$. In \cref{subsection:multiparameter-quantum-group}, we will recall the definition of the multi-parameter quantum group of type $\widehat{\mathfrak{a}}_n$ following \cite{HP12}. In \cref{subsection:mild-deformation}, we will define certain representations of the multi-parameter quantum group of type $\widehat{\mathfrak{a}}_2$ and show that the $R$, $U$, and $D$ matrices of \cref{subsection:R} intertwine these representations. Later in this paper, in \cref{problem:convolution-structure}, we will pose a problem involving the geometric representation theory of multi-parameter quantum groups: do they act on equivariant connective $K$-rings of Nakajima quiver varieties?

	\subsection{The multi-parameter quantum group of type $A$}\label{subsection:multiparameter-quantum-group}
	In this subsection, we will recall the definition of the multi-parameter quantum group of affine type $A$, following \cite{HP12}. Although the multi-parameter quantum group is defined for all symmetrizable Kac-Moody Lie algebras in \cite[Def. 7]{HP12}, we will only focus on affine type $A$ in this paper. See Appendix \ref{section:quantum-group} for the presentation of the usual quantum affine algebra  $U_q(\widehat{\mathfrak{a}}_n)$ of type $\widehat{\mathfrak{a}}_n$ in terms of generators and relations. 
	
	\begin{rmk}
		The study of multi-parameter deformations of the coordinate rings of linear algebraic groups was initiated independently in \cite{AST91} and \cite{R90}. Two-parameter deformations of these coordinate rings have appeared extensively in the literature. See the introduction of \cite{PHR10} for exposition on the history of multi-parameter quantum groups.
	\end{rmk}

	Let $\widehat{\mathfrak{a}}_n$ be the affine Lie algebra of type $A$, and let $C=(C_{i,j})_{i,j=0,\dotsc,n}$ be its generalized Cartan matrix:
	\[C_{i,j}=\begin{cases}
		2,\quad &\text{ $i=j$;}\\
		-1,\quad &\text{ $i= j+1$, or $i=j-1$, or $(i,j)\in\{(0,n),(n,0)\}$;}\\
		0,\quad&\text{ otherwise.}
	\end{cases}\]
	In other words, $C$ is the following $(n+1)\times (n+1)$-matrix:
	\[
	\begin{pmatrix}
		2 & -1 & 0 & 0 & 0 & \dotsb & 0 & 0 & 0 & -1\\
		-1 & 2 & -1 & 0 & 0 & \dotsb & 0 & 0 & 0 & 0 \\
		0 & -1 & 2 & -1 & 0 & \dotsb & 0 & 0 & 0 & 0 \\
		0 & 0 & -1 & 2 & -1 & \dotsb & 0 & 0 & 0 & 0 \\
		\vdots & \vdots & \vdots &\vdots  & \vdots  & \ddots & \vdots & \vdots & \vdots & \vdots  \\
		0 & 0 & 0 & 0 & 0 & \dotsb & -1 & 2 & -1 & 0 \\
		0 & 0 & 0 & 0 & 0 & \dotsb & 0 & -1 & 2 & -1\\
		-1 & 0 & 0 & 0 & 0 & \dotsb & 0 & 0 & -1 & 2
	\end{pmatrix}
	\]
	Consider the fraction field $\mbb{Q}(q_{i,j})$ of the polynomial ring $\mbb{Q}[q_{i,j}]$, with variables $q_{i,j}$ where $i,j=0,\dotsc,n$. Assume that $q_{i,i}q_{i,j}q_{j,i}=1$, $i\neq j$, and set $\mathbf{q}:=(q_{i,j})_{i,j=0,\dotsc,n}$. Let $\mbb{K}$ be a field containing $\mbb{Q}(q_{i,j})$.
	\begin{dfn}(\cite[Def. 7]{HP12}, \cite[Def. 7]{HPR10})
		The \textbf{multi-parameter quantum group} $U_{\mathbf{q}}(\widehat{a}_n)$ is the associative unital algebra over $\mbb{K}$ generated by elements $E_i$, $F_i$, $K_{i}^{(1)}$, $(K_i^{(1)})^{-1}$, $K_i^{(2)}$, $(K_i^{(2)})^{-1}$, where $i=0,1,\dotsc,n$, subject to the relations:
		
		\
		
		\begin{center}
			\begin{varwidth}{\textwidth}
				\begin{enumerate}[itemsep=7pt]
					\item  $K_i^{(2)}(K_i^{(2)})^{-1}=(K_i^{(2)})^{-1}K_i^{(2)}=1$,\quad\quad $K_i^{(1)}(K_i^{(1)})^{-1}=(K_i^{(1)})^{-1}K_i^{(1)}=1$,
					\item $K_i^{(1)}K_j^{(2)}=K_j^{(2)}K_i^{(1)}$,\quad\quad  $K_i^{(1)}K_j^{(1)}=K_j^{(1)}K_i^{(1)}$,\quad\quad $K_i^{(2)}K_j^{(2)}=K_j^{(2)}K_i^{(2)}$,
					\item $K_i^{(1)}E_j(K_i^{(1)})^{-1}=q_{i,j}E_j$,\quad\quad $(K_i^{(2)})^{-1}E_jK_i^{(2)}=q_{j,i}^{-1}E_j$,
					\item $K_i^{(1)}F_j(K_i^{(1)})^{-1}=q_{i,j}^{-1}F_j$,\quad\quad $(K_i^{(2)})^{-1}F_jK_i^{(2)}=q_{j,i}F_j$,
					\item\label{item:5} $[E_i,F_j]=\delta_{i,j}\frac{q_{i,i}}{q_{i,i}-1} \left(K_i^{(1)}-(K_i^{(2)})^{-1}\right)$,
					\item $ E_i^2E_j-q_{i,j}(1+q_{i,i}) E_iE_jE_i+\frac{q_{i,j}}{q_{j,i}}E_jE_i^2=0$, \quad $i\neq j$, 
					\item $\frac{q_{i,j}}{q_{j,i}}F_i^2F_j- q_{i,j}(1+q_{i,i})F_iF_jF_i+ F_jF_i^2=0$,\quad \quad $i\neq j$.
				\end{enumerate}
			\end{varwidth}
		\end{center}
	\end{dfn}
	
	\begin{rmk}
		In the definition of the multi-parameter quantum group given in \cite[Def. 7]{HP12}, the authors denote their generators by $e_i,f_i,\omega_i^{\pm 1}, (\omega_i')^{\pm1}$. The generators we are using correspond to theirs as follows:
		\[E_i=e_i, \quad F_i=f_i,\quad K_{i}^{(1)}=\omega_i,\quad  (K_i^{(1)})^{-1}=\omega_i^{-1},\quad K_i^{(2)}=(\omega_i')^{-1},\quad  (K_i^{(2)})^{-1}=\omega_i'.\]
	\end{rmk}

	\begin{prop}(\cite[\S 2.3]{HP12})
		The associative algebra $U_{\mathbf{q}}(\widehat{a}_n)$ has a Hopf algebra structure with the coproduct $\Delta$, counit $\epsilon$, and antipode $S$ given by:
		\[\Delta(E_i)=K_i^{(1)}\otimes E_i+E_i\otimes 1,\quad \Delta(F_i)=1\otimes F_i+F_i\otimes (K_i^{(2)})^{-1}\]
		\[\Delta(K_i^{(1)})=K_i^{(1)}\otimes K_i^{(1)},\quad  \Delta(K_i^{(2)})=K_i^{(2)}\otimes K_i^{(2)}\]
		\[\epsilon(E_i)=0,\quad \epsilon(F_i)=0,\quad \epsilon(K_i^{(1)})=1,\quad \epsilon(K_i^{(2)})=1\]
		\[S(E_i)=-(K_i^{(1)})^{-1}E_i,\quad S(F_i)=-F_iK_i^{(2)},\quad S(K_i^{(1)})=(K_i^{(1)})^{-1},\quad S(K_i^{(2)})=(K_i^{(2)})^{-1}.\]
	\end{prop}
	
	\begin{rmk}
		Assume that $q_{i,j}=q^{C_{i,j}}$. Then $(K_i^{(1)}-K_i^{(2)})$ is a Hopf ideal in $U_{\mathbf{q}}(\widehat{a}_n)$, and 
		\[U_{\mathbf{q}}(\widehat{a}_n)/(K_i^{(1)}-K_i^{(2)})\simeq U_q(\widehat{a}_n),\]
		where $U_q(\widehat{a}_n)$ is the usual quantum affine algebra of type $\widehat{a}_n$ (see Appendix \ref{section:quantum-group}).
	\end{rmk}

	\subsection{Representations of $U_{\mathbf{q}}(\widehat{a}_2)$}\label{subsection:mild-deformation}
	In this subsection, we will define representations of $U_{\mathbf{q}}(\widehat{a}_2)$ and show that the $R$, $U$, and $D$ matrices of \cref{subsection:R} intertwine these representations. The main theorem of this subsection is \cref{theo:R-representations}.
	
	Recall the element $Q(\beta,q)=q^2+\beta-q^2\beta\in \mbb{Q}[q,\beta]$ of Definition \ref{dfn:Q(b)}.
	
	\begin{dfn}
		Let $I$ be the ideal in $U_{\mathbf{q}}(\widehat{a}_2)$ generated by the following relations, and set $\mbb{K}=\mbb{Q}(q_{i,j},q,\beta)$:
		\begin{align*}
			&q_{0,0}-\tfrac{q^2}{Q(\beta,q)}, \quad q_{1,0}-\tfrac{1}{q},
			\quad q_{2,0}-\tfrac{Q(\beta,q)}{q},\\
			&q_{0,1}-\tfrac{Q(\beta,q)}{q},\quad q_{1,1}-\tfrac{q^2}{Q(\beta,q)},\quad q_{2,1}-\tfrac{1}{q},\\
			&q_{0,2}-\tfrac{1}{q},\quad q_{1,2}-\tfrac{Q(\beta,q)}{q},\quad q_{2,2}-\tfrac{q^2}{Q(\beta,q)}.
		\end{align*}	
		We will call the $\mbb{K}$-algebra $U_{q}(\beta,\widehat{a}_2):=U_{\mathbf{q}}(\widehat{a}_2)/I$ the \textbf{connective quantum affine algebra} of type $\widehat{a}_2$.
	\end{dfn}

	Define the following vector space:  $V^z:=(\mrm{Frac}((K_{\widehat{T}}(\mrm{pt})[\beta,\beta^{-1}])[z,z^{-1}]))^3$.
	We will now define three representations of ${U}_{q}(\beta,\widehat{a}_2)$ whose underlying vector space is $V^z$. We will denote these three representations by $\rho_g^z$, $\rho_r^z$, and $\rho_b^z$:
	\begin{align*}
		&	\rho_g^z(E_0)=\frac{q}{z}\begin{bmatrix}
			0 & 0 &0 \\
			0 & 0 & 0 \\
			-q & 0 & 0
		\end{bmatrix}\quad \rho_g^z(E_1)=q\begin{bmatrix}
			0 & 1 & 0\\
			0 & 0 & 0\\
			0 & 0 & 0
		\end{bmatrix}\quad 
		\rho_g^z(E_2)=q\begin{bmatrix}
			0 & 0 & 0\\
			0 & 0 & -\frac{1}{q}\\
			0 & 0 & 0
		\end{bmatrix}\\
		&\rho_g^z(F_0)=z\begin{bmatrix}
			0 & 0 & -\frac{1}{q} \\
			0 & 0 & 0 \\
			0 & 0 & 0
		\end{bmatrix}\quad \rho_g^z(F_1)=\begin{bmatrix}
			0 & 0 & 0\\
			\frac{1}{Q(\beta,q)} & 0 & 0\\
			0 & 0 & 0
		\end{bmatrix}\quad 
		\rho_g^z(F_2)=\begin{bmatrix}
			0 & 0 & 0\\
			0 & 0 & 0\\
			0 & -q & 0
		\end{bmatrix}\\
		&\rho_g^z(K_0^{(1)})=\begin{bmatrix}
			\frac{Q(\beta,q)}{q} & 0 & 0\\
			0 & 1 & 0\\
			0 & 0 & q
		\end{bmatrix}\quad  
		\rho_g^z(K_1^{(1)})=\begin{bmatrix}
			\frac{q}{Q(\beta,q)} & 0 & 0\\
			0 & \frac{1}{q} & 0\\
			0 & 0 & \frac{1}{Q(\beta,q)}
		\end{bmatrix} \quad
		\rho_g^z(K_2^{(1)})=\begin{bmatrix}
			1  & 0 & 0\\
			0 & q & 0\\
			0 & 0 & \frac{Q(\beta,q)}{q}
		\end{bmatrix}\\
		&\rho_g^z(K_0^{(2)})=\begin{bmatrix}
			\frac{1}{q} & 0 & 0\\
			0 & 1 & 0\\
			0 & 0 & \frac{q}{Q(\beta,q)}
		\end{bmatrix}\quad  
		\rho_g^z(K_1^{(2)})=\begin{bmatrix}
			q & 0 & 0\\
			0 & \frac{Q(\beta,q)}{q} & 0\\
			0 & 0 & Q(\beta,q)
		\end{bmatrix} \quad
		\rho_g^z(K_2^{(2)})=\begin{bmatrix}
			1  & 0 & 0\\
			0 & \frac{q}{Q(\beta,q)} & 0\\
			0 & 0 & \frac{1}{q}
		\end{bmatrix}
	\end{align*}
	
	\begin{align*}
		&\rho_r^z(E_0)=\frac{q}{z}\begin{bmatrix}
			0 & 0 &0 \\
			0 & 0 & -\frac{1}{q} \\
			0 & 0 & 0
		\end{bmatrix}\quad 
		\rho_r^z(E_1)=q\begin{bmatrix}
			0 & 0 & 0\\
			0 & 0 & 0\\
			-q & 0 & 0
		\end{bmatrix}\quad 
		\rho_r^z(E_2)=q\begin{bmatrix}
			0 &  \frac{1}{Q(\beta,q)} & 0\\
			0 & 0 & 0\\
			0 & 0 & 0
		\end{bmatrix}\\
		&\rho_r^z(F_0)=z\begin{bmatrix}
			0 & 0 &0 \\
			0 & 0 & 0 \\
			0 & -q & 0
		\end{bmatrix}\quad 
		\rho_r^z(F_1)=\begin{bmatrix}
			0 & 0 & -\frac{1}{q}\\
			0 & 0 & 0\\
			0 & 0 & 0
		\end{bmatrix}\quad 
		\rho_r^z(F_2)=\begin{bmatrix}
			0 & 0 & 0\\
			1 & 0 & 0\\
			0 & 0 & 0
		\end{bmatrix}\\
		&	\rho_r^z(K_0^{(1)})=\begin{bmatrix}
			1 & 0 & 0\\
			0 & q & 0\\
			0 & 0 & \frac{Q(\beta,q)}{q}
		\end{bmatrix}\quad  
		\rho_r^z(K_1^{(1)})=\begin{bmatrix}
			\frac{Q(\beta,q)}{q} & 0 & 0\\
			0 & 1 & 0\\
			0 & 0 & q
		\end{bmatrix} \quad
		\rho_r^z(K_2^{(1)})=\begin{bmatrix}
			\frac{q}{Q(\beta,q)}  & 0 & 0\\
			0 & \frac{1}{q} & 0\\
			0 & 0 & \frac{1}{Q(\beta,q)}
		\end{bmatrix}\\
		&\rho_r^z(K_0^{(2)})=\begin{bmatrix}
			1 & 0 & 0\\
			0 & \frac{q}{Q(\beta,q)} & 0\\
			0 & 0 & \frac{1}{q}
		\end{bmatrix}\quad  
		\rho_r^z(K_1^{(2)})=\begin{bmatrix}
			\frac{1}{q} & 0 & 0\\
			0 & 1 & 0\\
			0 & 0 & \frac{q}{Q(\beta,q)}
		\end{bmatrix} \quad
		\rho_r^z(K_2^{(2)})=\begin{bmatrix}
			q  & 0 & 0\\
			0 & \frac{Q(\beta,q)}{q} & 0\\
			0 & 0 & Q(\beta,q)
		\end{bmatrix}
	\end{align*}

	\begin{align*}
		&\rho_b^z(E_0)=\frac{q}{z}\begin{bmatrix}
			0 & 0 &0 \\
			-\frac{1}{q^2} & 0 & 0 \\
			0 & 0 & 0
		\end{bmatrix}\quad 
		\rho_b^z(E_1)=q\begin{bmatrix}
			0 & 0 & 0\\
			0 & 0 & 0\\
			0 & 1 & 0
		\end{bmatrix}\quad 
		\rho_b^z(E_2)=q\begin{bmatrix}
			0 & 0 &  \frac{1}{Q(\beta,q)}\\
			0 & 0 & 0\\
			0 & 0 & 0
		\end{bmatrix}\\
		&\rho_b^z(F_0)=z\begin{bmatrix}
			0 & -q^2 & 0 \\
			0 & 0 & 0 \\
			0 & 0 & 0
		\end{bmatrix}\quad 
		\rho_b^z(F_1)=\begin{bmatrix}
			0 & 0 & 0\\
			0 & 0 & \frac{1}{Q(\beta,q)}\\
			0 & 0 & 0
		\end{bmatrix}\quad 
		\rho_b^z(F_2)=\begin{bmatrix}
			0 & 0 & 0\\
			0 & 0 & 0\\
			1 & 0 & 0
		\end{bmatrix}\\
		&\rho_b^z(K_0^{(1)})=\begin{bmatrix}
			\frac{Q(\beta,q)}{q}& 0 & 0\\
			0 & q & 0\\
			0 & 0 & Q(\beta,q)
		\end{bmatrix}\quad  
		\rho_b^z(K_1^{(1)})=\begin{bmatrix}
			1 & 0 & 0\\
			0 & \frac{1}{q} & 0\\
			0 & 0 & \frac{q}{Q(\beta,q)}
		\end{bmatrix} \quad
		\rho_b^z(K_2^{(1)})=\begin{bmatrix}
			\frac{q}{Q(\beta,q)}  & 0 & 0\\
			0 & 1 & 0\\
			0 & 0 & \frac{1}{q}
		\end{bmatrix}\\
		&\rho_b^z(K_0^{(2)})=\begin{bmatrix}
			\frac{1}{q}& 0 & 0\\
			0 & \frac{q}{Q(\beta,q)} & 0\\
			0 & 0 & \frac{1}{Q(\beta,q)}
		\end{bmatrix}\quad  
		\rho_b^z(K_1^{(2)})=\begin{bmatrix}
			1 & 0 & 0\\
			0 & \frac{Q(\beta,q)}{q} & 0\\
			0 & 0 & q
		\end{bmatrix} \quad
		\rho_b^z(K_2^{(2)})=\begin{bmatrix}
			q  & 0 & 0\\
			0 & 1 & 0\\
			0 & 0 & \frac{Q(\beta,q)}{q}
		\end{bmatrix}
	\end{align*}

	In Theorem \ref{theo:R-representations}, we will see that the $R$ matrices defined in \cref{section:puzzle-formula} are intertwiners for the $\rho_g^z$, $\rho_r^z$, and $\rho_b^z$ representations defined in this section, up to reparametrizations which we will now define:
	\[\overline{R}_{b,r}(\beta,z):=R_{b,r}(\beta,1/z),\quad \overline{R}_{b,g}(\beta,z):=R_{b,g}(\beta,q^4/z),\quad \overline{R}_{r,g}(\beta,z):=R_{r,g}(\beta,q^4/z),\]
	\[\overline{R}_{r,b}(\beta,z):=R_{r,b}(\beta,1/z),\quad \overline{R}_{g,b}(\beta,z):=R_{g,b}(\beta,1/(q^4z)),\quad \overline{R}_{g,r}(\beta,z):=R_{g,r}(\beta,1/(q^4z)),\]
	\[\overline{R}_{g,g}(\beta,z):=R_{g,g}(\beta,1/z),\quad \overline{R}_{r,r}(\beta,z):=R_{r,r}(\beta,1/z),\quad \overline{R}_{b,b}(\beta,z):=R_{b,b}(\beta,1/z).\] 
	
	\begin{theo}\label{theo:R-representations}
		The $R$, $U$, and $D$ matrices of \cref{section:puzzle-formula} are intertwiners for the representations $\rho_g$, $\rho_r$, and $\rho_g$. Precisely, for all combinations of $(i,j)\in\{(g,r),(r,g),(b,r),(r,b),(g,b),(b,g)\}$ and $\ell\in\{ 0,1,2\}$, we have
		\[\overline{R}_{i,j}(\beta,z_2/z_1)\circ((\rho_i^{z_1}\tensor \rho_j^{z_2})(\Delta(E_\ell)))=((\rho_j^{z_2}\tensor \rho_i^{{z_1}})(\Delta(E_\ell)))\circ \overline{R}_{i,j}(\beta,z_2/z_1)\]
		\[\overline{R}_{i,j}(\beta,z_2/z_1)\circ((\rho_i^{z_1}\tensor \rho_j^{z_2})(\Delta(F_\ell)))=((\rho_j^{z_2}\tensor \rho_i^{{z_1}})(\Delta(F_\ell)))\circ \overline{R}_{i,j}(\beta,z_2/z_1)\]
		\[\overline{R}_{i,j}(\beta,z_2/z_1)\circ((\rho_i^{z_1}\tensor \rho_j^{z_2})(\Delta(K_\ell^{(1)})))=((\rho_j^{z_2}\tensor \rho_i^{{z_1}})(\Delta(K_\ell^{(1)})))\circ \overline{R}_{i,j}(\beta,z_2/z_1)\]
		\[\overline{R}_{i,j}(\beta,z_2/z_1)\circ((\rho_i^{z_1}\tensor \rho_j^{z_2})(\Delta(K_\ell^{(2)})))=((\rho_j^{z_2}\tensor \rho_i^{{z_1}})(\Delta(K_\ell^{(2)})))\circ \overline{R}_{i,j}(\beta,z_2/z_1)\]
		\[U(\beta)\circ ((\rho_g^{qz}\otimes \rho_r^{(1/q)z})(\Delta(E_\ell)))=(\rho_b^{(1/q)^2z}(E_\ell))\circ U(\beta)\]
		\[U(\beta)\circ ((\rho_g^{qz}\otimes \rho_r^{(1/q)z})(\Delta(F_\ell)))=(\rho_b^{(1/q)^2z}(F_\ell))\circ U(\beta)\]
		\[U(\beta)\circ ((\rho_g^{qz}\otimes \rho_r^{(1/q)z})(\Delta(K_\ell^{(1)})))=(\rho_b^{(1/q)^2z}(K_\ell^{(1)}))\circ U(\beta)\]
		\[U(\beta)\circ ((\rho_g^{qz}\otimes \rho_r^{(1/q)z})(\Delta(K_\ell^{(2)})))=(\rho_b^{(1/q)^2z}(K_\ell^{(2)}))\circ U(\beta)\]
		\[D(\beta)\circ(\rho_b^{(1/q)^2z}(E_\ell))=((\rho_r^{(1/q)z}\otimes \rho_g^{qz})(\Delta(E_\ell)))\circ D(\beta)\]
		\[D(\beta)\circ(\rho_b^{(1/q)^2z}(F_\ell))=((\rho_r^{(1/q)z}\otimes \rho_g^{qz})(\Delta(F_\ell)))\circ D(\beta)\]
		\[D(\beta)\circ(\rho_b^{(1/q)^2z}(K_\ell^{(1)}))=((\rho_r^{(1/q)z}\otimes \rho_g^{qz})(\Delta(K_\ell^{(1)})))\circ D(\beta)\]
		\[D(\beta)\circ(\rho_b^{(1/q)^2z}(K_\ell^{(2)}))=((\rho_r^{(1/q)z}\otimes \rho_g^{qz})(\Delta(K_\ell^{(2)})))\circ D(\beta)\]
	\end{theo}
	\begin{proof}
		These are all straightforward computations, verified using SageMath \cite{G25}.
	\end{proof}
	
	\section{Connective motivic Segre classes and interpolating between puzzle rules}\label{section:interpolating}
	
	The goal of this section is to explain the origins of the deformed motivic Segre classes of \cref{section:deformation} using the \textit{connective formal group law} and the \textit{formal affine Demazure algebra}. Using this approach, we define ``connective motivic Segre classes" of Schubert cells, we explain how to interpolate between the motivic Segre classes and the Segre-Schwartz-MacPherson classes by evaluating a parameter $\beta$ at $1$ and $0$, respectively, and we interpolate between the puzzle formulas for multiplying Segre-Schwartz-MacPherson classes and multiplying motivic Segre classes.	
	
	In \cref{subsection:forma-group-laws}, we recall the definition and properties of formal group laws and the construction of the formal group algebra from \cite{CPZ}. The formal group algebra can be thought of an an algebraic model for the so-called \textit{torus-equivariant oriented cohomology of a point}, which is discussed in \cref{section:localization-techniques}. In \cref{subsection:formal-divided-difference}, we define Demazure-Lusztig operators for arbitrary formal group laws, and we show that the operators defined for the connective formal group law satisfy the braid relations. In \cref{subsection:second-construction}, we define the ``connective motivic Segre classes" of Schubert cells using Demazure-Lusztig operators for the connective formal group law. These classes interpolate between Segre-Schwartz-MacPherson classes (when $\beta=0$) and the deformed classes of \cref{section:deformation} (when $\beta\neq 0$). We then define ``connective motivic Chern classes" of Schubert cells and show that they satisfy the so-called \textit{GKM conditions} (this implies that these classes have a geometric realization as canonical elements in a quotient of the $\widehat{T}$-equivariant algebraic cobordism ring of $T^*(G/P_\Theta)$, which we will explore in \cref{section:localization-techniques}). In  \cref{subsection:general-puzzle-rule}, we prove a general puzzle formula for the connective motivic Segre classes in the $d=1$ case, which interpolates between the puzzle formula for multiplying Segre-Schwartz-MacPherson classes of Schubert cells and the puzzle formula of \cref{thm:main} for multiplying the $\beta$-deformed motivic Segre classes.
	
	\subsection{The formal group algebra}\label{subsection:forma-group-laws}
	In this subsection, we will recall the definition and basic properties of formal group laws, and we will recall the formal group algebra of \cite{CPZ}.
	
	\begin{dfn}
		Let $S$ be a commutative ring. A \textbf{one-dimensional commutative formal group law} $(S,F)$ is a formal power series $F(x,y)=x+y+\sum_{i,j\geq 1}a_{i,j}x^iy^j\in S\llbracket x,y\rrbracket$, satisfying 		
		\begin{enumerate}
			\item $F(0,x)=x$.
			\item $F(x,y)=F(y,x)$.
			\item $F(x,F(y,z))=F(F(x,y),z)$.
		\end{enumerate}
		The \textbf{Lazard ring} $\mbb{L}$ is the quotient of the free $\mbb{Z}$-algebra with generators $a_{i,j}$, $i,j\geq 1$, modulo the relations imposed by the axioms of the formal group law. The \textbf{universal formal group law} is the pair $(\mbb{L},F_U)$, where $F_U(x,y)=x+y+\sum_{i,j\geq 1}a_{i,j}x^iy^j\in \mbb{L}\llbracket x,y\rrbracket$.
	\end{dfn}

	\begin{rmk}
		The Lazard ring $\mbb{L}$ is graded by setting $\mrm{deg}(a_{i,j})=1-i-j$. That is, $\mbb{L}$ is \textit{non-positively} graded.
	\end{rmk}
	
	\begin{rmk}\label{rmk:induced-homomorphism}
		Any ring homomorphism $f\colon \mbb{L}\to S$, where $S$ is a commutative ring, induces a formal group law $F(x,y)=x+y+\sum_{i,j\geq 1}f(a_{i,j})x^iy^j$. Conversely, any formal group law over $S$ is induced by a ring homomorphism $f\colon \mbb{L}\to S$.
	\end{rmk}
	
	\begin{dfn}
		Let $(S,F)$ be a formal group law. The \textbf{formal inverse} of an element $u\in\ S\llbracket x,y\rrbracket$ is an element $-_Fu\in S\llbracket x,y\rrbracket$ such that $F(u,-_Fu)=0$.
	\end{dfn}
	
	\begin{lem}(\cite[Ch. I, \S 3, Proposition 1]{F68})\label{lem:formal-inverse}
		Let $(S,F)$ be a formal group law. Every element $u\in S\llbracket x,y\rrbracket$ has a unique formal inverse, and the formal inverse can be expressed $-_Fu=-u-\sum_{i\geq 2} c_i u^i$, where $c_i\in S$.
	\end{lem}

	\begin{egg} Let $S$ be a commutative unital ring.
		\begin{enumerate}[leftmargin=20pt]
			\item The \textbf{additive formal group law} over $S$ is the pair ($S$, $F_A$), where $F_A(x,y)=x+y$. The formal inverse of an element $u\in S\llbracket x,y\rrbracket$ is $-u$.
			\item The \textbf{multiplicative-periodic formal group law} is the pair $(S,F_M)$, where $F_M(x,y)=x+y-\beta xy$ and $\beta\in S^\times$ is a unit. The formal inverse of an element $u\in S\llbracket x,y\rrbracket$ is the power series $\frac{u}{\beta u-1}$.
			\item The \textbf{connective formal group law} over $S$ is the pair $(S,F_C)$, where $F_C(x,y)=x+y-\beta xy$ and $\beta$ is \textit{not} invertible in $S$. The formal inverse of an element $u\in S\llbracket x,y\rrbracket$ is the power series $\frac{u}{\beta u-1}$.
		\end{enumerate}
	\end{egg}

	Let $S$ be a commutative unital ring and $(S,F)$ a formal group law over $S$. Let $M$ be a finitely generated free abelian group. Consider the formal power series ring $S\llbracket x_M\rrbracket$ over $S$ with variables indexed by elements of $M$. Given $u,v\in S\llbracket x_M\rrbracket$ and $n\in\mathbb{Z}_{\geq 0}$, we define the following notation:
	\[u+_Fv=F(u,v); \quad n\cdot_Fu:=\underbrace{u+_F+_F\dotsb+_Fu}_{\text{$n$ times}};\quad (-n)\cdot_F u=-_F(n\cdot_F u),\]
	where $-_F u$ is the formal inverse of $u$ under $(S,F)$. Observe that $S\llbracket x_M\rrbracket$ is a complete and Hausdorff topological ring with the $\mcal{I}$-adic topology, where $\mcal{I}$ is the kernel of the ring homomorphism $S\llbracket x_M\rrbracket\to S$ that fixes $S$ and sends all $x_M\mapsto 0$.
	\begin{dfn}
		Let $J_F$ be the ideal in $S\llbracket x_M\rrbracket$ generated by the elements 
		\[x_0 \quad \text{and}\quad x_{m_1+m_2}-(x_{m_1}+_Fx_{m_2}), \]
		over all $m_1,m_2\in  M$.
		Let $\mcal{J}_F$ be the topological closure of the ideal $J_F$ in the topological ring $S\llbracket x_M\rrbracket$. The quotient $S\llbracket M\rrbracket_F:=S\llbracket x_M\rrbracket/\mathcal{J}_F$ is the \textbf{formal group algebra} of \cite[Def. 2.4]{CPZ}.
	\end{dfn}
	
	\begin{lem}\label{lem:S-algebra-iso}(\cite[Theorem 2.10 and Lemma 2.11]{CPZ})
		Let $(S,F)$ be a formal group law, and let $M$ be a finitely generated free abelian group. Let $N$ be a finitely generated free abelian group isomorphic to $\mbb{Z}$, with basis $\{e\}$. There is an isomorphism of $S\llbracket M\rrbracket_F$-algebras:
		\[S\llbracket M\oplus N\rrbracket_F\xrightarrow{\simeq}  S\llbracket M\rrbracket_F\llbracket x\rrbracket,\quad x_{ne}\mapsto n\cdot_F x,\quad n\in\mbb{Z},\quad x_m\mapsto x_m,\quad m\in M.\]
	\end{lem}
	
	\begin{rmk}\label{rmk:integral-domain}
		Let $(S,F)$ be a formal group law, and choose a $\mbb{Z}$-basis $\{e_i\dotsc,e_n\}$ for the lattice $M$. It follows from \cref{lem:S-algebra-iso} that there is an $S$-algebra isomorphism $S\llbracket M\rrbracket_F\simeq S\llbracket x_1,\dotsc,x_n\rrbracket$, where $x_i$ is identified with the element $x_{e_i}\in S\llbracket M\rrbracket_F$. In particular, $S\llbracket M\rrbracket_F$ is an integral domain.
	\end{rmk}

	Let $(S,F)$ be a formal group law. Observe that the weight lattice $\Lambda$ of $T$ is a finitely generated free abelian group. Let $\Lambda\oplus \mbb{Z}q$ be the weight lattice of $\widehat{T}=T\times \mbb{C}^\times$. By \cref{lem:S-algebra-iso} we have $S\llbracket \Lambda\oplus \mbb{Z}q\rrbracket_F\simeq S\llbracket \Lambda\rrbracket_F\llbracket x_q\rrbracket$. 
	
	\begin{rmk}
		The Weyl group $W$ acts on $S\llbracket\Lambda\rrbracket_F$ by $w(x_\lambda):=x_{w(\lambda)}$ for all $w\in W$, $\lambda\in\Lambda$.
	\end{rmk}
	
	\begin{dfn}
		Define the rings $S_A:=\mbb{Z}$, $S_M:=\mbb{Z}[\beta,\beta^{-1}]$, and $S_C:=\mbb{Z}[\beta]$.
	\end{dfn}
	
	\begin{lem}(\cite[Example 2.18]{CPZ})\label{lem:H-iso}
		Consider the additive formal group law $F_A(x,y)=x+y$ over $S_A$. Let $\widehat{H}_T(\mrm{pt})$ be the completion of the ring $H_T(\mrm{pt})$ at the ideal generated by the elements $\lambda$ over all $\lambda\in\Lambda$. There is an $S_A$-algebra isomorphism
		\[S_A\llbracket\Lambda\rrbracket_{F_A}\llbracket x_q\rrbracket\to \widehat{H}_T(\mrm{pt})\llbracket x_{q}\rrbracket,\quad x_\lambda\mapsto \lambda, \quad \lambda\in\Lambda,\quad x_q\mapsto x_{q}.\]
	\end{lem}

	\begin{lem}(\cite[Example 2.19]{CPZ})\label{lem:K-iso}
		Consider the multiplicative-periodic formal group law $F_M(x,y)=x+y-\beta xy$ over $S_M$. Let $\widehat{K}_T(\mrm{pt})$ be the completion of the ring $K_T(\mrm{pt})$ at the ideal generated by the elements $1-e^\lambda$ over all $\lambda\in\Lambda$. There is an $S_M$-algebra isomorphism:
		\[S_M\llbracket\Lambda\rrbracket_{F_M}\llbracket x_q\rrbracket\to \widehat{K}_T(\mrm{pt})\llbracket x_q\rrbracket[\beta,\beta^{-1}],\quad x_\lambda\mapsto \beta^{-1}(1-e^\lambda),\quad \lambda\in\Lambda,\quad x_{q}\mapsto x_q.\]
		The inverse of this map sends $e^{\lambda}\mapsto 1-\beta x_\lambda$ for all $\lambda\in\Lambda$, and $x_q\mapsto x_q$.
	\end{lem}

	\begin{lem}\label{lem:ring-isos}
		Consider the following three formal group laws:
		\begin{itemize}
			\item The connective formal group law $F_C(x,y)=x+y-\beta xy$ over $S_C$.
			\item The multiplicative-periodic formal group law $F_M(x,y)=x+y-\beta xy$ over $S_M$.
			\item The additive formal group law $F_A(x,y)=x+y$ over $S_A$.
		\end{itemize}
		Let $(S_C\llbracket\Lambda\rrbracket_{F_C}\llbracket x_q\rrbracket)_\beta $ be the localization of $S_C\llbracket\Lambda\rrbracket_{F_C}\llbracket x_q\rrbracket$ at the element $\beta$, and the let $(\beta)$ be the ideal in $S_C\llbracket\Lambda\rrbracket_{F_C}\llbracket x_q\rrbracket$ generated by $\beta$.
		There are ring isomorphisms
		\begin{equation}\label{eqn:iso1}
			(S_C\llbracket\Lambda\rrbracket_{F_C}\llbracket x_q\rrbracket)_\beta\xrightarrow{\simeq}  S_M\llbracket\Lambda\rrbracket_{F_M}\llbracket x_q\rrbracket\xrightarrow{\simeq} \widehat{K}_T(\mrm{pt})\llbracket x_q\rrbracket[\beta,\beta^{-1}],\quad x_\lambda\mapsto x_\lambda\mapsto \beta^{-1}(1-e^\lambda);
		\end{equation}
		\begin{equation}\label{eqn:iso2}
			(S_C\llbracket\Lambda\rrbracket_{F_C}\llbracket x_q\rrbracket)/(\beta)\xrightarrow{\simeq}  S_A\llbracket\Lambda\rrbracket_{F_A}\llbracket x_q\rrbracket\xrightarrow{\simeq} \widehat{H}_T(\mrm{pt})\llbracket x_q\rrbracket,\quad x_\lambda\mapsto x_\lambda\mapsto \lambda.
		\end{equation}
	\end{lem}
	\begin{proof}
		Observe the ring isomorphisms $(S_C)_\beta\simeq S_M$ and $S_C/(\beta)\simeq S_A$. This result follows from \cref{lem:H-iso} and \cref{lem:K-iso}.
	\end{proof}

	\begin{lem}\label{lem:injects-into-localization}
		The ring $S_C\llbracket\Lambda\rrbracket_{F_C}\llbracket x_q\rrbracket$ injects into its localization $(S_C\llbracket\Lambda\rrbracket_{F_C}\llbracket x_q\rrbracket)_\beta\simeq \widehat{K}_T(\mrm{pt})\llbracket x_q\rrbracket[\beta,\beta^{-1}]$ at $\beta$.
	\end{lem}
	\begin{proof}
		This follows since $\beta$ is a nonzero element in the integral domain $S_C\llbracket\Lambda\rrbracket_{F_C}$ (see Remark \ref{rmk:integral-domain}).
	\end{proof}

	The diagram below illustrates the main takeaway from \cref{lem:ring-isos}:
	
	\begin{equation}\label{eqn:non-eq-ck}\begin{tikzcd}
			& x_\lambda \in S_C\llbracket\Lambda\rrbracket_{F_C}\llbracket x_q\rrbracket  \arrow [dl, swap] \arrow[dr ] &\\
			\lambda\in\widehat{H}_{T}(\mrm{pt})\llbracket x_q\rrbracket	
			& & \beta^{-1}(1-e^{\lambda}) \in \widehat{K}_{T}(\mrm{pt})\llbracket x_q\rrbracket[\beta,\beta^{-1}]
		\end{tikzcd}
	\end{equation}

	The following technical lemmas will be used in the proof of \cref{prop:GKM}.
	\begin{lem}\label{lem:technical}
		For all $\lambda\in\Lambda$, the following relations hold in $S_C\llbracket \Lambda\rrbracket_{F_C}\llbracket x_q\rrbracket$:
		\begin{enumerate}
			\item $x_\lambda=x_{-\lambda}(\beta x_{\lambda}-1)$.
			\item $(\beta x_{\lambda}-1)(\beta x_{-\lambda}-1)=1$.
		\end{enumerate}
	\end{lem}
	\begin{proof}
		(1) The relations defining the ring $S_C\llbracket \Lambda\rrbracket_{F_C}$ imply that $0=x_0=x_{-\lambda+\lambda}=x_{-\lambda}+x_\lambda-\beta x_\lambda x_{-\lambda}$. 
		
		(2) We compute $(\beta x_{\lambda}-1)(\beta x_{-\lambda}-1)=1-\beta(x_{-\lambda}+x_\lambda-\beta x_\lambda x_{-\lambda})=1$.
	\end{proof}
	
	\begin{lem}\label{lem:does-not-divide}
		Let $(S,F)$ be a formal group law, and choose $i_1,i_2,j_1,j_2=1,\dotsc,n$ with $i_1\neq i_2$ and $j_1\neq j_2$. Then $x_{y_{j_1}-y_{j_2}}$ divides $x_{y_{i_1}-y_{i_2}}$ in $S\llbracket \Lambda\rrbracket_F\llbracket x_q\rrbracket$ if and only if $\{j_1,j_2\}=\{i_1,i_2\}$.
	\end{lem}
	\begin{proof}
		If $\{j_{1},j_{j_2}\}=\{i_1,i_2\}$, then $x_{y_{j_1}-y_{j_2}}$ divides $x_{y_{i_1}-y_{i_2}}$ by \cref{lem:formal-inverse}. Now assume that $\{j_{1},j_2\}\neq \{i_1,i_2\}$.
		The isomorphism $S\llbracket \Lambda\rrbracket_F\llbracket x_q\rrbracket\simeq S\llbracket x_1,\dotsc,x_n\rrbracket\llbracket x_q\rrbracket$ of \cref{rmk:integral-domain} identifies $x_{y_{r}-y_{s}}$ with the element
		\[x_r+_F(-_F x_s)=x_r+(-_F x_s)+\sum_{k_1,k_2\geq 1} a_{k_1,k_2} x_r^{k_1}(-_F x_s)^{k_2}.\]
		Thus, the elements $x_{y_{j_1}-y_{j_2}}$ and $x_{y_{i_1}-y_{i_2}}$ are identified with 
		\[x_{j_1}+(-_F x_{j_2})+\sum_{k_1,k_2\geq 1} a_{k_1,k_2} x_{j_1}^{k_1}(-_F x_{j_2})^{k_2}\quad \text{and}\quad x_{i_1}+(-_F x_{i_2})+\sum_{k_1,k_2\geq 1} a_{k_1,k_2} x_{i_1}^{k_1}(-_F x_{i_2})^{k_2},\]
		respectively in $S\llbracket x_1,\dotsc,x_n\rrbracket\llbracket x_q\rrbracket$. Assume towards a contradiction that $h\in S\llbracket x_1,\dotsc,x_n\rrbracket\llbracket x_q\rrbracket$ satisfies 
		\[h\cdot\left(x_{j_1}+(-_F x_{j_2})+\sum_{k_1,k_2\geq 1} a_{k_1,k_2} x_{j_1}^{k_1}(-_F x_{j_2})^{k_2}\right)=x_{i_1}+(-_F x_{i_2})+\sum_{k_1,k_2\geq 1} a_{k_1,k_2} x_{i_1}^{k_1}(-_F x_{i_2})^{k_2}.\]
		If $i_1\neq j_1$ and $i_1\neq j_2$, then one of the terms on the left hand side must equal $x_{i_1}$, which is impossible. Similarly, if $i_2\neq j_1$ and $i_2\neq j_2$, then \cref{lem:formal-inverse} implies that one of the terms on the left hand side must equal $-x_{i_2}$. This is impossible as well.
	\end{proof}

	\subsection{Formal divided difference operators}\label{subsection:formal-divided-difference}
	
	In this subsection, we will define Demazure-Lusztig operators for every formal group law $(S,F)$. We show that our Demazure-Lusztig operators satisfy the braid relations if and only if $F(x,y)=x+y-\beta xy$ for some $\beta\in S$.
	
	\begin{dfn}(\cite[Def. 3.11]{CPZ})
		Define the following $S$-linear operators on the fraction field of $S\llbracket\Lambda\rrbracket_F$:
		\[\delta_i:=\frac{1}{x_{-\alpha_i}}+\frac{1}{x_{\alpha_i}}r_i=(1+r_i)\tfrac{1}{x_{-\alpha_i}},\quad \alpha_i\in\Delta.\]
	\end{dfn}
	\begin{lem}(\cite[Def. 3.11]{CPZ})\label{lem:restricts}
		The operator $\delta_i$ restricts to an $S$-linear operator on $S\llbracket\Lambda\rrbracket_F$.
	\end{lem}
	Set $\kappa_i:=\frac{1}{x_{\alpha_i}}+\frac{1}{x_{-\alpha_i}}$ and $\kappa_{i,j}:=\frac{1}{x_{-\alpha_i-\alpha_{j}}}\left(\frac{1}{x_{\alpha_i}}-\frac{1}{x_{-\alpha_{j}}}\right)-\frac{1}{x_{-\alpha_i}x_{-\alpha_{j}}}$.
	
	\begin{lem}(\cite[Example 6.6 and Remark 6.10]{HMSZ}, \cite[p. 809]{BE90})\label{lem:equals0}
		Let $(S,F)$ be a formal group law.
		Both $\kappa_i$ and $\kappa_{i,j}$ lie in $S\llbracket\Lambda\rrbracket_F$. The element $\kappa_{i,j}$ equals $0$ if and only if $F(x,y)=x+y-\beta xy$ for some $\beta \in S$. If $F(x,y)=x+y-\beta xy$ for some $\beta \in S$, then $\kappa_i=\beta$.
	\end{lem}
	
	\begin{prop}\label{prop:braid-relations}
		The operator $\delta_i$ satisfies the following relations:
		\begin{enumerate}
			\item $\delta_i\delta_j=\delta_j\delta_i$, for $|i-j|\geq 2$.
			\item $\delta_i\delta_{i+1}\delta_{i}-\delta_{i+1}\delta_i\delta_{i+1}=\kappa_{i+1,i}\delta_i-\kappa_{i,i+1}\delta_{i+1}$,\quad $i=1,\dotsc,n-2$.
			\item $\delta_i^2=\kappa_i\delta_i$.
		\end{enumerate} 
		Moreover, the element $\kappa_{i,j}$ is zero for all $i,j=1,\dotsc,n-1$ if and only if $(S,F)$ is the additive, multiplicative-periodic, or connective formal group law over $S$.
	\end{prop}
	\begin{proof}
		This is a direct computation, following the technique in the proof of \cite[Prop. 6.8]{HMSZ}. The latter statement follows from \cref{lem:equals0}.
	\end{proof}
	
	\begin{rmk}\label{rmk:braid-relations}
		Proposition \ref{prop:braid-relations} says that the operators $\delta_i$ satisfy the braid relations if and only if $(S,F)$ is the additive, multiplicative-periodic, or connective formal group law over $S$.
	\end{rmk}
	
	Extend $\delta_i$ to an $S\llbracket x_q\rrbracket$-linear operator on $S\llbracket\Lambda\rrbracket_F\llbracket x_q\rrbracket$.
	\begin{dfn}
		Define the following $S$-linear operators on $S\llbracket\Lambda\rrbracket_F\llbracket x_q\rrbracket$:
		\[\partial_i:=(1-x_q^2)\delta_{i}+x_q^2r_{i},\quad \alpha_i\in\Delta.\]
	\end{dfn}

	\begin{prop}\label{prop:independent}
		The operator $\partial_i$ satisfies the following relations: 
		\begin{enumerate}
			\item $\partial_i\partial_j=\partial_j\partial_i$, for $|i-j|\geq 2$.
			\item $\partial_i\partial_{i+1}\partial_{i}-\partial_{i+1}\partial_i\partial_{i+1}=(1-x_q^2)(\kappa_{i,i+1}\partial_i-\kappa_{i+1,i}\partial_{i+1}+x_q^2\left(\kappa_{i+1}-\kappa_{i}\right))$,\quad $i=1,\dotsc,n-2$.
			\item $(\partial_i+x_q^2)\circ(\partial_i+(x_q^2\kappa_i-x_q^2-\kappa_i))=0$.
		\end{enumerate} 
		Moreover, the operators $\partial_i$ satisfy the braid relations if and only if $(S,F)$ is the additive, multiplicative-periodic, or connective formal group law over $S$.
	\end{prop}
	
	\begin{proof}
		This is a straightforward calculation. The latter statement follows from \cref{lem:equals0}.
	\end{proof}
	
	\begin{rmk}
		Operators similar to the $\partial_i$ are defined in \cite{ZZ17}. Their operators satisfy similar relations to ours (see \cite[Theorem 2.13]{ZZ17}).
	\end{rmk}
	
	Write $w\in W$ as a reduced product of simple transpositions $w=r_{i_1}\cdots r_{i_k}$. When $(S,F)=(S_C,F_C)$ is the connective formal group law over $S_C$, we will use the notation $\partial_w:=\partial_{{i_1}}\circ\dotsb\circ \partial_{{i_k}}$; it follows from Proposition \ref{prop:independent} that $\partial_w$ is independent of the choice of reduced word for $w$.

	\subsection{Connective motivic Chern classes and the GKM conditions}\label{subsection:second-construction}
	
	In this subsection, we will define connective motivic Segre classes of Schubert cells using the same method that we used to define the deformed motivic Segre classes in \cref{subsection:motive-Segre}. In \cref{egg:H}, we will explain how to recover the Segre-Schwartz-MacPherson classes from the connective motivic Segre classes by setting $\beta=0$. We will use the connective motivic Segre classes to define connective motivic Chern classes, and we will show that the connective motivic Chern classes satisfy the so-called \textit{GKM conditions} (see \cref{prop:GKM}). The fact that the connective motivic Chern classes satisfy the GKM conditions will imply that the connective motivic Chern classes can be viewed as canonical elements in a quotient of the $\widehat{T}$-equivariant algebraic cobordism ring of $T^*(G/P_\Theta)$, as we will show in \cref{theo:geometric-connective-stable}.

	We will assume in this subsection that $(S,F)=(S_C,F_C)$ is the connective formal group law over $S_C$. Let $(S_C\llbracket\Lambda\rrbracket_{F_C}\llbracket x_q\rrbracket)^{\mrm{loc}}$ be the localization of $S_C\llbracket\Lambda\rrbracket_{F_C}\llbracket x_q\rrbracket$ at the elements $\{1-x_q^2(1-x_\lambda)\}_{\lambda\in \Sigma}$. As $S_C\llbracket\Lambda\rrbracket_{F_C}\llbracket x_q\rrbracket$ is an integral domain, so is $(S_C\llbracket\Lambda\rrbracket_{F_C}\llbracket x_q\rrbracket)^{\mrm{loc}}$. Consider the rings
	\[\widetilde{\mcal{R}}_\Lambda:=\bigoplus_{\lambda\in W^\Theta}S_C\llbracket\Lambda\rrbracket_{F_C}\llbracket x_q\rrbracket\quad\text{and}\quad \widetilde{\mcal{R}}_\Lambda^{\text{loc}}:=\bigoplus_{\lambda\in W^\Theta}(S_C\llbracket\Lambda\rrbracket_{F_C}\llbracket x_q\rrbracket)^{\mrm{loc}},\]
	each with direct sum addition and multiplication. There is a natural action of $W^\Theta$ on $\widetilde{\mcal{R}}_\Lambda$ given by $w\cdot(f_\lambda)_\lambda:=(w(f_\lambda))_{w(\lambda)}$ for all $w\in W^\Theta$. Moreover, the action of $W^\Theta$ on $S_C\llbracket\Lambda\rrbracket_{F_C}$ preserves the set of elements $\{1-x_q^2(1-x_\lambda)\}_{\lambda\in \Sigma}$. Therefore, the action of $W^\Theta$ on $\widetilde{\mcal{R}}_\Lambda^\mrm{loc}$ induces an action on $\widetilde{\mcal{R}}_\Lambda^\mrm{loc}$. In turn, we get an action of $\partial_w$ on $\widetilde{\mcal{R}}_\Lambda^\mrm{loc}$ for all $w\in W^\Theta$. Consider the element $S_{w_0}\in\widetilde{\mcal{R}}_\Lambda^\mrm{loc}$ defined by 
	\[S_{w_0}|_\lambda=\begin{cases}
		\prod_{\alpha\in\Sigma_\Theta^+} \frac{x_{-\alpha}}{1-x_q^2(1-x_{-\alpha})},& \text{if $\lambda=w_0$};\\
		0,&\text{otherwise.}
	\end{cases}\] 
	Define $S_{w\cdot w_0}:=\partial_w(S_{w_0})$. The set $\{S_\lambda\}_{\lambda\in W^\Theta}$ forms a basis for $\widetilde{\mcal{R}}_\Lambda^\mrm{loc}$ over $(S_C\llbracket\Lambda\rrbracket_{F_C}\llbracket x_q\rrbracket)^{\text{loc}}$. 
	\begin{dfn}
		We define $S_w$ to be the \textbf{connective motivic Segre class} of the Schubert cell $X_w^\circ$.
	\end{dfn}
	\begin{rmk}
		Observe that the classes $S_w$ are well-defined because, by \cref{prop:independent}, the $\partial_w$ satisfy the braid relations for the formal group law $(S_C,F_C)$. 
	\end{rmk}
	We will now discuss two interesting specializations of the parameter $\beta$.
	\begin{egg}\label{egg:connective:beta}
		When we apply the isomorphism (\ref{eqn:iso1}) and set $x_q$ to $q$, the operator $\partial_i$ and class $S_{w_0}$ are the following:
		\[\partial_i=\frac{\beta(1-q^2)}{1-e^{-\alpha_i}}+\frac{Q(\beta,q)-q^2e^{\alpha_i}}{1-e^{\alpha_i}}r_{i}
		\quad\text{and}\quad S_{w_0}|_\lambda=\begin{cases}
			\prod_{\alpha\in\Sigma_\Theta^+} \frac{1-e^{-
					\alpha}}{Q(\beta,q)-q^2e^{-\alpha}},& \text{if $\lambda=w_0$};\\
			0,&\text{otherwise.}
		\end{cases}
		\]
		Thus, we can view the images of the classes $\{S_\lambda\}_{\lambda\in W^\Theta}$ under the isomorphism (\ref{eqn:iso1}) as the set of deformed motivic Segre classes from \cref{section:deformation}.
	\end{egg}
	
	\begin{egg}\label{egg:H}
		When we apply the isomorphism (\ref{eqn:iso2}) and set $x_q$ to $\hbar$, the operator $\partial_i$ and class $S_{w_0}$ are the following:
		\[\partial_i=\frac{\hbar^2-1}{\alpha_i}+\frac{1-\hbar^2+\hbar^2\alpha_i}{\alpha_i}r_i
		\quad\text{and}\quad S_{w_0}|_\lambda=\begin{cases}
			\prod_{\alpha\in\Sigma_\Theta^+} \frac{\alpha}{\hbar^2+\hbar^2\alpha-1},& \text{if $\lambda=w_0$};\\
			0,&\text{otherwise.}
		\end{cases}
		\]
		Replace the operator $\partial_i$ with $\partial_i^H:=\frac{\partial_i}{\hbar^2}$, replace $S_{w_0}$ with $S_{w_0}^H:=(\hbar^2)^{\ell(w_0)}S_{w_0}$, and set $\overline{\hbar}:=\tfrac{\hbar^2-1}{\hbar^2}$. Then 
		\[\partial_i^H=\frac{\overline{\hbar}}{\alpha_i}+\frac{\alpha_i-\overline{\hbar}}{\alpha_i}r_i\quad\text{and}\quad S_{w_0}^H|_\lambda=\begin{cases}
			\prod_{\alpha\in\Sigma_\Theta^+}\frac{\alpha}{\alpha+\overline{\hbar}},& \text{if $\lambda=w_0$};\\
			0,&\text{otherwise.}
		\end{cases}\]
		Thus, we can view the set of homogenizations $\{(\hbar^2)^{\ell(\lambda)}S_\lambda\}_{\lambda\in W^\Theta}$ as the set of images of the Segre-Schwartz-MacPherson classes of Schubert cells in $H_{\widehat{T}}^{\text{loc}}(T^*(G/P_\Theta))$ under the localization map $\iota$.
	\end{egg}
	
	Define following the operators on $\widetilde{\mcal{R}}_\Lambda$:
	\[\widetilde{\partial}_i:=(1-x_q^2)\delta_i+x_q^2r_i+\beta(1-x_q^2), \quad i=1,\dotsc,n-1.\]
	As explained in \cref{rmk:inverse-relations}, $\widetilde{\partial}_i$ is the inverse of $\tfrac{1}{x_q^2 (x_q^2+\beta-x_q^2 \beta)}\partial_i$ and satisfies the braid relations. Set $\widetilde{\partial}_w:=\widetilde{\partial}_{i_1}\circ\dotsb \circ \widetilde{\partial}_{i_k}$ for any choice of reduced decomposition $w=r_{i_1}\dotsb r_{i_k}$ of $w\in W^\Theta$. Consider the element $\mrm{St}_{w_0}\in\widetilde{\mcal{R}}_\Lambda$ defined by $\mrm{St}_{w_0}:=\kappa S_{w_0}$,
	where $\kappa:=\prod_{\alpha\in\Sigma^+_\Theta}(1-x_q^2(1-x_{-\alpha}))$. Now recursively define $\mrm{St}_{w\cdot w_0}:=\widetilde{\partial}_w(\mrm{St}_{w_0})$.
	
	\begin{dfn}
		We will call $\mrm{St}_w$ the \textbf{connective motivic Chern class} of the Schubert cell $X_w^\circ$.
	\end{dfn}
	
	\begin{rmk}\label{rmk:connective-stable-class}
		When $\beta=1$, the class $\mrm{St}_w$ is the motivic Chern class of the Schubert cell $X_w^\circ$ (cf. \cite[\S 2.5]{KZJ21}), and when $\beta=0$, the class $\mrm{St}_w$ can be identified with the Chern-Schwartz-MacPherson class of the Schubert cell $X_w^\circ$ (this follows from \cref{egg:H} and \cite[\S 5]{KZJ21}). In light of \cref{theo:CSM-classes} and \cref{theo:motivic-Chern-classes}, it is also reasonable to call the classes $\{\mrm{St_w}\}_{w\in W^\Theta}$ the \textbf{connective stable classes} for $T^*(G/P_\Theta)$.
	\end{rmk}
	Below we prove the \textit{GKM conditions} for the connective motivic Chern classes. We will see in \cref{section:localization-techniques} that \cref{prop:GKM} implies the connective motivic Chern classes arise from geometric classes in a \textit{quotient of} the $\widehat{T}$-equivariant algebraic cobordism ring of $T^*(G/P_\Theta)$.
	
	\begin{prop}\label{prop:GKM}
		Assume $j,k\in\{1,\dotsc,n\}$ satisfy $j<k$, and suppose $y_j-y_k=\alpha_j+\dotsb+\alpha_{k-1}$ lies in $\Sigma_\Theta$.
		Then $\mrm{St}_\lambda$ lies in $\widetilde{\mcal{R}}_\Lambda$, and the difference $\mrm{St}_\lambda|_{\mu}-\mrm{St}_\lambda|_{(j,k)\cdot \mu}$ is divisible by $x_{y_j-y_k}$ in $S_C\llbracket \Lambda\rrbracket_{F_C}\llbracket x_q\rrbracket$, where $(j,k)$ is the transposition in $S_n$ that swaps $j$ and $k$ and fixes all other $i$.
	\end{prop}
	\begin{proof}
		We will proceed by induction, decreasing the length of $\lambda$. For the base case, will consider $\mrm{St}_{w_0}$. By construction, $\mrm{St}_{w_0}$ lies in $\widetilde{\mcal{R}}_\Lambda$. Additionally, as $\mrm{St}_{w_0}|_{\mu}$ is nonzero if and only if $\mu=w_0$, and as $\mrm{St}_{w_0}|_{\mu}$ is divisible by all $x_{-(y_j-y_k)}$ whenever $y_j-y_k\in \Sigma_\Theta$, it follows that  $\mrm{St}_{w_0}|_{\mu}-\mrm{St}_{w_0}|_{(j,k)\cdot \mu}$
		is divisible by $x_{y_j-y_k}=(j,k)\cdot x_{-(y_j-y_k)}$ in $S_C\llbracket \Lambda\rrbracket_{F_C}\llbracket x_q\rrbracket$. 
		
		For our induction hypothesis, we assume that $\mrm{St}_\lambda$ lies in $\widetilde{\mcal{R}}_\Lambda$, and that, for all $j<k$, the element $\mrm{St}_{\lambda}|_{\mu}-\mrm{St}_{\lambda}|_{(j,k)\cdot \mu}$
		is divisible by $x_{y_j-y_k}$ in $S_C\llbracket \Lambda\rrbracket_{F_C}\llbracket x_q\rrbracket$. For $i$ such that $r_i(\lambda)<\lambda$, we will show that $\mrm{St}_{r_i(\lambda)}=\widetilde{\partial}_i(\mrm{St}_\lambda)$ lies in $\widetilde{\mcal{R}}_\Lambda$, and that $\widetilde{\partial}_i(\mrm{St}_\lambda)|_\mu-\widetilde{\partial}_i(\mrm{St}_\lambda)|_{(j,k)\cdot \mu}$ is divisible by $x_{y_j-y_k}$ in $S_C\llbracket \Lambda\rrbracket_{F_C}\llbracket x_q\rrbracket$. Recall that $\widetilde{\partial}_i=(1-x_q^2)\delta_i+x_q^2r_i+\beta(1-x_q^2)$ and that $\delta_i=\frac{1}{x_{-\alpha_i}}+\frac{1}{x_{\alpha_i}}r_i$. It follows from \cref{lem:technical} and the expression for $\delta_i$ that $r_i=x_{\alpha_i}\delta_i-\tfrac{x_{\alpha_i}}{x_{-\alpha_i}}=x_{\alpha_i}\delta_i-(\beta x_{\alpha_i}-1)$. By our induction hypothesis, both $x_q^2(\beta x_{\alpha_i}-1)\mrm{St}_\lambda$ and $(\beta(1-x_q^2))\mrm{St}_\lambda$ lie in $\widetilde{\mcal{R}}_\lambda$ and both $x_q^2(\beta x_{\alpha_i}-1)(\mrm{St}_{\lambda}|_{\mu}-\mrm{St}_{\lambda}|_{(j,k)\cdot \mu})$ and $(\beta(1-x_q^2))(\mrm{St}_{\lambda}|_{\mu}-\mrm{St}_{\lambda}|_{(j,k)\cdot \mu})$ are divisible by $x_{y_j-y_k}$ in $S_C\llbracket \Lambda\rrbracket_{F_C}\llbracket x_q\rrbracket$. Therefore, we only need to show that $\delta_i(\mrm{St}_\lambda)$ lies in $\widetilde{\mcal{R}}_\Lambda$, and that $\delta_i(\mrm{St}_\lambda)|_\mu-\delta_i(\mrm{St}_\lambda)|_{(j,k)\cdot \mu}$
		is divisible by $x_{y_j-y_k}$ in $S_C\llbracket \Lambda\rrbracket_{F_C}\llbracket x_q\rrbracket$.
		
		To see that $\delta_i(\mrm{St}_\lambda)$ lies in $\widetilde{\mcal{R}}_\Lambda$, we first note that, by \cref{lem:technical}, we have $\tfrac{1}{x_{\alpha_i}}=\tfrac{\beta x_{-\alpha_i}-1}{x_{-\alpha_i}}$.
		Therefore,
		\begin{align*}
			(\delta_i(\mrm{St}_\lambda))|_{\mu}&=\frac{\mrm{St}_\lambda|_{\mu}+(\beta x_{-\alpha_i}-1)(r_i(\mrm{St}_\lambda|_{r_i (\mu)}))}{x_{-\alpha_i}}
			\\ & = \frac{\mrm{St}_\lambda|_{\mu}+(\beta x_{-\alpha_i}-1)(x_{\alpha_i}\delta_i-(\beta x_{\alpha_i}-1))(\mrm{St}_\lambda|_{r_i (\mu)})}{x_{-\alpha_i}}\\
			&= \frac{\mrm{St}_\lambda|_\mu-(\beta x_{-\alpha_i}-1)(\beta x_{\alpha_i}-1)\mrm{St}_\lambda|_{r_i( \mu)}}{x_{-\alpha_i}}+\frac{(\beta x_{-\alpha_i}-1)x_{\alpha_i}\delta_i(\mrm{St}_\lambda|_{r_i(\mu)})}{x_{-\alpha_i}}.
		\end{align*}
		Again by \cref{lem:technical}, we have $(\beta x_{-\alpha_i}-1)(\beta x_{\alpha_i}-1)=1$ and $\frac{(\beta x_{-\alpha_i}-1)x_{\alpha_i}}{x_{-\alpha_i}}=1$. 
		By the induction hypothesis, $\mrm{St}_\lambda|_\mu-\mrm{St}_\lambda|_{(i,i+1)\cdot \mu}$ is divisible by $x_{-\alpha_i}$ in $S_C\llbracket \Lambda\rrbracket_{F_C}\llbracket x_q\rrbracket$, and by \cref{lem:restricts}, the element $\delta_i(\mrm{St}_\lambda|_{r_i(\mu)})$ lies in $S_C\llbracket \Lambda\rrbracket_{F_C}\llbracket x_q\rrbracket$. We conclude that $\delta_i(\mrm{St}_\lambda)$ lies in $\widetilde{\mcal{R}}_\lambda$.
		
		To prove that $\delta_i(\mrm{St}_\lambda)|_\mu-\delta_i(\mrm{St}_\lambda)|_{(j,k)\cdot \mu}$
		is divisible by $x_{y_j-y_k}$ in $S_C\llbracket \Lambda\rrbracket_{F_C}\llbracket x_q\rrbracket$, there are two cases to consider. 
		
		\underline{Case I}: Assume $(j,k)\neq r_i$. There exist $j',k'$ such that $r_i\cdot (j,k)=(j',k')\cdot r_i$. We compute
		\begin{align*}
			&\delta_i(\mrm{St}_\lambda)|_\mu-\delta_i(\mrm{St}_\lambda)|_{(j,k)\cdot \mu}\\
			&=\frac{\mrm{St}_\lambda|_{\mu}+(\beta x_{-\alpha_i}-1)r_i(\mrm{St}_\lambda|_{r_i\cdot \mu})-\mrm{St}_\lambda|_{(j,k)\cdot \mu}-(\beta x_{-\alpha_i}-1)r_i(\mrm{St}_\lambda|_{r_i\cdot (j,k)\cdot \mu})}{x_{-\alpha_i}}\\
			&=\frac{\mrm{St}_\lambda|_{\mu}-\mrm{St}_\lambda|_{(j,k)\cdot \mu}+(\beta x_{-\alpha_i}-1)r_i(S_\lambda|_{r_i\cdot \mu}-\mrm{St}_\lambda|_{r_i\cdot (j,k)\cdot \mu})}{x_{-\alpha_i}}\\
			&=\frac{\mrm{St}_\lambda|_{\mu}-\mrm{St}_\lambda|_{(j,k)\cdot \mu}+(\beta x_{-\alpha_i}-1)r_i(\mrm{St}_\lambda|_{r_i\cdot \mu}-\mrm{St}_\lambda|_{(j',k')\cdot r_i\cdot \mu})}{x_{-\alpha_i}}
		\end{align*}
		By the induction hypothesis, the element $x_{y_{j'}-y_{k'}}$ divides $\mrm{St}_\lambda|_{r_i\cdot \mu}-\mrm{St}_\lambda|_{(j',k')\cdot r_i\cdot \mu}$ in $S_C\llbracket \Lambda\rrbracket_{F_C}\llbracket x_q\rrbracket$. Hence, $x_{y_j-y_k}=r_i(x_{y_{j'}-y_{k'}})$ divides $r_i(\mrm{St}_\lambda|_{r_i\cdot \mu}-\mrm{St}_\lambda|_{(j',k')\cdot r_i\cdot \mu})$ in $S_C\llbracket \Lambda\rrbracket_{F_C}\llbracket x_q\rrbracket$. We also have by the induction hypothesis that $x_{y_j-y_k}$ divides $\mrm{St}_\lambda|_{\mu}-\mrm{St}_\lambda|_{(j,k)\cdot \mu}$ in $S_C\llbracket \Lambda\rrbracket_{F_C}\llbracket x_q\rrbracket$.
		Thus, we conclude that
		$x_{y_j-y_k}$ divides $	x_{-\alpha_i}(\delta_i(\mrm{St}_\lambda)|_\mu-\delta_i(\mrm{St}_\lambda)|_{(i,j)\cdot \mu})$ in $S_C\llbracket \Lambda\rrbracket_{F_C}\llbracket x_q\rrbracket$. Since $(j,k)\neq r_i$, it follows from 
		\cref{lem:does-not-divide} that $x_{y_j-y_k}$ does \textit{not} divide $x_{-\alpha_i}$. Therefore, $x_{y_j-y_k}$ must divide $\delta_i(\mrm{St}_\lambda)|_\mu-\delta_i(\mrm{St}_\lambda)|_{(i,j)\cdot \mu}$, which is what we wanted.
		
		\underline{Case II}: Assume that $(j,k)=r_i$. By the induction hypothesis, we can assume $\mrm{St}_{\lambda}|_{\mu}-\mrm{St}_{\lambda}|_{r_i\cdot \mu}=x_{-\alpha_i}D$ for some $D\in S_C\llbracket \Lambda\rrbracket_{F_C}\llbracket x_q\rrbracket$. We compute
		\begin{align*}
			&\delta_i(\mrm{St}_\lambda)|_\mu-\delta_i(\mrm{St}_\lambda)|_{r_i\cdot \mu}\\
			&=\frac{\mrm{St}_\lambda|_{\mu}+(\beta x_{-\alpha_i}-1)r_i(\mrm{St}_\lambda|_{r_i\cdot \mu})-\mrm{St}_\lambda|_{r_i\cdot \mu}-(\beta x_{-\alpha_i}-1)r_i(\mrm{St}_\lambda|_{ \mu})}{x_{-\alpha_i}}\\
			&=\frac{\mrm{St}_\lambda|_{\mu}-\mrm{St}_\lambda|_{r_i\cdot \mu}}{x_{-\alpha_i}}-\frac{(\beta x_{-\alpha_i}-1)r_i(\mrm{St}_\lambda|_{\mu}-\mrm{St}_\lambda|_{r_i\cdot \mu})}{x_{-\alpha_i}}\\
			&=\frac{x_{-\alpha_i}D}{x_{-\alpha_i}}-\frac{(\beta x_{-\alpha_i}-1) x_{\alpha_i}r_i(D)}{x_{-\alpha_i}}\\
			&=D-r_i(D).
		\end{align*}
		In the last equality, we used \cref{lem:technical} to deduce that $\frac{(\beta x_{-\alpha_i}-1) x_{\alpha_i}}{x_{-\alpha_i}}=1$. It follows from \cite[Corollary 3.4]{CPZ} that $D-r_i(D)$ is divisible by $x_{\alpha_i}$ in $S_C\llbracket \Lambda\rrbracket_{F_C}\llbracket x_q\rrbracket$, which completes the proof.
	\end{proof}
	
	The following remark will be used in \cref{subsection:geometric-stabclasses}.

	\begin{rmk}\label{rmk:general-motivic-Chern-classes}
		One can define classes analogous to the $\mrm{St}_\lambda$ for \textit{any} formal group law $(S,F)$. Indeed, fix a formal group law $(S,F)$, and fix a reduced expression $w=s_{i_1}\dotsc s_{i_k}$ for every $w\in W^\Theta$. Set $I_w:=(s_{i_1},\dotsc,s_{i_k})$. Consider the ring $\widetilde{\mcal{R}}_\Lambda^{(S,F)}:=\oplus_{w\in W^\Theta} S\llbracket\Lambda\rrbracket_F\llbracket x_q\rrbracket$, which has a natural action of $W^\Theta$ given by $w\cdot(f_\lambda)_{\lambda}:=(w(f_\lambda))_{w(\lambda)}$ for all $w\in W^\Theta$. Define the element $\mrm{St}_{I_{w_0}}^{(S,F)}\in \widetilde{\mcal{R}}_\Lambda^F$ by
		\[\mrm{St}_{I_{w_0}}^{(S,F)}|_\lambda=\begin{cases}
			\prod_{\alpha\in\Sigma_\Theta^+} x_{-\alpha},& \text{if $\lambda=w_0$};\\
			0,&\text{otherwise.}
		\end{cases}\]
		Set $\mrm{St}_{I_{w\cdot w_0}}^{(S,F)}:=\widetilde{\partial}_{I_w}(\mrm{St}_{I_{w_0}}^{(S,F)})$, where the notation $\widetilde{\partial}_{I_w}$ means $\widetilde{\partial}_{I_w}:=\widetilde{\partial}_{i_1}\circ\dotsb\circ \widetilde{\partial}_{i_k}$ when $I_w=(s_{i_1},\dotsc,s_{i_k})$, and in the definition of $\widetilde{\partial}_i$, the element $\beta$ can be replaced by \textit{any} element in $S$. It follows from \cref{prop:independent} that the $\widetilde{\partial}_{I_w}$ are independent of the choices of reduced words $I_w$ if and only if $(S,F)$ is the additive, multiplicative-periodic, or connective formal group law over $S$. Thus, the $\mrm{St}_{I_w}^{(S,F)}$ are independent of the choices of reduced words $I_w$ if and only if $(S,F)$ is the additive, multiplicative-periodic, or connective formal group law over $S$. When $(S,F)=(S_C,F_C)$ (resp. $(S,F)=(S_M,F_M)$; resp. $(S,F)=(S_A,F_A)$), the class $\mrm{St}_{I_{w}}^{(S_C,F_C)}$ is the connective motivic Chern class (resp. motivic Chern class; resp. Chern-Schwartz-MacPherson class) of the Schubert cell $X_w^\circ$.
		
		The analogue of \cref{prop:GKM} holds for the classes $\mrm{St}_{I_{w}}^{(S,F)}$. To be precise, assume $j,k\in\{1,\dotsc,n\}$ satisfy $j<k$, and suppose $y_j-y_k=\alpha_j+\dotsb+\alpha_k$ lies in $\Sigma_\Theta$. Then $\mrm{St}_{I_\lambda}^{(S,F)}$ lies in $\widetilde{\mcal{R}}_\Lambda^{(S,F)}$, and the difference $\mrm{St}_{I_\lambda}^{(S,F)}|_{\mu}-\mrm{St}^{(S,F)}_{I_\lambda}|_{(j,k)\cdot \mu}$ is divisible by $x_{y_j-y_k}$ in $S\llbracket \Lambda\rrbracket_{F}\llbracket x_q\rrbracket$, where $(j,k)$ is the transposition in $S_n$ that swaps $j$ and $k$ and fixes all other $i$. The proof of this statement is identical to the proof of \cref{prop:GKM}, noticing that $\tfrac{x_{-\alpha}}{x_{\alpha}}=g(\alpha)$ is a unit in $S\llbracket \Lambda\rrbracket_{F}\llbracket x_q\rrbracket$, as $g(\alpha)$ has constant term $-1$ by \cref{lem:formal-inverse}. A similar statement is proven in \cite[Theorem 10.7]{CZZ2}.
	\end{rmk}
	
	\subsection{The general puzzle rule}\label{subsection:general-puzzle-rule}
	In this subsection, we will prove a combinatorial formula in terms of Knutson-Tao puzzles for the structure constants in the basis of connective motivic Segre classes when $d=1$. This result essentially follows from \cref{thm:main} and the injectivity result \cref{lem:injects-into-localization}.
	
	We will make the following assumption for the rest of this section.
	
	\begin{assumption}
		For the rest of this section, we will assume that $d=1$.
	\end{assumption}

	\begin{dfn}\label{dfn:homogenization2}
		The \textbf{homogenization} of $S_\lambda$ by a factor $x_q$ is $x_q^{\ell(\lambda)}S_\lambda$, where $\ell(\lambda)$ is the length of $\lambda\in W^\Theta$.
	\end{dfn}

	Analogously to \cref{section:deformation}, we would like to compute the structure constants in the $S_\lambda$ basis. Below we list the fugacities of the rhombi and triangle tiles. Let $\lambda\in\Lambda$, and let $Q(\beta,x_q)=\beta+x_q^2-x_q^2\beta$.
	
	\[
	\begin{gathered}
		\rh0000=
		\rh1111=
		\rh0101=
		\rh1{10}1{10}=
		\rh{10}0{10}0=1
	\end{gathered}\quad\quad 
	\begin{gathered}
		\rh{10}{10}{10}{10}=Q(\beta,x_q)
	\end{gathered}
	\]
	\[
	\begin{gathered}
		\rh1{10}00=
		\rh{10}011=\frac{1-x_q^2}{1-x_q^2(1-x_\lambda)}
	\end{gathered}\quad\quad 
	\begin{gathered}
		\rh1010=\frac{x_qx_\lambda}{1-x_q^2(1-x_\lambda)}
	\end{gathered}
	\]
	\[
	\begin{gathered}
		\rh11{10}0=
		\rh001{10}=\frac{(1-x_q^2)(1-\beta x_{\lambda})}{1-x_q^2(1-x_\lambda)}
	\end{gathered}\quad\quad 
	\begin{gathered}
		\rh{10}{10}01=\frac{x_q(x_q^2-1)}{1-x_q^2(1-x_\lambda)}
	\end{gathered}	
	\]
	\[\begin{gathered}
		\rh0{10}0{10}=
		\rh{10}1{10}1=\frac{x_qQ(\beta,x_q)x_\lambda}{1-x_q^2(1-x_\lambda)}
	\end{gathered}\quad\quad
	\begin{gathered}
		\rh01{10}{10}=\frac{Q(\beta,x_q)(x_q^2-1)(1-\beta x_\lambda)}{x_q(1-x_q^2(1-x_\lambda))}
	\end{gathered}
	\]
	\[
	\uptri000=
	\uptri111=
	\uptri01{10}=
	\uptri{10}01=
	\uptri1{10}0=1 \quad\quad
	\uptri{10}{10}{10}=\frac{-Q(\beta,x_q)}{x_q}
	\]
	
	\[
	\downtri000=
	\downtri111=
	\downtri01{10}=
	\downtri{10}01=
	\downtri1{10}0=1\quad \quad
	\downtri{10}{10}{10}= - x_q
	\]
	
	\begin{theo}\label{thm:main2}
		The product of two classes $S_\lambda$ and $S_\mu$ in $\widetilde{\mcal{R}}_\Lambda^{\mrm{loc}}$ is given by the ``puzzle'' formula
		\begin{equation}\label{eq:main2}
			(x_q^{\ell(\lambda)}S_\lambda)(x_q^{\ell(\mu)} S_\mu)=
			\sum_\nu \tikz[scale=1.8,baseline=0.5cm]{\uptri{\lambda}{\nu}{\mu}}
			\ (x_q^{\ell(\nu)}S_\nu)  
		\end{equation}
	\end{theo}
	\begin{proof}
		Observe that the images of the puzzle piece fugacities and classes $S_\lambda$ under the localization at $\beta$ map $S_C\llbracket\Lambda\rrbracket_{F_C}\llbracket x_q\rrbracket\to \widehat{K}_{\widehat{T}}(\mrm{pt})[\beta,\beta^{-1}]$ are precisely the puzzle fugacties and deformed motivic Segre classes from \cref{section:puzzle-formula}. (Note that ${K}_{{T}}(\mrm{pt})[\beta,\beta^{-1}]$ injects into the completion $\widehat{K}_{{T}}(\mrm{pt})[\beta,\beta^{-1}]$; see \cref{lem:inclusion of rings} and \cref{rmk:inclusion of rings0}). In particular, the image of (\ref{eq:main2}) under the localization at $\beta$ map is precisely (\ref{eq:main}). As (\ref{eq:main}) holds and the localization at $\beta$ map is injective (\cref{lem:injects-into-localization}), it immediately follows that (\ref{eq:main2}) holds as well. 
	\end{proof}
	
	\begin{rmk}
		It follows from \cref{egg:connective:beta} that the puzzle rule (\ref{eq:main2}) specializes to the puzzle rule in \cref{thm:main} after applying the map (\ref{eqn:iso1}), and it follows from \cref{egg:H} that the puzzle rule (\ref{eq:main2}) specializes to a puzzle rule for multiplying Segre-Schwartz-MacPherson classes of Schubert cells after setting $\beta=0$.
	\end{rmk}
	
	\begin{rmk}\label{rmk:positive2}
		Consider the submonoid $M$ of $S_C\llbracket\Lambda\rrbracket_{F_C}\llbracket x_q\rrbracket$ under addition, defined as the set of sums of products of the factors
		\[-x_q^{\pm1}\qquad Q(\beta,x_q)\qquad  1-\beta x_{-\alpha}\qquad \tfrac{1-x_q^2}{1-x_q^2(1-x_{-\alpha})}\qquad -\tfrac{x_{-\alpha}}{1-x_q^2(1-x_{-\alpha})}\]
		over all $\alpha\in\Sigma_\Theta^+$. Then $M$ is a positivity monoid in $(S_C\llbracket\Lambda\rrbracket_{F_C}\llbracket x_q\rrbracket)^{\text{loc}}$. To see this, first note that the images of these factors under the localization at $\beta$ map $S_C\llbracket\Lambda\rrbracket_{F_C}\llbracket x_q\rrbracket\to \widehat{K}_{{T}}(\mrm{pt})\llbracket x_q\rrbracket[\beta,\beta^{-1}]$ are precisely the factors from Remark \ref{rmk:positive}. (Note that ${K}_{{T}}(\mrm{pt})[\beta,\beta^{-1}]$ injects into the completion $\widehat{K}_{{T}}(\mrm{pt})[\beta,\beta^{-1}]$; see \cref{lem:inclusion of rings} and \cref{rmk:inclusion of rings0}). As the localization at $\beta$ map is injective (\cref{lem:injects-into-localization}), and as the images of the factors above form a positivity monoid in ${K}_{\widehat{T}}^{\text{loc}}(\mrm{pt})[\beta,\beta^{-1}]$, it follows that $M$ is a positivity monoid as well. 
	\end{rmk}

	\section{Equivariant algebraic cobordism and localization techniques}\label{section:localization-techniques}
	
	In this section, we prove that the connective motivic Chern classes of Schubert cells defined in \cref{subsection:second-construction} can be interpreted as elements in a quotient of the equivariant algebraic cobordism ring of $T^*(G/P_\Theta)$. In \cref{subsection:algebraic-cobordism}, we give an overview of algebraic cobordism. In \cref{subsection:equivariant-algebraic-cobordism}, we give an overview of \textit{equivariant} algebraic cobordism, with an emphasis on torus-equivariant algebraic cobordism. In \cref{subsection:geometric-stabclasses}, we show that the connective motivic Chern classes defined in \cref{subsection:second-construction} can be viewed as canonical elements in a quotient of the $\widehat{T}$-equivariant algebraic cobordism ring of $T^*(G/P_\Theta)$. 
	
	\subsection{Algebraic cobordism}\label{subsection:algebraic-cobordism}
	In this subsection, we give an overview of algebraic cobordism, following the exposition in \cite{LM07}.

	Let $\textbf{Sch}_{\mbb{C}}$ be the category consisting of schemes that are separated and of finite type over $\mbb{C}$, and let $\textbf{Sm}_{\mbb{C}}$ be the full subcategory of $\textbf{Sch}_{\mbb{C}}$ consisting of schemes that are smooth and quasi-projective over $\mbb{C}$. Let $\textbf{Ring}^*$ be the category of graded, commutative, unital rings.
	In \cite{LM07}, the authors defined the notion of an \textit{oriented cohomology theory}, which is a contravariant functor $\mrm{h}^*\colon \textbf{Sm}_{\mbb{C}}\to \textbf{Ring}^*$ satisfying several cohomological axioms, such as localization, homotopy invariance, existence of push-forwards along projective morphisms (see \cite[Def. 1.1.2]{LM07}). A \textit{morphism} of oriented cohomology theories is a natural transformation of functors that commutes with the push-forwards. 
	
	Given an oriented cohomology theory $\mrm{h}^*$ and vector bundle $E\to X$ in $\textbf{Sm}_{\mbb{C}}$, there are unique elements $c_{i}(E)\in \mrm{h}^i(X)$, where $i=0,1,\dotsc,n$, characterized by the axioms in \cite[pp. 3]{LM07}, which include the ``Whitney sum formula," and the requirement that the $c_i$ commute with pullbacks. Every oriented cohomology theory $\mrm{h}^*$ is equipped with a formal group law $F$ over $S=\mrm{h}^*(\mrm{pt})$, where $\mrm{pt}=\mrm{Spec}(\mbb{C})$ is the point. This formal group law is characterized by the following formula, called the \textit{Quillen formula}: given line bundles $\mcal{L}_1\to X$, $\mcal{L}_2\to X$ in $\textbf{Sm}_{\mbb{C}}$, we have $c_1(\mcal{L}_1\otimes\mcal{L}_2)=F(c_1(\mcal{L}_1),c_1(\mcal{L}_2))$. 
	
	\begin{egg}
		The functor $\mrm{CH}^*\colon \textbf{Sm}_{\mbb{C}}\to \textbf{Ring}^*$ that sends a smooth scheme $X$ to its Chow ring $\mrm{CH}^*(X)$ of algebraic cycles modulo rational equivalence is an oriented cohomology theory. The formal group law associated to $\mrm{CH}^*$ is the additive formal group law $F_A(x,y)=x+y$ over $\mrm{CH}^*(\mrm{pt})\simeq \mbb{Z}$.
	\end{egg}
	
	\begin{egg}
		Consider the $K^0$ functor\footnote{Although a common notation for this functor is $K_0$, we follow the notation $K^0$ that is used in \cite{LM07}.} that sends a smooth variety $X$ to the Grothendieck ring $K^0(X)$ of vector bundles on $X$. Form the graded ring $K^*(X):=K^0(X)\otimes_{\mbb{Z}}\mbb{Z}[\beta,\beta^{-1}]$, where $\mrm{deg}(\beta)=-1$. Then the functor $K^*\colon \textbf{Sm}_{\mbb{C}}\to \textbf{Ring}^*$ that sends a smooth scheme $X$ to $K^*(X)$ is an oriented cohomology theory. The formal group law associated to $K^*$ is the multiplicative-periodic formal group law $F_M(x,y)=x+y-\beta xy$ over $\mrm{K}^*(\mrm{pt})\simeq \mbb{Z}[\beta,\beta^{-1}]$.
	\end{egg}
	
	For a finite-dimensional complex manifold $X$, one can define the \textit{complex cobordism ring} $MU^*(X)$ of $X$; see \cite[\S 1]{Q71}, \cite[\S 3]{BE92}, and \cite[\S 1]{T97} for exposition. In \cite{Q69}, Quillen showed that complex cobordism is the ``universal" complex-oriented cohomology theory, and the axioms given in \cite{LM07} to define an (algebraic) oriented cohomology theory were inspired by the axioms that Quillen gave in \cite{Q69} to define a ``complex-oriented cohomology theory." Complex cobordism also defines an (algebraic) oriented cohomology theory: 
	
	\begin{egg}
		Given $X\in\textbf{Sm}_{\mbb{C}}$, let $X(\mbb{C})$ be the set of complex points of $X$. Then $X(\mbb{C})$ is a complex manifold. The functor $MU^*$ that sends $X\in\textbf{Sm}_{\mbb{C}}$ to the complex cobordism ring $MU^*(X(\mbb{C}))$ is an oriented cohomology theory. The formal group law associated with $MU^*$ is the universal formal group law $F_U$ over the Lazard ring $MU^*(\mrm{pt})\simeq \mbb{L}^*$.
	\end{egg}
	
	In analogy to Quillen's result that complex cobordism is the universal complex-oriented cohomology theory, Levine--Morel proved in \cite{LM07} that there is a ``universal" oriented cohomology theory, which they called \textit{algebraic cobordism} $\Omega^*$:
	\begin{theo}(\cite[Theorem 1.2.6]{LM07})
		There exists an oriented cohomology theory $\Omega^*$, called \textbf{algebraic cobordism}, satisfying the following universal property:
		for any oriented cohomology theory $\mrm{h}^*$, there is a unique morphism of oriented cohomology theories $\Omega^*\to \mrm{h}^*$.
		The algebraic cobordism ring of a point $\Omega^*(\mrm{pt})$ is isomorphic to the Lazard ring $\mbb{L}$, and the formal group law associated to $\Omega^*$ is the universal formal group law $F_U$ over $\mbb{L}$.
	\end{theo}
	
	\begin{rmk}
		The functors $\Omega^*$ and $MU^*$ are \textit{not} isomorphic, yet they have the same formal group law. This says that an oriented cohomology theory is not uniquely identified by its formal group law.
	\end{rmk}
	
	The algebraic cobordism ring $\Omega^*(X)$ of a smooth variety $X$ is a ring that is generated by projective morphisms $Y\to X$ from smooth, irreducible, and quasi-projective varieties, modulo certain relations. The group $\Omega^k(X)$ is generated by classes $[f\colon Y\to X]$, such that $k=\mrm{dim}(Y)-\mrm{dim}(X)$. The pullback $p^*\colon \Omega^*(\mrm{pt})\to\Omega^*(X)$ along the structure map $p\colon X\to\mrm{pt}$ gives $\Omega^*(X)$ the structure of a graded $\mbb{L}^*$-algebra. A good exposition on the Levine-Morel (\cite{LM07}) construction of the algebraic cobordism ring of an smooth $\mbb{C}$-scheme can be found in \cite[\S 2.1]{K12}. A more geometric construction of algebraic cobordism in terms of so-called ``double-point degenerations" was developed by Levine-Pandharipande \cite{LP09}, and a good exposition on this construction can be found in \cite[\S 2.2]{K12}.

	\begin{dfn}(\cite[page 39]{LM07})
		A \textbf{free oriented cohomology theory} is an oriented cohomology theory of the form $\mrm{h}^*(-):=\Omega^*(-)\otimes_{\mbb{L}}S$, where $S$ is a commutative unital ring. The formal group law associated to $\mrm{h}^*(-)$ is the formal group law induced by the ring homomorphism $\mbb{L}\to S$ (see \cref{rmk:induced-homomorphism}.)
	\end{dfn}
	
	We will see in \cref{theo:CH-iso} and \cref{theo:K-iso} that both $\mrm{CH}^*$ and $K^*$ are free oriented cohomology theories.
	
	\begin{theo}(\cite[Theorem 1.2.19]{LM07})\label{theo:CH-iso}
		The universal morphism $\Omega^*\to \mrm{CH}^*$ induces an isomorphism of oriented cohomology theories \[\Omega^*(-)\otimes_{\mbb{L}}\mbb{Z}\xrightarrow{\simeq}\mrm{CH}^*(-),\] where the morphism $\mbb{L}\to\mbb{Z}$ is induced by the additive formal group law $F_A(x,y)=x+y$ over $\mbb{Z}$ (see \cref{rmk:induced-homomorphism}.)
	\end{theo}
	
	\begin{theo}(\cite[Theorem 1.2.18]{LM07})\label{theo:K-iso}
		The universal morphism $\Omega^*\to K^*$ induces an isomorphism of oriented cohomology theories \[\Omega^*(-)\otimes_{\mbb{L}}\mbb{Z}[\beta,\beta^{-1}]\xrightarrow{\simeq}K^*(-),\] where the morphism $\mbb{L}\to\mbb{Z}[\beta,\beta^{-1}]$ is induced by the multiplicative-periodic formal group law $F_M(x,y)=x+y-\beta xy$ over $\mbb{Z}[\beta,\beta^{-1}]$ (see \cref{rmk:induced-homomorphism}.)
	\end{theo}

	Given a ring homomorphism $\mbb{L}\to S$, one could ask whether $\mrm{h}^*(-):=\Omega^*(-)\otimes_{\mbb{L}}S$ is an oriented cohomology theory. When $\mbb{L}\to \mbb{Z}[\beta]$ is induced by the connective formal group law $F(x,y)=x+y-\beta xy$ over $S=\mbb{Z}[\beta]$, the functor $\Omega^*(-)\otimes_{\mbb{L}}\mbb{Z}[\beta]$ does indeed define an oriented cohomology theory:
	
	\begin{dfn}(\cite{C08}, \cite{DL14})\label{dfn:CK-iso}
		The oriented cohomology theory $CK^*(-):=\Omega^*(-)\otimes_{\mbb{L}}\mbb{Z}[\beta]$ is called \textbf{algebraic connective $K$-theory}.
	\end{dfn}
	
	The oriented cohomology theory $CK^*(-)$ has a geometric description due to Cai \cite{C08} and Dai-Levine \cite{DL14}. In the papers \cite{C08}, \cite{DL14}, the authors defined the algebraic connective $K$-ring $CK^*(X)$ of a smooth variety $X$ in terms of coherent sheaves on $X$, showed that $CK^*$ defines an oriented cohomology theory, and proved that $CK^*(-)$ is isomorphic to $\Omega^*(-)\otimes_{\mbb{L}}\mbb{Z}[\beta]$. The geometric description of $CK^*(X)$ is not relevant for the current work, so we will refer the reader to the original sources (\cite{C08}, \cite{DL14}) for exposition on the geometric description of $CK^*(X)$ for a smooth variety $X$.
	
	\begin{lem}\label{lem:maps-to-both}
		Let $X\in\textbf{Sm}_{\mbb{C}}$. There are ring homomorphisms $\mrm{CK}^*(X)\to \mrm{CH}^*(X)$ and $\mrm{CK}^*(X)\to K^*(X)$.
	\end{lem}
	\begin{proof}	
		Recall the ring homomorphisms $\mbb{Z}[\beta]\to (\mbb{Z}[\beta])_{\beta}=\mbb{Z}[\beta,\beta^{-1}]$ and $\mbb{Z}[\beta]\to \mbb{Z}[\beta]/(\beta)\to \mbb{Z}$. This result follows from \cref{theo:CH-iso}, \cref{theo:K-iso}, and the realization of $\mrm{CK}^*(X)$ as $\Omega^*(X)\otimes_{\mbb{L}^*}\mbb{Z}[\beta]$.
	\end{proof}
	
	\begin{rmk}\label{rmk:Hudson}
		Some aspects of Schubert calculus in algebraic connective $K$-theory have been explored in \cite{H14}, \cite{HM18}, \cite{HIMN17}, \cite{HIMN20}. An important application to combinatorics was provided by Hudson in \cite{H14}, where he showed that a family of polynomials defined by Fomin and Kirillov \cite{FK} can be interpreted geometrically as classes in the algebraic connective $K$-ring of the flag variety. As explained in \cite[\S 3.1]{H14}, and using the conventions of \cite{H14}, the Fomin-Kirillov polynomials interpolate between ``double Schubert polynomials" (at $\beta=0$) and ``double Grothendieck polynomials" (at $\beta=-1$).
	\end{rmk}
	
	\subsection{Equivariant algebraic cobordism}\label{subsection:equivariant-algebraic-cobordism}
	
	In this subsection, we will recall equivariant algebraic cobordism and various technical results surrounding this theory. Equivariant algebraic cobordism was constructed independently by Krishna \cite{K12} and Malag\'{o}n-L\'{o}pez--Heller \cite{HM13}. The two constructions are equivalent, as explained in \cite[Remark 14]{HM13}. We will recall the construction of equivariant algebraic cobordism due to Krishna \cite{K12}, closely following the exposition in \cite[pp. 393--394]{KK13}. This construction is based on the definition of the algebraic cobordism ring of a classifying spaces due to Deshpande \cite{D09}, and uses ideas from work of Totaro \cite{T9797} and Edidin-Graham \cite{EG98}. Many of the results in this section regarding torus actions come from ideas of Brion, where he proved analogous results for the torus-equivariant Chow ring \cite{B97} (for example, \cref{theo:cobordism-GKM}).

	Let $G$ be a connected linear algebraic group over $\mbb{C}$. Let $\mathbf{Sch}_{\mbb{C}}^{G}$ be the category whose objects are separated and finite-type schemes over $\mbb{C}$ with a linear $G$-action, and whose morphisms are $G$-equivariant maps of schemes. Let $\textbf{Sm}_{\mbb{C}}^G$ be the full subcategory of $\mathbf{Sch}_{\mbb{C}}^{G}$ whose objects are smooth and quasi-projective schemes.

	\begin{dfn}
		Given an integer $j\geq 0$, let $V$ be a finite-dimensional $G$-representation, and let $U\subseteq V$ be an open subset such that the codimension of the complement $U\setminus V$ in $V$ is at least $j$. The pair $(V,U)$ is called a \textbf{good pair} corresponding to $j$ if $G$ acts freely on $U$ and the quotient $U/G$ is a quasi-projective scheme. 
	\end{dfn}
	
	\begin{lem}(\cite[Lemma 9]{EG98})
		Let $j\geq 0$ be an integer. Then a good pair $(V,U)$ corresponding to $j$ exists.
	\end{lem}

	\begin{lem}(\cite[Theorem 6.1]{K12})\label{lem:sequence-of-good-pairs}
		There exists a sequence of good pairs $\{(V_j,U_j)\}_{j\geq 0}$ for $G$, such that 
		\begin{enumerate}
			\item $V_{j+1}=V_j\oplus W_j$ as representations for $G$ with $\mrm{dim}(W_j)>0$, and 
			\item $U_j\oplus W_j\subseteq U_{j+1}$ as $G$-invariant open subsets.
		\end{enumerate}
	\end{lem}

	\begin{dfn}(\cite[Theorem 6.1]{K12})\label{dfn:equivariant-algebraic-cobordism}
		Let $X\in\textbf{Sm}_{\mbb{C}}^G$, and suppose $\{(V_j,U_j)\}_{j\geq 0}$ is a sequence of good pairs for $G$ satisfying (1) and (2) of \cref{lem:sequence-of-good-pairs}. Define the $i$-th $G$-equivariant algebraic cobordism group of $X$ as the inverse limit of groups
		\[\Omega_G^i(X):=\varprojlim \Omega^i\left(X\times^G U_j\right),\]
		where $X\times^G U_j:=(X\times U_j)/G$, and $G$ acts diagonally on $X\times U_j$.
		The \textbf{$G$-equivariant algebraic cobordism ring} of $X$ is $\Omega_G(X):=\bigoplus_{i\in\mbb{Z} }\Omega_G^i(X)$. Given a vector bundle $E\to X$ in $\textbf{Sm}_{\mbb{C}}^G$, there are unique elements $c_i^G(E)\in \Omega_G^i(X)$ called \textbf{$G$-equivariant Chern classes}, characterized by the axioms in \cite[Lemma 25]{HM13}, which include the ``Whitney sum formula," and the requirement that the $c_i$ commute with pullbacks.
	\end{dfn}
	
	\begin{rmk}
		In \cite{K12}, Krishna \textit{defined} the group $\Omega_G^i(X)$ using the so-called ``coniveau filtration," and then proved that this definition agrees with the definition given in \cref{dfn:equivariant-algebraic-cobordism}. The definition given in \cref{dfn:equivariant-algebraic-cobordism} is useful for performing computations.
	\end{rmk}

	Following the conventions in \cite[\S 7.3]{K12},
	for a finite-dimensional complex manifold $X$ with the action of a complex Lie group $H$, the \textbf{equivariant complex cobordism ring} of $X$ is defined as
	\[MU^*_H(X):=MU^*(X\times^H EH),\]
	where $EH$ is any contractible space on which $H$ acts freely. 
	\begin{egg}(\cite[Proposition 7.4]{K12})\label{egg:eq-MU-iso}
		Let $X\in\textbf{Sm}_{\mbb{C}}^{G}$, and let $G(\mbb{C})$ be the complex Lie group associated to the algebraic group $G$. Denote by  $X(\mbb{C})$ the complex manifold consisting of the complex points of $X$. If the $G(\mbb{C})$-equivariant cohomology ring $H_{G(\mbb{C})}^*(X(\mbb{C}))$ of $X(\mbb{C})$ is torsion-free, then there is a natural homomorphism of graded rings
		$\Omega_G^*(X)\to MU_{G(\mbb{C})}^{2*}(X(\mbb{C}))$.
	\end{egg}
	
	Now let us assume that $G=A:=(\mbb{C}^\times)^n$ is an algebraic torus, with weight lattice $\Lambda^A:=\mrm{Hom}(A,\mbb{C}^\times)$. Let $\mcal{L}_\lambda$ be the $A$-equivariant line bundle on $\mrm{pt}=\mrm{Spec}(\mbb{C})$ corresponding to the one-dimensional representation of $T$ with weight $\lambda\in\Lambda^A$. The $A$-equivariant algebraic cobordism ring of a point $\Omega_A^*(\mrm{pt})$ is generated by the first Chern classes of the equivariant line bundles $c_1^A(\mcal{L}_\lambda)$. The ring $\Omega_A^*(\mrm{pt})$ has an algebraic description in terms of the formal affine Demazure algebra:
	
	\begin{theo}\label{theo-algebraic-structure-of-cobordism-of-point}(\cite[Theorem 3.3]{CZZ3})
		There is an $\mbb{L}$-algebra isomorphism
		\[\mbb{L}\llbracket\Lambda^A\rrbracket_{F_U}\xrightarrow{\simeq} \Omega_{A}(\mrm{pt}),\quad x_\lambda\mapsto c_1^A(\mcal{L}_\lambda). \]
	\end{theo}

	\begin{dfn}(\cite[\S 4.2]{K122})
		Let $X\in\textbf{Sm}_{\mbb{C}}^{A}$. Then $X$ is called \textbf{$A$-equivariantly filtrable} if the fixed point locus $X^A$ for the action of $A$ on $X$ is smooth and projective, there is a numbering $X^A=\sqcup_{m=0}^nZ_m$ of the connected components of the fixed point locus, a filtration on $X$
		\[\emptyset =X_{n+1}\subsetneq X_{n}\subsetneq X_{n-1}\subsetneq \dotsb\subsetneq X_1\subsetneq X_0=X\] 
		by $T$-invariant and closed subvarieties $X_i$ of $X$, and maps $\phi_m\colon (X_m\setminus X_{m-1})\to Z_m$ for $0\leq m\leq n$ which are all $A$-equivariant vector bundles.
	\end{dfn}
	
	\begin{theo}(\cite{B73})\label{theo:BB-decomp}
		Suppose $X$ is a smooth projective algebraic variety over $\mbb{C}$ with an algebraic action of $A$. Then $X$ is $A$-equivariantly filtrable.
	\end{theo}
	
	\begin{theo}(\cite[Theorem 3.7]{KK13})\label{theo:alg-to-complex-iso}
		Suppose that $X\in\textbf{Sm}_{\mbb{C}}^{A}$, and assume $X$ is $A$-equivariantly filtrable. Then the graded ring homomorphism $\Omega_A^*(X)\to MU_{A(\mbb{C})}^{2*}(X(\mbb{C}))$ of \cref{egg:eq-MU-iso} is an isomorphism.
	\end{theo}
	
	Let $X\in\textbf{Sm}_{\mbb{C}}^{A}$, and assume that $X$ has finitely many $A$-fixed points $F$. Given a fixed point $f\in F$, let $\iota_f\colon f\to X$ be the inclusion of $f$ into $X$. The pullback along $\iota$ is denoted $\iota_f^*\colon \Omega_A^*(X)\to \Omega_A^*(f)\simeq \mbb{L}\llbracket\Lambda^A\rrbracket_{F_U}$. 
	
	\begin{theo}\label{theo:cobordism-GKM}(\cite[Theorem 7.8]{K122})
		Let $X$ be a smooth $A$-equivariantly filtrable variety, and assume that $X$ has finitely many $A$-fixed points $F$ and finitely many $A$-invariant curves. Then the localization map 
		\[\iota^*:=\oplus_{f\in F} \iota_f^*\colon \Omega_A^*(X)\to \bigoplus_{f\in F}\Omega_A^*(\mrm{pt})\] is injective, and its image is the set of tuples $(g_f)_{f\in F}$ satisfying the conditions that $c_1^A(\mcal{L}_\lambda) | (g_f-g_{f'})$, whenever the fixed points $f$ and $f'$ lie in the same $A$-invariant irreducible curve and $\lambda$ is the weight of $A$ acting on this curve. 
	\end{theo}
	
	\begin{rmk}\label{rmk:GKM-explanation}
		The conditions characterizing the image of $\iota^*$ are called the \textit{GKM conditions}, after Goresky--Kottwitz--MacPherson \cite{GKM98}. See also \cref{theo:GKM-for-H} and \cref{theo:GKM-for-K} below.
	\end{rmk}
	
	We will now compare \cref{theo:cobordism-GKM} to its topological counterparts in $A$-equivariant cohomology and $A$-equivariant $K$-theory. See also \cref{rmk:GMK-cobordism}. 
	
	\begin{theo}\label{theo:GKM-for-H}(\cite{GKM98}, cf. \cite[\S 2.2]{KT03})
		Let $X$ be a smooth projective variety with an algebraic action of $A$, and assume that $X$ has finitely many $A$-fixed points $F$. Then the localization map 
		\[\iota_{H_A}^*:=\oplus_{f\in F} (\iota_f^*)_{H_A}\colon H_A^*(X)\to \bigoplus_{f\in F}H_A^*(\mrm{pt})\]
		is injective, and its image consists of the set of tuples $(g_f)_{f\in F}$ such that, whenever $f$ and $f'$ are contained in an $A$-invariant irreducible curve in $X$ with weight $\lambda$, the difference $g_f-g_{f'}$ is divisible by $\lambda$ in the ring $H_A(\mrm{pt})$.
	\end{theo}
	
	\begin{theo}\label{theo:GKM-for-K}(\cite[Corollary A.5]{R03})
		Let $X$ be a smooth projective variety with an algebraic action of $A$, and assume that $X$ has finitely many $A$-fixed points $F$. Then the localization map 
		\[\iota_{K_A}^*:=\oplus_{f\in F} (\iota_f^*)_{K_A}\colon K_A(X)\to \bigoplus_{f\in F}K_A(\mrm{pt})\]
		is injective, and its image consists of the set of tuples $(g_f)_{f\in F}$ such that, whenever $f$ and $f'$ are contained in an $A$-invariant irreducible curve in $X$ with weight $\lambda$, the difference $g_f-g_{f'}$ is divisible by $1-e^{-\lambda}$ in the ring $K_A(\mrm{pt})$.
	\end{theo}

	\begin{rmk}\label{rmk:GMK-cobordism}
		The theorem analogous to \cref{theo:GKM-for-H} and \cref{theo:GKM-for-K} for $A$-equivariant complex cobordism can be found in \cite{HHH05}. 
	\end{rmk}

	\subsection{Connective motivic Chern classes revisited}\label{subsection:geometric-stabclasses}
	
	In this subsection, we will show that the connective motivic Chern classes of Schubert cells can be viewed as canonical elements in a quotient of the $\widehat{T}$-equivariant algebraic cobordism ring of $T^*(G/P_\Theta)$ (see \cref{theo:geometric-connective-stable}). 
	
	We note that if one would like to work 
	with $\widehat{T}(\mbb{C})$-equivariant complex cobordism rather than $\widehat{T}$-equivariant algebraic cobordism, this can be done without any trouble:
	\begin{lem}\label{lem:cotangent-iso}
		The ring homomorphism $\Omega_{\widehat{T}}^*(T^*(G/P_\Theta))\to MU_{\widehat{T}(\mbb{C})}^*((T^*(G/P_\Theta))(\mbb{C}))$ of \cref{egg:eq-MU-iso} is an isomorphism.
	\end{lem}
	\begin{proof}
		The cotangent bundle $T^*(G/P_\Theta)$ is smooth and $\widehat{T}$-equivariantly filtrable. Now \cref{theo:alg-to-complex-iso} says that the natural ring homomorphism $\Omega_{\widehat{T}}^*(T^*(G/P_\Theta))\to MU_{\widehat{T}(\mbb{C})}^*(T^*(G/P_\Theta)(\mbb{C}))$ is an isomorphism.
	\end{proof}
	
	Consider the localization map $\iota^*\colon \Omega_{\widehat{T}}^*(T^*(G/P_\Theta))\to \bigoplus_{w\in W^\Theta}\Omega_{\widehat{T}}^*(\mrm{pt})$, where we have identified $W^\Theta$ with the set of $\widehat{T}$-fixed points of $T^*(G/P_\Theta)$. By \cref{theo-algebraic-structure-of-cobordism-of-point}, we identify $\Omega_{\widehat{T}}^*(\mrm{pt})$ with $\mbb{L}\llbracket \Lambda\oplus \mbb{Z}q\rrbracket_{F_U}\simeq \mbb{L}\llbracket \Lambda\rrbracket_{F_U}\llbracket x_q\rrbracket$. 
	
	\begin{theo}\label{theo:localization}
		The localization map below is injective: 
		\[\iota^* \colon \Omega_{\widehat{T}}(T^*(G/P_\Theta))\hookrightarrow \bigoplus_{w\in W^\Theta}\Omega_{\widehat{T}}(\mrm{pt})\simeq \bigoplus_{w\in W^\Theta}\mbb{L}\llbracket \Lambda\rrbracket_{F_U}\llbracket x_q\rrbracket.\]
		The image of $\iota^*$ is the set of sequences $(f_w)_{w\in W^\Theta}$ such that $x_{\alpha} | (f_w-f_{s_\alpha w})$ over all roots $\alpha\in\Sigma_\Theta$.
	\end{theo}
	\begin{proof}
		Note that $T^*(G/P_\Theta)$ is $\widehat{T}$-equivariantly filtrable, and the action of $\widehat{T}$ on $T^*(G/P_\Theta)$ has finitely many fixed points and finitely many invariant curves. The weight of the $\widehat{T}$-invariant $\mbb{P}^1$ that connects two fixed points $f_w$ and $f_{s_\alpha w}$ is precisely the root $\lambda=\alpha$. This result now follows from \cref{theo:cobordism-GKM}.
	\end{proof}

	Let $(S,F)$ be a formal group law.
	
	\begin{dfn}\label{dfn:GKM-conditions}
		Let $\widehat{\mcal{F}}(S,F)$ be the subring of $\oplus_{\lambda\in\Lambda} S\llbracket\Lambda\rrbracket_F\llbracket x_q\rrbracket$ consisting of sequences $(f_w)_{w\in W^\Theta}$ that satisfy the \textbf{GKM conditions}, i.e., such that $x_\alpha|(f_w-f_{s_\alpha w})$ over all roots $\alpha\in\Sigma_\Theta$.
	\end{dfn}
	
	\begin{rmk}
		It follows from \cref{theo:localization} that
		the ring $\widehat{\mcal{F}}(\mbb{L},F_U)$ is isomorphic to $\Omega_{\widehat{T}}(T^*(G/P_\Theta))$.
	\end{rmk}
	
	The formal group law $(S,F)$ induces a ring homomorphism $\mbb{L}\to S$ (see \cref{rmk:induced-homomorphism}), which, as explained in \cite[\S 2.5]{CPZ} induces a ring homomorphism $\overline{h}_{(S,F)}\colon \mbb{L}\llbracket \Lambda\rrbracket_{F_U}\llbracket x_q\rrbracket\to S\llbracket \Lambda \rrbracket_{F}\llbracket x_q\rrbracket$ that sends $x_\lambda\mapsto x_\lambda$ for all $\lambda\in\Lambda$ and sends $x_q\mapsto x_q$. If the map $\mbb{L}\to S$ is \textit{surjective}, then $\overline{h}_{(S,F)}$ is surjective as well. We define the following ring homomorphism:
	\[h_{(S,F)}:=\oplus_{w\in W^\Theta}\overline{h}_{(S,F)}\colon \bigoplus_{w\in W^\Theta}\Omega_{\widehat{T}}^*(\mrm{pt})\to \bigoplus_{w\in W^\Theta}S\llbracket \Lambda\rrbracket_{F}\llbracket x_q\rrbracket;\]
	
	\begin{lem}\label{lem:surjective-restriction}
		The ring homomorphism $h_{(S,F)}$ restricts to the ring homomorphism 
		\[h_{(S,F)}|_{\widehat{\mcal{F}}(\mbb{L},F_U)}\colon \Omega_{\widehat{T}}(T^*(G/P_\Theta))=\widehat{\mcal{F}}(\mbb{L},F_U)\to \widehat{\mcal{F}}(S,F).\]
	\end{lem}
	\begin{proof}
		Consider a sequence $(f_w)_{w\in W^\Theta}$ in $\widehat{\mcal{F}}(\mbb{L},F_U)$. By definition, for each $w\in W^\Theta$, there is $g_{w,s_\alpha w}\in\Omega_{\widehat{T}}^*(\mrm{pt})$ such that $f_w-f_{s_\alpha w}=x_\alpha g_{w,s_\alpha w}$. Thus, the equation $\overline{h}_{(S,F)}(f_w)-\overline{h}_{(S,F)}(f_{s_\alpha w})=x_{\alpha} \overline{h}_{(S,F)}(g_{w,s_\alpha w})$ holds in $S\llbracket\Lambda\rrbracket_F\llbracket x_q\rrbracket$. As a consequence, the image $h_{(S,F)}((f_w)_{w\in W^\Theta})=(\overline{h}_{(S,F)}(f_w))_{w\in W^\Theta}$ lies in $\widehat{\mcal{F}}(S,F)$.
	\end{proof}

	Recall the connective motivic Chern classes $\mrm{St}_w$ defined in \cref{subsection:stabclass}. In \cref{prop:GKM}, we showed that the elements $\mrm{St}_w$ satisfy the GKM conditions, so $\mrm{St}_w\in\widehat{\mcal{F}}_{(\mbb{Z}[\beta],F_C)}$. We showed in \cref{rmk:general-motivic-Chern-classes} that, once we fix a reduced word $I_w$ for each $w\in W^\Theta$ (and replace $\beta$ by the element $-a_{1,1}$ of $\mbb{L}^*$), there are classes $\mrm{St}_{I_w}^{(\mbb{L},F_U)}$ that lie in $\widehat{\mcal{F}}(\mbb{L},F_U)=\Omega_{\widehat{T}}(T^*(G/P_\Theta))$. By construction of the classes $\mrm{St}_{I_w}^{(\mbb{L},F_U)}$ (and replacing $\beta$ by the element $-a_{1,1}$ in $\mbb{L}^*$; see \cref{rmk:general-motivic-Chern-classes}), we have $h_{(\mbb{Z}[\beta],F_C)}|_{\widehat{\mcal{F}}(\mbb{L},F_U)}(\mrm{St}_{I_w}^{(\mbb{L},F_U)})=\mrm{St}_w$. In particular, the connective motivic Chern classes $\mrm{St}_w$ lie in the image of $h_{(\mbb{Z}[\beta],F_C)}|_{\widehat{\mcal{F}}(\mbb{L},F_U)}$. Let $I_{(\mbb{Z}[\beta],F_C)}$ be the kernel of the ring homomorphism $h_{(\mbb{Z}[\beta],F_C)}|_{\widehat{\mcal{F}}(\mbb{L},F_U)}$. Thus, the ring $\Omega_{\widehat{T}}(T^*(G/P_\Theta))/I_{(\mbb{Z}[\beta],F_C)}$ is isomorphic to the image of $h_{(\mbb{Z}[\beta],F_C)}|_{\widehat{\mcal{F}}(\mbb{L},F_U)}$, so we obtain the following:
	\begin{theo}\label{theo:geometric-connective-stable}
		The connective motivic Chern classes $\mrm{St}_w$ can be viewed as canonical elements in the ring
		\[\Omega_{\widehat{T}}(T^*(G/P_\Theta))/I_{(\mbb{Z}[\beta],F_C)}.\]
	\end{theo}

	The following technical lemma is used in \cref{rmk:inclusion of rings0}.
	
	\begin{lem}\label{lem:inclusion of rings}
		The following are true:
		\begin{enumerate}
			\item The ring $\mrm{CH}_{T}(\mrm{pt})\simeq \mbb{Z}[y_1,\dotsc,y_n]$ injects into the completion $\widehat{\mrm{CH}}_T(\mrm{pt})=\mbb{Z}\llbracket y_1,\dotsc,y_n\rrbracket$ at the ideal $I=(y_1,\dotsc,y_n)$.
			\item The ring $K_{T}(\mrm{pt})$ injects into the completion $\widehat{K}_{T}(\mrm{pt})$ at the ideal $I=(e^\lambda-1)$ over all $\lambda\in\Lambda$.
		\end{enumerate}
	\end{lem}
	\begin{proof}
		(1). This is straightforward. (2). The kernel of the natural map $\pi\colon K_{T}(\mrm{pt})\to\widehat{K}_{T}(\mrm{pt})$ is precisely the intersection $\mrm{ker}(\pi)=\cap_{m=1}^\infty I^m$. As the ring $K_{T}(\mrm{pt})$ is a domain, it follows from Krull's intersection theorem (see, e.g., \cite[Corollary 5.4]{E95}) that the intersection $\cap_{m=1}^\infty I^m=(0)$.
	\end{proof}

	\begin{rmk}\label{rmk:inclusion of rings0}
		Recall the additive formal group law $(\mbb{Z},F_A)$ and the multiplicative-periodic formal group law $(\mbb{Z},F_M)$. By definition, the rings $\widehat{\mcal{F}}(\mbb{Z},F_A)$ and  $\widehat{\mcal{F}}(\mbb{Z},F_M)$ are the subrings of $\bigoplus_{w\in W^\Theta}\mbb{Z}\llbracket \Lambda\rrbracket_{F_A}\llbracket x_q\rrbracket$ and $\bigoplus_{w\in W^\Theta}\mbb{Z}\llbracket \Lambda\rrbracket_{F_M}\llbracket x_q\rrbracket$, respectively, consisting of sequence $(f_w)_{w\in W^\Theta}$ such that $x_\alpha$ divides $(f_w-f_{s_\alpha w})$ over all roots $\alpha\in\Sigma$.
		\cref{lem:inclusion of rings} explains that $H_T(\mrm{pt})$ injects into $\mbb{Z}\llbracket \Lambda\rrbracket_{F_A}$, and that $K_T(\mrm{pt})$ injects into $\mbb{Z}\llbracket \Lambda\rrbracket_{F_M}$. Taking direct sums gives injections $\oplus_{\lambda \in W^\Theta} H_T(\mrm{pt})\hookrightarrow \oplus_{\lambda \in W^\Theta} \mbb{Z}\llbracket \Lambda\rrbracket_{F_A}$ and $\oplus_{\lambda \in W^\Theta} K_T(\mrm{pt})\hookrightarrow \oplus_{\lambda \in W^\Theta} \mbb{Z}\llbracket \Lambda\rrbracket_{F_M}$. Thus, it follows from \cref{theo:GKM-for-H} that there is an injection $H_{T}(G/P_\Theta)[\hbar]\hookrightarrow \widehat{\mcal{F}}(\mbb{Z},F_A)[\tfrac{1}{x_q}]$, where $\hbar\mapsto x_q$, and it follows from \cref{theo:GKM-for-K} that there is an injection $K_T(G/P_\Theta)[q^{\pm}]\hookrightarrow \widehat{\mcal{F}}(\mbb{Z},F_M)[\tfrac{1}{x_q}]$, where $q\mapsto x_q$. As explained in \cref{rmk:connective-stable-class}, the ring homomorphisms 
		\[\widehat{\mcal{F}}(\mbb{Z}[\beta],F_C)\to \widehat{\mcal{F}}(\mbb{Z},F_A)[\tfrac{1}{x_q}]\quad\text{ and}\quad \widehat{\mcal{F}}(\mbb{Z}[\beta],F_C)\to \widehat{\mcal{F}}(\mbb{Z},F_M)[\tfrac{1}{x_q}]\] send $\overline{\mrm{St}}_w$ to the corresponding Chern-Schwartz-MacPherson class in $H_{\widehat{T}}(T^*(G/P_\Theta))$ (homogenized by a factor $(\hbar^2)^{\ell(w)}$) and to the corresponding motivic Chern class in $K_{\widehat{T}}(T^*(G/P_\Theta))$, respectively.
	\end{rmk}

	\begin{rmk}\label{rmk:Karpenko}
		The definition of \textit{equivariant algebraic connective $K$-theory} $CK_{(-)}(-)$ is given in \cite{KM22}. Based on \cref{lem:maps-to-both}, \cref{rmk:Hudson}, and \cref{rmk:inclusion of rings0}, we expect that $\mrm{CK}_{\widehat{T}}(T^*(G/P_\Theta))$ can be identified via the localization map with the $\mrm{CK}_{\widehat{T}}(\mrm{pt})$-subalgebra of $\widehat{\mcal{F}}_{(\mbb{Z}[\beta],F_C)}(T^*(G/P_{\Theta}))[\tfrac{1}{x_q}]$ generated by the elements ${\mrm{St}}_w$. Furthermore, we expect that there is an equivariant algebraic connective $K$-theory analogue of the motivic Chern natural transformation that associates to a Schubert cell $X_w^\circ$ in $G/P_\Theta$ an element ${\mrm{St}}_w^{CK_{\widehat{T}}}$ in the $\widehat{T}$-equivariant algebraic connective $K$-ring $CK_{\widehat{T}}(T^*(G/P_\Theta))$ which agrees with the element ${\mrm{St}}_w$ in the ring $\widehat{\mcal{F}}_{(\mbb{Z}[\beta],F_C)}(T^*(G/P_{\Theta}))[\tfrac{1}{x_q}]$ under the localization map. Proving the statements described in this remark is work in progress of the author of this paper and Anubhav Nanavaty.
	\end{rmk}

	\section{Future work}\label{section:future-work}
	
	In this section, we will describe several directions one could take this work.

	\subsection{Future work} In \cite{KZJ17}, \cite{KZJ21}, Knutson and Zinn-Justin prove a puzzle formula for the structure constants in the basis of motivic Segre classes in $\widehat{K}_{\widehat{T}}(T^*(G/P_\Theta))$, where $G/P_\Theta$ is a $d$-step flag variety for $d=1,2,3,4$. The puzzle formula when $d=2$ (resp. $d=3$, resp. $d=4$) is related to the representation theory of $U_q(\mathfrak{d}_4)$ (resp. $U_q(\mathfrak{e}_6)$, resp. $U_q(\mathfrak{e}_8)$). This motivates the following problem:
	\begin{problem}
		Give a puzzle formula for the structure constants in the basis of connective motivic Segre classes for $d=2,3,4$, using the representation theory of the multi-parameter quantum groups of types $\mathfrak{d}_4$, $\mathfrak{e}_6$, and $\mathfrak{e}_8$, respectively.
	\end{problem}
	
	In \cite{KZJ23}, Knutson and Zinn-Justin prove a puzzle formula for the structure constants arising in the product of two motivic Segre classes $\mrm{St}_u$ and $\mrm{St}_v$, where $u$ and $v$ are permutations in $S_n$ that have ``separated descents" or ``almost-separated descents." This motivates the following problem:
	\begin{problem}
		Give a puzzle formula for the structure constants that arise in the product of two connective motivic Segre classes $\mrm{St}_u$ and $\mrm{St}_v$, where $u$ and $v$ are permutations in $S_n$ that have ``separated descents" or ``almost-separated descents."
	\end{problem}
	
	In \cite{N01}, Nakajima defined an action of the quantum affine algebra $U_q(\widehat{a}_2)$ on the $\widehat{T}$-equivariant $K$-ring of the Nakajima quiver variety $\mcal{M}:=\sqcup_{k=0}^{2}T^*(\mrm{Gr}(k,2))$, which was applied in \cite{KZJ21} to obtain a geometric prove of the puzzle formula for multiplying motivic Segre classes of Schubert cells. Observe
	\[K_{\widehat{T}}(\mcal{M})\simeq \bigoplus_{k=0}^2 K_{\widehat{T}}(T^*(\mrm{Gr}(k,2))).\]
	Therefore, the action of $U_q(\widehat{a}_2)$ on $K_{\widehat{T}}(\mcal{M})$ is completely determined by the actions of generators $E_i,F_i,K_i$ of $U_q(\widehat{a}_2)$ (see Appendix \ref{section:quantum-group}) on each $K_{\widehat{T}}(T^*(\mrm{Gr}(k,n)))$ component. The action of the $E_i,F_i,K_i$ generators is completely determined by their effects on the basis of motivic Chern classes of the Schubert cells in $K_{\widehat{T}}(T^*(\mrm{Gr}(k,n)))$. This leads to a natural question, which the author plans to pursue in future work:
	
	\begin{problem}\label{problem:convolution-structure}
		The ring $K_{\widehat{T}}(\mcal{M})\otimes_{\mbb{Z}}\mbb{Z}[\beta,\beta^{-1}]$ contains the $\widehat{T}$-equivariant connective $K$-ring, $\mrm{CK}_{\widehat{T}}(\mcal{M})$, of $\mcal{M}$ as a subring. Is there an action of the multi-parameter quantum group $U_{\textbf{q}}(\widehat{a}_2)$ on the ring $\mrm{CK}_{\widehat{T}}(\mcal{M})$, with its convolution algebra structure, involving a parameter $\beta$ that specializes to the action of $U_q(\widehat{a}_2)$ on $K_{\widehat{T}}(\mcal{M})$ when $\beta=1$?
	\end{problem}

	We saw in \cref{theo:CSM-classes} and \cref{theo:motivic-Chern-classes} that the Chern-Schwartz-MacPherson classes and motivic Chern classes of Schubert cells are closely related to cohomological and $K$-theoretic stable envelopes for $T^*(G/P_\Theta)$. The stable classes defined by Maulik and Okounkov \cite{MO19}, \cite{O17} exist for any ``symplectic resolution" and are characterized by three axioms: \textit{degree}, \textit{support}, and \textit{normalization}. Note that the construction of a motivic Chern transformation for equivariant algebraic connective $K$-theory is work in progress of the author of this paper; see \cref{rmk:Karpenko} for details. We pose the following:
	\begin{problem}
		Can the methods used in this paper by working in the GKM ring of \cref{dfn:GKM-conditions} be extended to define and study ``connective" motivic Chern classes or ``connective" stable classes for other families of symplectic resolutions?
	\end{problem}

	Schubert calculus in the $\widehat{T}$-equivariant elliptic cohomology of $T^*(G/P_\Theta)$ has been developed in \cite{RW20}, \cite{KRW20}, \cite{LZZ23}, \cite{ZZ23}, \cite{LXZ25}. In \cite[Appendix A]{LXZ25}, the authors explain how to recover motivic Chern classes of Schubert cells as limits of the so-called ``elliptic stable classes." This leads to a natural question:
	
	\begin{problem}
		Is there a relationship between the elliptic stable classes and the connective motivic Chern classes?
	\end{problem}

	\appendix

	\section{The Yang-Baxter equations}\label{section:YBE}
	
	\subsection{The equations}\label{subsection:the-equations}
	The following equations were verified in the proof of \cref{prop:ybe} using SageMath \cite{G25}:
	
	\begin{itemize}
		\item \textit{Yang-Baxter equations:}
		\begin{align*}
			(R_{r,g}(\beta,q^2z_3/z_2)\otimes \mrm{id}_{3})&\circ ( \mrm{id}_3\otimes R_{r,g}(\beta,q^2z_3/z_1))\circ 
			( R_{r,r}(\beta,z_2/z_1)\otimes  \mrm{id}_3)\\&=(\mrm{id}_3\otimes R_{r,r}(\beta,z_2/z_1))\circ
			(R_{r,g}(\beta,q^2z_3/z_1)\otimes \mrm{id}_3)\circ
			( \mrm{id}_3\otimes R_{r,g}(\beta,q^2z_3/z_2)).
		\end{align*}	
		\begin{align*}
			(R_{g,r}(\beta,z_3/(q^2z_2))\otimes \mrm{id}_{3})&\circ ( \mrm{id}_3\otimes R_{r,r}(\beta,z_3/z_1))\circ 
			( R_{r,g}(\beta,q^2z_2/z_1)\otimes  \mrm{id}_3)\\&=(\mrm{id}_3\otimes R_{r,g}(\beta,q^2z_2/z_1))\circ
			(R_{r,r}(\beta,z_3/z_1)\otimes \mrm{id}_3)\circ
			( \mrm{id}_3\otimes R_{g,r}(\beta,z_3/(q^2z_2)).
		\end{align*}				
		\begin{align*}
			(R_{r,r}(\beta,z_3/z_2)\otimes \mrm{id}_{3})&\circ ( \mrm{id}_3\otimes R_{g,r}(\beta,z_3/(q^2z_1))\circ 
			( R_{g,r}(\beta,z_2/(q^2z_1))\otimes  \mrm{id}_3)\\&=(\mrm{id}_3\otimes (R_{g,r}(\beta,z_2/(q^2z_1)))\circ
			(R_{g,r}(\beta,z_3/(q^2z_1))\otimes \mrm{id}_3)\circ
			( \mrm{id}_3\otimes R_{r,r}(\beta,z_3/z_2)).
		\end{align*}		
		\begin{align*}
			(R_{g,r}(\beta,z_3/(q^2z_2))\otimes \mrm{id}_{3})&\circ ( \mrm{id}_3\otimes R_{g,r}(\beta,z_3/(q^2z_1)))\circ 
			( R_{g,g}(\beta,z_2/z_1)\otimes  \mrm{id}_3)\\&=(\mrm{id}_3\otimes 	(R_{g,g}(\beta,z_2/z_1))\circ
			(R_{g,r}(\beta,z_3/(q^2z_1)))\otimes \mrm{id}_3)\circ
			( \mrm{id}_3\otimes R_{g,r}(\beta,z_3/(q^2z_2))).
		\end{align*}		
		\begin{align*}
			(R_{r,g}(\beta,q^2z_3/z_2)\otimes \mrm{id}_{3})&\circ ( \mrm{id}_3\otimes R_{g,g}(\beta,z_3/z_1))\circ 
			(R_{g,r}(\beta,z_2/(q^2z_1))\otimes  \mrm{id}_3)\\&=(\mrm{id}_3\otimes 	R_{g,r}(\beta,z_2/(q^2z_1)))\circ
			(R_{g,g}(\beta,z_3/z_1)\otimes \mrm{id}_3)\circ
			( \mrm{id}_3\otimes R_{r,g}(\beta,q^2z_3/z_2)).
		\end{align*}	
		\begin{align*}
			(R_{g,g}(\beta,z_3/z_2)\otimes \mrm{id}_{3})&\circ ( \mrm{id}_3\otimes R_{r,g}(\beta,q^2z_3/z_1))\circ 
			(R_{r,g}(\beta,q^2z_2/z_1)\otimes  \mrm{id}_3)\\&=(\mrm{id}_3\otimes 	R_{r,g}(\beta,q^2z_2/z_1))\circ
			(R_{r,g}(\beta,q^2z_3/z_1)\otimes \mrm{id}_3)\circ
			( \mrm{id}_3\otimes R_{g,g}(\beta,z_3/z_2)).
		\end{align*}

		\begin{align*}
			(R_{b,r}(\beta,z_3/(qz_2))\otimes \mrm{id}_{3})&\circ ( \mrm{id}_3\otimes R_{b,r}(\beta,z_3/(qz_1)))\circ 
			(R_{b,b}(\beta,z_2/z_1)\otimes  \mrm{id}_3)\\&=(\mrm{id}_3\otimes R_{b,b}(\beta,z_2/z_1))\circ
			(R_{b,r}(\beta,z_3/(qz_1)\otimes \mrm{id}_3)\circ
			( \mrm{id}_3\otimes R_{b,r}(\beta,z_3/(qz_2))).
		\end{align*}
		\begin{align*}
			(R_{r,b}(\beta,qz_3/z_2))\otimes \mrm{id}_{3})&\circ ( \mrm{id}_3\otimes R_{b,b}(\beta,z_3/z_1))\circ 
			(R_{b,r}(\beta,z_2/(qz_1))\otimes  \mrm{id}_3)\\&=(\mrm{id}_3\otimes R_{b,r}(\beta,z_2/(qz_1)))\circ
			(R_{b,b}(\beta,z_3/z_1)\otimes \mrm{id}_3)\circ
			( \mrm{id}_3\otimes R_{r,b}(\beta,qz_3/z_2)).
		\end{align*}				
		\begin{align*}
			(R_{b,b}(\beta,z_3/z_2))\otimes \mrm{id}_{3})&\circ ( \mrm{id}_3\otimes R_{r,b}(\beta,qz_3/z_1))\circ 
			(R_{r,b}(\beta,qz_2/z_1)\otimes  \mrm{id}_3)\\&=(\mrm{id}_3\otimes R_{r,b}(\beta,qz_2/z_1))\circ
			(R_{r,b}(\beta,qz_3/z_1)\otimes \mrm{id}_3)\circ
			( \mrm{id}_3\otimes R_{b,b}(\beta,z_3/z_2)).
		\end{align*}		
		\begin{align*}
			(R_{r,b}(\beta,qz_3/z_2))\otimes \mrm{id}_{3})&\circ ( \mrm{id}_3\otimes R_{r,b}(\beta,qz_3/z_1))\circ 
			(R_{r,r}(\beta,z_2/z_1)\otimes  \mrm{id}_3)\\&=(\mrm{id}_3\otimes R_{r,r}(\beta,z_2/z_1))\circ
			(R_{r,b}(\beta,qz_3/z_1)\otimes \mrm{id}_3)\circ
			( \mrm{id}_3\otimes R_{r,b}(\beta,qz_3/z_2)).
		\end{align*}		
		\begin{align*}
			(R_{b,r}(\beta,z_3/(qz_2))\otimes \mrm{id}_{3})&\circ ( \mrm{id}_3\otimes R_{r,r}(\beta,z_3/z_1))\circ 
			(R_{r,b}(\beta,qz_2/z_1)\otimes  \mrm{id}_3)\\&=(\mrm{id}_3\otimes 	(R_{r,b}(\beta,qz_2/z_1))\circ
			(R_{r,r}(\beta,z_3/z_1)\otimes \mrm{id}_3)\circ
			( \mrm{id}_3\otimes R_{b,r}(\beta,z_3/(qz_2))).
		\end{align*}
		\begin{align*}
			(R_{r,r}(\beta,z_3/z_2)\otimes \mrm{id}_{3})&\circ ( \mrm{id}_3\otimes R_{b,r}(\beta,z_3/(qz_1)))\circ 
			(R_{b,r}(\beta,z_2/(qz_1))\otimes  \mrm{id}_3)\\&=(\mrm{id}_3\otimes R_{b,r}(\beta,z_2/(qz_1)))\circ
			(R_{b,r}(\beta,z_3/(qz_1))\otimes \mrm{id}_3)\circ
			( \mrm{id}_3\otimes R_{r,r}(\beta,z_3/z_2)).
		\end{align*}
		
		\begin{align*}
			(R_{b,b}(\beta,z_3/z_2)\otimes \mrm{id}_{3})&\circ ( \mrm{id}_3\otimes R_{g,b}(\beta,z_3/(qz_1)))\circ 
			(R_{g,b}(\beta,z_2/(qz_1))\otimes  \mrm{id}_3)\\&=(\mrm{id}_3\otimes R_{g,b}(\beta,z_2/(qz_1))\circ
			(R_{g,b}(\beta,z_3/(qz_1)))\otimes \mrm{id}_3)\circ
			( \mrm{id}_3\otimes R_{b,b}(\beta,z_3/z_2)).
		\end{align*}
		\begin{align*}
			(R_{g,b}(\beta,z_3/(qz_2))\otimes \mrm{id}_{3})&\circ ( \mrm{id}_3\otimes R_{b,b}(\beta,z_3/z_1))\circ 
			(R_{b,g}(\beta,qz_2/z_1)\otimes  \mrm{id}_3)\\&=(\mrm{id}_3\otimes R_{b,g}(\beta,qz_2/z_1))\circ
			(R_{b,b}(\beta,z_3/z_1)))\otimes \mrm{id}_3)\circ
			( \mrm{id}_3\otimes R_{g,b}(\beta,z_3/(qz_2))).
		\end{align*}
		\begin{align*}
			(R_{b,g}(\beta,qz_3/z_2)\otimes \mrm{id}_{3})&\circ ( \mrm{id}_3\otimes R_{b,g}(\beta,qz_3/z_1))\circ 
			(R_{b,b}(\beta,z_2/z_1)\otimes  \mrm{id}_3)\\&=(\mrm{id}_3\otimes R_{b,b}(\beta,z_2/z_1))\circ
			((R_{b,g}(\beta,qz_3/z_1))\otimes \mrm{id}_3)\circ
			( \mrm{id}_3\otimes R_{b,g}(\beta,qz_3/z_2)).
		\end{align*}		
		\begin{align*}
			(R_{g,b}(\beta,z_3/(qz_2))\otimes \mrm{id}_{3})&\circ ( \mrm{id}_3\otimes R_{g,b}(\beta,z_3/(qz_1)))\circ 
			(R_{g,g}(\beta,z_2/z_1)\otimes  \mrm{id}_3)\\&=(\mrm{id}_3\otimes R_{g,g}(\beta,z_2/z_1))\circ
			(R_{g,b}(\beta,z_3/(qz_1))\otimes \mrm{id}_3)\circ
			( \mrm{id}_3\otimes R_{g,b}(\beta,z_3/(qz_2))).
		\end{align*}		
		\begin{align*}
			(R_{b,g}(\beta,qz_3/z_2)\otimes \mrm{id}_{3})&\circ ( \mrm{id}_3\otimes R_{g,g}(\beta,z_3/z_1))\circ 
			(R_{g,b}(\beta,z_2/(qz_1))\otimes  \mrm{id}_3)\\&=(\mrm{id}_3\otimes R_{g,b}(\beta,z_2/(qz_1)))\circ
			(R_{g,g}(\beta,z_3/z_1)\otimes \mrm{id}_3)\circ
			( \mrm{id}_3\otimes R_{b,g}(\beta,qz_3/z_2)).
		\end{align*}		
		\begin{align*}
			(R_{g,g}(\beta,z_3/z_2)\otimes \mrm{id}_{3})&\circ ( \mrm{id}_3\otimes R_{b,g}(\beta,qz_3/z_1))\circ 
			(R_{b,g}(\beta,qz_2/z_1)\otimes  \mrm{id}_3)\\&=(\mrm{id}_3\otimes R_{b,g}(\beta,qz_2/z_1))\circ
			(R_{b,g}(\beta,qz_3/z_1)\otimes \mrm{id}_3)\circ
			( \mrm{id}_3\otimes R_{g,g}(\beta,z_3/z_2)).
		\end{align*}
		
		\begin{align*}
			(R_{g,b}(\beta,z_3/(qz_2))\otimes \mrm{id}_{3})&\circ ( \mrm{id}_3\otimes R_{r,b}(\beta,qz_3/z_1))\circ 
			(R_{r,g}(\beta,q^2z_2/z_1)\otimes  \mrm{id}_3)\\&=(\mrm{id}_3\otimes R_{r,g}(\beta,q^2z_2/z_1))\circ
			(R_{r,b}(\beta,qz_3/z_1)\otimes \mrm{id}_3)\circ
			( \mrm{id}_3\otimes R_{g,b}(\beta,z_3/(qz_2))).
		\end{align*}
		\begin{align*}
			(R_{b,g}(\beta,qz_3/z_2)\otimes \mrm{id}_{3})&\circ ( \mrm{id}_3\otimes (R_{r,g}(\beta,q^2z_3/z_1))\circ 
			(R_{r,b}(\beta,qz_2/z_1)\otimes  \mrm{id}_3)\\&=(\mrm{id}_3\otimes R_{r,b}(\beta,qz_2/z_1))\circ
			(R_{r,g}(\beta,q^2z_3/z_1)\otimes \mrm{id}_3)\circ
			( \mrm{id}_3\otimes R_{b,g}(\beta,qz_3/z_2)).
		\end{align*}		
		\begin{align*}
			(R_{r,g}(\beta,q^2z_3/z_2)\otimes \mrm{id}_{3})&\circ ( \mrm{id}_3\otimes (R_{b,g}(\beta,qz_3/z_1))\circ 
			(R_{b,r}(\beta,z_2/(qz_1))\otimes  \mrm{id}_3)\\&=(\mrm{id}_3\otimes R_{b,r}(\beta,z_2/(qz_1))\circ
			(R_{b,g}(\beta,qz_3/z_1)\otimes \mrm{id}_3)\circ
			( \mrm{id}_3\otimes R_{r,g}(\beta,q^2z_3/z_2)).
		\end{align*}		
		\begin{align*}
			(R_{g,r}(\beta,z_3/(q^2z_2))\otimes \mrm{id}_{3})&\circ ( \mrm{id}_3\otimes (R_{b,r}(\beta,z_3/(qz_1)))\circ 
			(R_{b,g}(\beta,qz_2/z_1)\otimes  \mrm{id}_3)\\&=(\mrm{id}_3\otimes R_{b,g}(\beta,qz_2/z_1))\circ
			(R_{b,r}(\beta,z_3/(qz_1))\otimes \mrm{id}_3)\circ
			( \mrm{id}_3\otimes R_{g,r}(\beta,z_3/(q^2z_2))).
		\end{align*}		
		\begin{align*}
			(R_{r,b}(\beta,qz_3/z_2)\otimes \mrm{id}_{3})&\circ ( \mrm{id}_3\otimes (R_{g,b}(\beta,z_3/(qz_1)))\circ 
			(R_{g,r}(\beta,z_2/(q^2z_1))\otimes  \mrm{id}_3)\\&=(\mrm{id}_3\otimes R_{g,r}(\beta,z_2/(q^2z_1)))\circ
			(R_{g,b}(\beta,z_3/(qz_1)))\otimes \mrm{id}_3)\circ
			( \mrm{id}_3\otimes R_{r,b}(\beta,qz_3/z_2)).
		\end{align*}		
		\begin{align*}
			(R_{b,r}(\beta,z_3/(qz_2))\otimes \mrm{id}_{3})&\circ ( \mrm{id}_3\otimes (R_{g,r}(\beta,z_3/(q^2z_1)))\circ 
			(R_{g,b}(\beta,z_2/(qz_1))\otimes  \mrm{id}_3)\\&=(\mrm{id}_3\otimes R_{g,b}(\beta,z_2/(qz_1)))\circ
			(R_{g,r}(\beta,z_3/(q^2z_1)))\otimes \mrm{id}_3)\circ
			( \mrm{id}_3\otimes R_{b,r}(\beta,z_3/(qz_2))).
		\end{align*}

		\item \textit{Bootstrap equations:}
		\begin{align*}
			R_{b,r}(\beta,z_2/(qz_1))\circ (U(\beta)\otimes \mrm{id}_3)=(\mrm{id}_3\tensor U(\beta))\circ (R_{g,r}(\beta,z_2/(q^2z_1)\otimes \mrm{id}_3)\circ (\mrm{id}_3\otimes R_{r,r}(\beta,z_2/z_1)).
		\end{align*}
		\begin{align*}
			R_{b,g}(\beta,qz_2/z_1)\circ (U(\beta)\otimes \mrm{id}_3)=(\mrm{id}_3\tensor U(\beta))\circ (R_{g,g}(\beta,z_2/z_1)\otimes \mrm{id}_3)\circ (\mrm{id}_3\otimes R_{r,g}(\beta,q^2z_2/z_1)).
		\end{align*}
		\begin{align*}
			R_{b,b}(\beta,z_2/z_1)\circ (U(\beta)\otimes \mrm{id}_3)=(\mrm{id}_3\tensor U(\beta))\circ (R_{g,b}(\beta,z_2/(qz_1)\otimes \mrm{id}_3)\circ (\mrm{id}_3\otimes R_{r,b}(\beta,qz_2/z_1)).
		\end{align*}
		
		\begin{align*}
			R_{r,b}(\beta,qz_2/z_1)\circ (\mrm{id}_3\otimes U(\beta))=	(U(\beta)\tensor \mrm{id}_3)\circ (\mrm{id}_3\otimes R_{r,r}(\beta,z_2/z_1)\circ ( R_{r,g}(\beta,q^2z_2/z_1)\otimes \mrm{id}_3).
		\end{align*}
		\begin{align*}
			R_{g,b}(\beta,z_2/(qz_1))\circ (\mrm{id}_3\otimes U(\beta))=(U(\beta)\tensor \mrm{id}_3)\circ (\mrm{id}_3\otimes R_{g,r}(\beta,z_2/(q^2z_1))\circ ( R_{g,g}(\beta,z_2/z_1)\otimes \mrm{id}_3).
		\end{align*}
		\begin{align*}
			R_{b,b}(\beta,z_2/z_1)\circ (\mrm{id}_3\otimes U(\beta))=(U(\beta)\tensor \mrm{id}_3)\circ (\mrm{id}_3\otimes R_{b,r}(\beta,z_2/(qz_1))\circ ( R_{b,g}(\beta,qz_2/z_1)\otimes \mrm{id}_3).
		\end{align*}
		
		\begin{align*}
			(\mrm{id}_3\otimes D(\beta))\circ R_{b,r}(\beta,z_2/(qz_1))= (R_{r,r}(\beta,z_2/z_1)\otimes \mrm{id}_3)\circ ( \mrm{id}_3\otimes R_{g,r}(\beta,z_2/(q^2z_1)))\circ (D(\beta)\tensor \mrm{id}_3).
		\end{align*}
		\begin{align*}
			(\mrm{id}_3\otimes D(\beta))\circ R_{b,g}(\beta,qz_2/z_1)=(R_{r,g}(\beta,q^2z_2/z_1)\otimes \mrm{id}_3)\circ ( \mrm{id}_3\otimes R_{g,g}(\beta,z_2/z_1))\circ (D(\beta)\tensor \mrm{id}_3).
		\end{align*}
		\begin{align*}
			(\mrm{id}_3\otimes D(\beta))\circ R_{b,b}(\beta,z_2/z_1)=(R_{r,b}(\beta,qz_2/z_1)\otimes \mrm{id}_3)\circ ( \mrm{id}_3\otimes R_{g,b}(\beta,z_2/(qz_1)))\circ (D(\beta)\tensor \mrm{id}_3).
		\end{align*}
		
		\begin{align*}
			(D(\beta)\otimes \mrm{id}_3)\circ R_{r,b}(\beta,qz_2/z_1)=( \mrm{id}_3\otimes R_{r,g}(\beta,q^2z_2/z_1))\circ ( R_{r,r}(\beta,z_2/z_1)\otimes \mrm{id}_3)\circ ( \mrm{id}_3\otimes D(\beta)).
		\end{align*}
		\begin{align*}
			(D(\beta)\otimes \mrm{id}_3)\circ R_{g,b}(\beta,z_2/(qz_1))=
			( \mrm{id}_3\otimes R_{g,g}(\beta,z_2/z_1))\circ ( R_{g,r}(\beta,z_2/(q^2z_1))\otimes \mrm{id}_3)\circ ( \mrm{id}_3\otimes D(\beta)).
		\end{align*}
		\begin{align*}
			(D(\beta)\otimes \mrm{id}_3)\circ R_{b,b}(\beta,z_2/z_1)=( \mrm{id}_3\otimes R_{b,g}(\beta,qz_2/z_1))\circ ( R_{b,r}(\beta,z_2/(qz_1))\otimes \mrm{id}_3)\circ ( \mrm{id}_3\otimes D(\beta)).
		\end{align*}
		
		\item \textit{Unitarity equations:}
		
		\begin{align*}
			R_{r,r}(\beta,z_2/z_1)\circ R_{r,r}(\beta,z_1/z_2)=\mrm{id}_3\otimes \mrm{id}_3.
		\end{align*}
		\begin{align*}
			R_{g,g}(\beta,z_2/z_1)\circ R_{g,g}(\beta,z_1/z_2)=\mrm{id}_3\otimes \mrm{id}_3.
		\end{align*}
		\begin{align*}
			R_{b,b}(\beta,z_2/z_1)\circ R_{b,b}(\beta,z_1/z_2)=\mrm{id}_3\otimes \mrm{id}_3.
		\end{align*}
		\begin{align*}
			R_{r,g}(\beta,q^2z_2/z_1)\circ R_{g,r}(\beta,z_1/(q^2z_2))=\mrm{id}_3\otimes \mrm{id}_3.
		\end{align*}
		\begin{align*}
			R_{b,g}(\beta,qz_2/z_1)\circ R_{g,b}(\beta,z_1/(qz_2))=\mrm{id}_3\otimes \mrm{id}_3.
		\end{align*}
		\begin{align*}
			R_{r,b}(\beta,qz_2/z_1)\circ R_{b,r}(\beta,z_1/(qz_2))=\mrm{id}_3\otimes \mrm{id}_3.
		\end{align*}
		\item \textit{Value at equal spectral parameters:}
		\begin{align*}
			R_{g,g}(\beta,1)=\mrm{id}_3\otimes \mrm{id}_3.
		\end{align*}
		\begin{align*}
			R_{r,r}(\beta,1)=\mrm{id}_3\otimes \mrm{id}_3.
		\end{align*}
		\begin{align*}
			R_{b,b}(\beta,1)=\mrm{id}_3\otimes \mrm{id}_3.
		\end{align*}
	\end{itemize}
	
	\section{The quantum group ${U}_q(\widehat{a}_n)$}\label{section:quantum-group}
	We will recall the presentation for ${U}_q(\widehat{\mathfrak{a}}_n)$ given in \cite[Appendix A.1]{KZJ21}
	
	\subsection{Generators and relations} 
	let $C=(C_{i,j})_{i,j=0,\dotsc,n}$ be the generalized Cartan matrix for the affine Lie algebra of type $\widehat{\mathfrak{a}}_n$:
	\[C_{i,j}=\begin{cases}
		2,\quad &\text{ $i=j$;}\\
		-1,\quad &\text{ $i= j+1$, or $i=j-1$, or $(i,j)\in\{(0,n),(n,0)\}$;}\\
		0,\quad&\text{ otherwise.}
	\end{cases}\]
	The quantum affine algebra ${U}_q(\widehat{\mathfrak{a}}_n)$ is the $\mbb{C}(q)$-algebra with generators $E_i, F_i,K_i$, where $i=0,1,\dotsc,n$, subject only to the following relations: 
	\[K_iK_j=K_jK_i,\]
	\[K_iE_jK_i^{-1}=C_{i,j}E_j,\quad K_iF_jK_i^{-1}=C_{i,j}^{-1}F_j,\quad [E_i,F_j]=\delta_{i,j}\frac{K_i-K_i^{-1}}{q-q^{-1}},\]
	\[E_i^2E_j-(q+q^{-1})E_iE_jE_i+E_jE_i^2=0,\quad i\neq j,\]
	\[F_i^2F_j-(q+q^{-1})F_iF_jF_i+F_jF_i^2=0,\quad i\neq j.\]
	The algebra ${U}_q(\widehat{\mathfrak{a}}_n)$ has the structure of a Hopf algebra, and the coproduct $\Delta$, counit $\epsilon$, and antipode $S$ of this Hopf algebra are given by the formulas below:
	\[\Delta(E_i)=K_i\otimes E_i+E_i\otimes 1,\quad \Delta(F_i)=1\otimes F_i+F_i\otimes K_i^{-1},\quad \Delta(K_i)=K_i\otimes K_i,\]
	\[\epsilon(E_i)=0,\quad \epsilon(F_i)=0,\quad \epsilon(K_i)=1,\]
	\[S(E_i)=-K_i^{-1}E_i,\quad S(F_i)=-F_iK_i,\quad S(K_i)=K_i^{-1}.\]

	
	\bibliographystyle{alphaurl}
	\bibliography{biblist}

\end{document}